\documentclass[11pt]{amsart}

\usepackage{amsfonts,amsmath,amssymb,amsthm,mathrsfs,url,xcolor,latexsym,rawfonts,graphicx}
\usepackage{hyperref}
\usepackage{enumerate} 
\usepackage[]{caption2}
\usepackage{accents}

\theoremstyle{plain}
  \newtheorem{theorem}{Theorem}[section]
  \newtheorem{lemma}{Lemma}[section]
  \newtheorem{proposition}{Proposition}[section]
  
  \newtheorem{corollary}{Corollary}[section]
  \newtheorem{definition}{Definition}[section]
  \newtheorem{remark}{Remark}[section]

\newcommand{\U}{\overline{U}}

   \newcommand{\beqn}{\begin{eqnarray}}
   \newcommand{\eeqn}{\end{eqnarray}}
   \newcommand{\beqs}{\begin{eqnarray*}}
   \newcommand{\eeqs}{\end{eqnarray*}}
   \newcommand{\ban}{\begin{eqnarray*}}
   \newcommand{\nan}{\end{eqnarray*}}
   \newcommand{\beq}{\begin{equation}}
   \newcommand{\eeq}{\end{equation}}
   
   \newcommand{\p}{\partial}
\newcommand{\eps}{\varepsilon}
\newcommand{\Lm}{\Omega^*}
\newcommand{\Om}{\Omega}
\newcommand{\pom}{{\p\Om}}
\newcommand{\bom}{{\overline\Om}}

\renewcommand{\det}{\mbox{det}}

\newcommand{\F}{\mathcal F}

\newcommand{\R}{\mathbb{R}}

\numberwithin{equation}{section}

\newcommand{\ut}[1]{\underaccent{\tilde}{#1}}
\renewcommand{\vec}[1]{\ut{#1}}

  \numberwithin{equation}{section}
  \numberwithin{figure}{section}

\begin{document}

\title[$C^{2,\alpha}$ regularity of free boundaries in optimal transportation]
{\textbf{$C^{2,\alpha}$ regularity of free boundaries \\ in optimal transportation}}

\author{Shibing Chen}
\address{Shibing Chen, School of Mathematical Sciences,
University of Science and Technology of China,
Hefei, 230026, P.R. China.}
\email{chenshib@ustc.edu.cn}

\author[J. Liu]
{Jiakun Liu}
\address
	{Jiakun Liu, School of Mathematics and Applied Statistics,
	University of Wollongong,
	Wollongong, NSW 2522, AUSTRALIA}
\email{jiakunl@uow.edu.au}

\author[X.-J. Wang]
{Xu-Jia Wang}
\address
{Xu-Jia Wang, Centre for Mathematics and Its Applications,
The Australian National University,
Canberra, ACT 0200, AUSTRALIA}
\email{Xu-Jia.Wang@anu.edu.au}

\subjclass[2000]{35J96, 35J25, 35B65.}

\keywords{Optimal transportation, Monge-Amp\`ere equation, free boundary}

\thanks{Research of Chen was supported by National Key R\&D program of China 2022YFA1005400, 2020YFA0713100, National Science Fund for Distinguished Young Scholars (No. 12225111), and NSFC No. 12141105.
Research of Liu and Wang was supported by ARC DP200101084 and DP230100499.  Research of Liu was supported by FT220100368.
}

%\subjclass[2000]{Primary 53C44, 35K55}

\date{\today}

\dedicatory{}

\begin{abstract}  
The regularity of the free boundary in optimal transportation is equivalent to that of the potential function along the free boundary.
By  establishing new geometric estimates of  the free boundary and studying the second boundary value problem of the Monge-Amp\`ere equation, 
we obtain the  $C^{2,\alpha}$ regularity of the potential function as well as that of the free boundary,
thereby resolve an open problem raised by Caffarelli and McCann in \cite{CM}.
\end{abstract}

\maketitle

\baselineskip=16.4pt
\parskip=3pt

\section{Introduction}
Let $\Omega$ and $\Omega^*$ be two disjoint, bounded, convex domains in the Euclidean space $\R^n$.
Let $f$ and $g$ be the densities in $\Omega$ and $\Omega^*$,  respectively. 
Let $m$ be a positive constant satisfying 
\begin{equation}\label{mass1}
m\leq \min\Big\{\int_\Omega f, \,\, \int_{\Omega^*} g\Big\}.
\end{equation}
A non-negative, finite Borel measure $\gamma$ on $\mathbb{R}^n\times \mathbb{R}^n$ is called 
a transport plan (with mass $m$) from the distribution $(\Omega, f)$ to the distribution $(\Omega^*, g)$,
if $\gamma(\R^n\times\R^n)  =m$ and 
\beq\label{tp}
%{\begin{split}
       \gamma(A\times \mathbb{R}^n) \leq \int_{A\cap\Om} f(x)\,dx, \ \ \ 
       \gamma(\mathbb{R}^n \times A) \leq \int_{A\cap\Om^*} g(y)\,dy
%       \end{split}}
\eeq
for any Borel set $A\subset\R^n$.
A transport plan $\gamma$ is {\it optimal} if it minimises the cost functional 
\begin{equation} \label{mini}
 \int_{\mathbb{R}^n \times \mathbb{R}^n} |x-y|^2\, d\gamma(x,y) 
\end{equation}    
among all transport plans.
 
In the pioneering work \cite{CM},  Caffarelli and McCann proposed to study the above optimal partial transport problem. 
The word ``partial" means that under the condition \eqref{mass1}, 
not all of the mass in $\Omega$ is transported to $\Omega^*$.  
The existence and uniqueness of the optimal transport plan have been proved in  \cite{CM}.
%{\Small\color{blue} (for any given $m$, the optimal transport plan and the optimal mapping are 1-1?)}
Let $U\subset \Omega$ be the sub-domain in which the mass $m=\int_U f$ 
is transported to $V\subset\Om^*$ by the optimal transport plan.
%%% {\Small\color{blue} (A transport plan is a measure $\gamma$.  Can we say a ``measure'' transport a mass from $U$ to $V$?)}
The sets $\mathcal F=:\partial U\cap \Omega$ and $\mathcal F^*=:\partial V\cap\Omega^*$ are called {\it free boundaries}
of the problem.

When $\Omega, \Omega^*$ are strictly convex and separate (i.e. their closures are disjoint), and $f, g$ are positive and bounded, 
Caffarelli and McCann \cite{CM} proved that the free boundaries $\mathcal F$ and $\mathcal F^*$ are $C^{1,\alpha'}$ smooth
for some $\alpha' \in(0,1)$. 
%{\Small\color{blue} (The previous version uses $\alpha$ many times for $u, \F\in C^{1,\alpha}$ and $\beta$ a few times. I think it is better to use $\alpha$ here. There is no confusion to use $\alpha$ throughout the paper.)}
If $\Om$ and $\Om^*$ partly overlap, namely if $\Omega\cap\Omega^*\ne\emptyset$,
Figalli \cite{AFi2, AFi}  proved that $\mathcal F$ and $\mathcal F^*$ are locally $C^1$ smooth 
away from the common region $\Omega\cap\Omega^*$.
Later, Indrei \cite{I}  improved the $C^1$ regularity to $C^{1,\alpha'}$, also away from $\Omega\cap\Omega^*$.
Related problems were also studied by Kitagawa-McCann \cite{KM2} and Kitagawa-Pass \cite{KM1}.

An open problem raised in \cite{CM}  is the higher regularity of free boundaries.
In this paper we resolve the problem completely.

 \begin{theorem}\label{main1}
Let $\Omega, \Omega^* \subset \R^n$ be two separate, uniformly convex domains with $C^2$ boundaries.
Assume that $f\in C^{\alpha}(\overline{\Omega})$ and $g\in C^{\alpha}(\overline{\Omega^*})$
are positive densities for some $\alpha\in(0,1)$, and $m$ is a positive constant satisfying \eqref{mass1}.
Then the free boundaries $\mathcal{F}$ and $\mathcal{F}^*$ are $C^{2,\alpha}$ smooth. 
 If furthermore, $f, g\in C^{\infty}$ and $\partial\Omega, \partial\Omega^*\in C^{\infty}$,  
then $\mathcal{F}, \mathcal{F}^*$ are $C^{\infty}$ smooth.
 \end{theorem}

We remark that the above theorem also holds for the more general case when two convex domains have overlap as considered by Figalli \cite{AFi2, AFi} and Indrei \cite{I}. In particular, the main result holds for the part of free boundary away from the closure of the common region.

Recall that for the complete transport problem,  
namely when $m=\|f\|_{L^1(\Om)}=\|g\|_{L^1(\Om^*)}$  and $U=\Om$, $V=\Om^*$, 
the optimal transport plan is characterised by a convex potential function $u$ in $\Omega$, 
which satisfies the Monge-Amp\`ere equation 
\beq\label{MAEu}
\det\, D^2 u = \frac{f}{g\circ Du}\quad \ \text{ in }\  \Om 
\eeq
subject to the natural boundary condition 
\beq \label{bdy-u}
Du(\Om)= \Om^*.
\eeq 
Caffarelli proved that $u\in C^{1,\alpha'}(\overline\Om)$
if $\Omega, \Omega^*$ are bounded and convex, and $f, g$ are positive and bounded \cite{C92b}.
He also proved that $u\in C^{2,\alpha}(\overline\Om)$ if
$\Omega, \Omega^*$ are uniformly convex and $C^2$ smooth, and $f, g\in C^\alpha$ \cite{C96}.
If $f, g$ are smooth, the global $C^{2,\alpha}$ regularity was first obtained by Delano\"e \cite{D91} in dimension two,
and later by Urbas \cite{U1} for higher dimensions.
In a recent paper  \cite{CLW1}, the authors relaxed the uniform convexity and $C^2$ regularity 
of the boundaries $\pom, \pom^*$ in  \cite{C96}.   
In dimension two, the regularity assumption on the boundaries can be further relaxed  \cite{CLW2, SY}.

For the partial transport problem, let $u$ be the potential function of the optimal transport map from the active region $U$ to $V$. 
Then $u$ satisfies the boundary value problem \eqref{MAEu} and \eqref{bdy-u}
with the domains $\Om$ and $\Om^*$ replaced by $U$ and $V$, respectively.
By relation \eqref{normalformular} in Section \ref{S2},  
the regularity of $\mathcal{F}$ follows from that of $u$ at the free boundary $\mathcal{F}$.
Therefore, to prove the free boundary $\mathcal{F}\in C^{2,\alpha}$, 
we aim to establish the $C^{2,\alpha}$ regularity of $u$ up to the free boundary $\mathcal{F}$.
If the $C^{2,\alpha}$ regularity of $u$ is established, 
higher regularity then follows from the standard elliptic theory \cite{GT}, see Remark \ref{rmkhr}.

 %{\color{blue} Let $x_0\in \mathcal{F},$ it is known that $y_0=Du(x_0)\in \partial V\cap \partial\Omega^*.$
% Suppose $\nu_U(x_0), \nu_V(y_0)$ are the unit inner normals of $U, V$ at $x_0, y_0$ respectively. }

Recall that to obtain the $C^{2,\alpha}$ regularity for the problem \eqref{MAEu} and \eqref{bdy-u} in \cite{C96, CLW1}, 
one first proves the uniform density and the tangential $C^{1, 1-\eps}$ regularity for $u$ and its dual
function $v$, and then uses them to establish the uniform obliqueness. 
But in our current case, the free boundary $\mathcal F$, as part of the boundary $\p U$,  
is not convex in general, nor is it known to be $C^{1,1}$ smooth in advance.
The convexity and  the $C^{1, 1}$ regularity of the domains are crucial 
in  \cite{C96, CLW1}, and in \cite{D91, U1} as well,
and are used throughout the proofs in these papers.
%%% We cannot prove the uniform density and the the tangential $C^{1, 1-\eps}$ regularity for $u$
%%% before the uniform obliqueness is established.
Therefore to prove the regularity of the free boundary,  we cannot follow the route in \cite{C96,CLW1}.
Innovative observations and ideas are needed.  One of the main new ingredients we introduced is that a delicate application of the interior ball property to the carefully chosen points can give us some unexpected geometric estimates of the free boundary and control the shape of  the centred sub-level sets $S_h^c[v]$  (see Lemma \ref{keylemma}, \ref{hkey1}, \ref{hkey2} and Corollary \ref{coroabove}).

%Since the breakthrough by Caffarelli \cite{Ca77}, 
%the free boundary problem has been extensively studied in the last few decades.
%A well-known approach  \cite{Ca98} is to study the blow-up profile of the free boundary, 
%and to prove the regularity of the free boundary according to the classification of the limit profiles.
% In our current case, we are unable to prove that a blow-up sequence at the free boundary converges 
%to a limit with simple geometry.
%such as quadratic functions.
% By blowing-up,
%we mean to normalise a sequence of sub-level sets of the solution, which is different from that in the classical free boundary problem.
%{\color{purple} (Instead, we shall emphasise the difficulty in our current case as pointed by Caffarelli-McCann.)}

The argument in this paper is built upon a careful local geometric analysis in \S3 and a blow-up analysis in \S5,
for the potential functions $u$ and its dual $v$. 
The whole proof can be roughly divided into two parts. 
In the first part (\S3 and \S4),  we assume a uniform obliqueness condition, 
such that the problem \eqref{MAEu} and \eqref{bdy-u} (with $\Omega, \Omega^*$ replaced by
$U, V$ respectively) locally becomes a uniformly oblique derivative problem of the Monge-Amp\`ere equation.
%{\color{purple} (our problem is $Du(U)=V$ but not $Du(\Om)=\Om^*$. Even though, it is just a locally uniformly oblique problem but not globally.)}
We remark that generally there is no a priori $C^{1,1}$ estimate for the Monge-Amp\`ere equation subject to the oblique condition
$\p_\beta u =\psi$ on $\pom$ even if the domain $\Om$ is uniformly convex and smooth, 
and the vector $\beta$ is smooth \cite {U2}, see Remark \ref{r3.1}.
In this paper we establish the a priori $C^{2,\alpha}$ estimate for the solution,
using various local estimates on the potential functions $u, v$ and the free boundary $\F$ in \cite{C96, CM, CLW1}

In the second part (\S5 and \S6), we verify the assumption of the uniform obliqueness condition.
 Assume by contradiction that the uniform obliqueness condition fails.  
% {\color{blue}In the complete transport case, one can use duality and the tangential $C^{1,1-\epsilon}$ regularity of $u$ to control the shape of $v.$ However, in our current case, due to the non-convexity of $\mathcal{F},$ we do not have such kind of regularity of $u$ a priori. To overcome this difficulty, we
%{\color{blue} In this case we perform a careful blow-up analysis.
%Even though we cannot prove the convergence of a blow-up sequence to a nice limit profile,
%but we obtain some helpful properties, }
In this case, by utilising the interior ball property (Lemma \ref{intball}), we can  give a precise characterization of the shape of  the centred sub-level sets $S_h^c[v]$, which is a crucial ingredient of performing a blow-up analysis. Then in the limit profile, we have the following helpful properties,
such as 1): the blow-up limit of the free boundary is convex;
2):  the blow-up limit of the free boundary can be decomposed 
as a product $\R^{n-2}\otimes \gamma$ for a convex curve $\gamma$.
With these properties, and using some techniques from \cite{C96, CLW1}
we derive a contradiction. Hence the uniform obliqueness condition is satisfied.

%Our idea is that if  the obliqueness fails at some point $x_0\in\mathcal{F}$,
%we want to prove that the blow-up limit of the free boundary at $x_0$ is convex 
%{and satisfies some nice geometric behaviour,}
%from which we then derive a contradiction.
%Therefore a main ingredient of the paper is the analysis for the blow-up process.
%In the argument for the uniform obliqueness, we will study the dual potential function $v$, instead of $u$.
%{\color{blue} Instead of using the estimates for $u$ and the duality as in \cite{CLW1}, 
%here we apply the interior ball property in a novel way to control the shape of the level set of $v.$}  Once the uniform obliqueness is obtained, 
%we then prove in turn the uniform density and the tangential $C^{1, 1-\eps}$ regularity for $u$,
%{by different techniques}.
%%% which are two ingredients in the arguments in  \cite{C96,CLW1}.
%When all these properties are proved, we obtain the boundary $C^{2,\alpha}$ regularity as in \cite{CLW1}.

% {\Small\color{blue} (I feel inappropriate to emphasize that we cannot establish the
%the uniform density and the the tangential $C^{1, 1-\eps}$ regularity for $u$
%before the uniform obliqueness. 
%Actually we will be working with $v$ and we do obtain the uniform density 
%and the tangential $C^{1, 1-\eps}$ regularity for $v$ before the uniform obliqueness.
%It is a question how to describe the difference and technical difficulties of this paper?)}

This paper is organised as follows. 
In \S\ref{S2} we recall some results from \cite{C96, CM, CLW1} which will be used in subsequent sections.
In \S\ref{S5} we prove the $C^{1,1-\epsilon}$ regularity of the free boundary $\mathcal{F}$ for any given small $\epsilon\in(0,1),$
assuming the uniform obliqueness condition.
In \S\ref{S6}, we  raise the $C^{1,1-\epsilon}$ regularity to $C^{2,\alpha}$ by a perturbation method
and thus prove Theorem \ref{main1}. 
\S\ref{Sbu} deals with the blow-up analysis at the free boundary where the obliqueness fails,
which leads to a contradiction in \S\ref{Sob} 
and thus confirming the obliqueness property.

\section{Preliminaries}\label{S2}

\subsection{Potential functions}\label{s21}
Throughout the paper, we always assume that the densities $f, g$ satisfy
\begin{equation}\label{bdfg}
\lambda^{-1}<f, g<\lambda
\end{equation}
 in $\Omega, \Omega^*$, respectively, for a positive constant $\lambda$,  
and  $\overline{\Omega}, \overline{\Omega^*}$ are disjoint and uniformly convex.
For a fixed constant $m$ satisfying \eqref{mass1}, 
it is shown in \cite{CM} that the optimal transport plan $\gamma$, 
namely the minimiser of \eqref{mini}, is characterised by 
\beq\label{optmeas}
\gamma =(\text{Id} \times T)_{\#}f_m=(T^{-1} \times \text{Id})_{\#} g_m, 
\eeq
where $f_m=f \chi_{_U}$, $g_m=g\chi_{_V}$, 
and $T$ is the optimal transport map from the active domain $U\subset \Omega$ 
to the active target $V\subset \Lm$.
The notation $T_{\#}\mu$ denotes the pushforward of measure $\mu$ by the mapping $T$  \cite{V1,V2}. 
Moreover, there exist convex potentials $u, v$ on $\R^n$ such that
\beq\label{opm}
\begin{split}
   T(x) & =Du(x)\quad \forall\, x\in U, \\
   T^{-1}(y) & = Dv(y)\quad \, \forall\ y\in V,
    \end{split}
\eeq 
and 
 \begin{equation} \label{push}
 \begin{split} 
  & (Du)_{\#}(f_m+(g-g_m))= g, \\
  & (Dv)_{\#}((f-f_m)+g_m) = f .
  \end{split}
 \end{equation}
The convex functions $u,v$ also satisfy 
\beq\label{globlip}
	Du(\R^n)=\overline{\Om^*},\ \ \ Dv(\R^n)=\bom , 
\eeq 
and can be expressed by 
\beq\label{newv}
 \begin{split} 
 & u(x) = \sup\{L(x) \,:\, L\ \text{is affine}, L\leq u\ \text{in}\  (\Omega^*\setminus \overline{V})\cup U, \ \text{and} \ DL\in \Omega^*\}, \\
 & v(y) = \sup\{L(y) \,:\, L\ \text{is affine}, L\leq v\ \text{in}\  (\Om\setminus \U)\cup V, \ \text{and} \  DL\in \Omega\}.
  \end{split}
\eeq

Let 
 \begin{align*} 
  & u^*(y):= {\sup}_{x\in \mathbb{R}^n}  \left\{  y\cdot x-u(x) \right\} \  \ {\text{for}\ y\in \overline{\Omega^*} }, \\
 & v^*(x):= {\sup}_{y\in \mathbb{R}^n} \left\{ x\cdot y-v(y)  \right\} \  \ {\text{for}\ x\in \overline{\Omega} }
 \end{align*} 
be the standard Legendre transforms of $u, v$, respectively.
The following properties are proved in \cite{CM}:
\begin{itemize}
\item  [$(i)$]  $u=v^*$ in $U$; and  $v=u^*$ in $V$. \label{relu*}

\item [$(ii)$]   $Du(x)=x$ for $x\in \Omega^*\setminus \overline V$  and 
$Dv(y)=y$ for $y\in \Omega\setminus \overline U$.
Hence  
\begin{align*}
 u(x) &=\frac12 |x|^2+C\ \text{in each connected component of $\Omega^*\setminus \overline{V}$,} \\
 v(y)& =\frac12 |y|^2+C\ \text{in each connected component of $\Omega\setminus \overline{U}$.}
 \end{align*} 
\item [$(iii)$]   $u^*$ (resp. $v^*$) is strictly convex in $\Omega^*$ (resp. $\Omega$).
\end{itemize}

\begin{remark} \label{uvu*}
Note that $u^*$ and $v$ are two different functions. 
$u^*$ is the Legendre transform of $u$, it is defined in $\overline{\Omega^*}$.
But $v$ is defined in $\R^n$, and $v$ is strictly convex in and only in $V\cup (\Om\setminus \U)$.
By property ($i$) we have $v=u^*$ in $V$. 
There are similar relations between $u$ and $v^*$. 
%{\Small\color{blue} (What is the behaviour of $u$ in $\Om-U$?
%it seems later we also need the behaviour of $u$ in $\Om-U$ in \S5.
%Also in \S5, we regard $u$ as the Legendre transform of $v$ in \S5.)}
%{\Small\color{red}(In \S5 we have corrected the corresponding part, there we should regard $u$ as the Legendre transform of $u^*.$ \S5 now becomes \S3.)}
\end{remark}

By \eqref{push} and Property $(ii)$, $u$ satisfies the Monge-Amp\`ere equation 
\begin{align}\label{MAuU}
\det\, D^2 u &= \frac{f}{g\circ Du}\quad \ \text{ in }\  U, \\
Du(U) &= V. \nonumber
\end{align}
and the dual function $v$ satisfies  
\begin{align}\label{MAEv} 
\det\, D^2 v  &= \frac{g}{f\circ Dv}\quad \ \text{ in }\  V, \\
Dv(V) &= U. \nonumber
\end{align}
Furthermore, by \eqref{newv} and since  $\Omega,\Omega^*$ are bounded, 
$u$ and $v$ are globally Lipschitz in $\R^n$.
By \eqref{push}, $u$ and $v$ satisfy respectively
\begin{align}\label{Asolv} 
	C^{-1}(\chi_{_{\Omega^*\setminus V}} +\chi_{_U}) 
	      & \leq \det\, D^2 u \leq C (\chi_{_{\Omega^*\setminus V}}+\chi_{_U}), \nonumber \\
	C^{-1}(\chi_{_{\Om\setminus U}}+\chi_{_V}) 
	      &\leq \det\, D^2 v \leq C (\chi_{_{\Om\setminus U}}+\chi_{_V})
  \end{align}
in the sense of Alexandrov \cite{C92}, where $C$ is a positive constant depending only on $\lambda$.

For a convex function $w : \R^n\to(-\infty, \infty]$,
the associated \emph{Monge-Amp\`ere measure} $\mu_w$ is defined by
\beq\label{MAmeas}
\mu_w(E) := \left| \p w(E) \right| 
\eeq 
for any measurable set  $ E\subset\R^n$,
where $\p w$ is the sub-gradient of $w$ and $|\cdot|$ denotes the $n$-dimensional Hausdorff measure.
If $w$ is $C^2$ smooth, then
$$ \mu_w(E) = \int_E \det\, D^2w(x) \,dx. $$ 
We say that $w$ satisfies $C_1 \chi_{_W}\leq \det\, D^2 w \leq C_2\chi_{_W}$ in the sense of Alexandrov,
if
$$ C_1|E\cap W|\leq \mu_w(E) \leq C_2|E\cap W| \ \ \ \forall \ E\subset\R^n. $$
Hence \eqref{Asolv} implies that the Monge-Amp\`ere measure 
$\mu_v$ (resp. $\mu_u$) is actually supported and bounded on $(\bom\setminus U)\cup \overline V$ 
(resp. $(\overline{\Omega^*}\setminus V)\cup\overline U$).

\subsection{$C^{1,\alpha'}$ regularity of $\mathcal{F}$}
We recall the \emph{interior ball condition} proved in \cite{CM}, 
which will be useful in our subsequent analysis. 

\begin{lemma}[{\cite[Corollary 2.4]{CM}}]\label{intball}
Let $x\in U$ and $y=Du(x)$, then 
	\[ \Om\cap B_{|x-y|}(y) \subset U. \]
Likewise, let $y\in V$ and $x=Dv(y)$, then
	\[ \Lm\cap B_{|x-y|}(x) \subset V. \]
\end{lemma}

By Lemma \ref{intball}, it is shown in \cite{CM} that $u$ is $C^1$ smooth up to the free boundary $\mathcal F$, 
and the unit inner normal vector of $\mathcal F$ is given by
 \begin{equation}\label{normalformular}
 \nu(x)=\frac{Du(x)-x}{|Du(x)-x|}\quad \forall\ x\in \mathcal F.
 \end{equation}
Hence, the regularity of $u$ up to the free boundary $\mathcal F$ implies the regularity of the free boundary $\mathcal F$ itself.
The following regularity results have been obtained in \cite{CM}.

\begin{theorem}[\cite{CM}] \label{CMCL}
Assume that $\Omega, \Omega^*$ are disjoint and strictly convex,
 the densities $f, g$ satisfy
$\lambda^{-1}<f, g<\lambda$ for a positive constant $\lambda$. 
Then
\begin{itemize}

\item [$i)$]  $u, v\in C^1(\mathbb{R}^n)$,  $Dv$  is 1-1 from $\overline{V}$ to $\U$, and $Du$  is 1-1 from $\overline{U}$ to $\overline V$. 
 
\item [$ii)$]  $u\in C^{1,\alpha'}$ up to the free boundary $\mathcal{F}$, and thus $\mathcal{F}$ is $C^{1,\alpha'}$ for some $\alpha\in(0,1)$.

\item [$iii)$]  $\forall\,x_0\in \mathcal{F}$, $\exists$ a neighborhood $\mathcal{N}$ of $x_0$ such that $v$ is strictly convex in $Du(\mathcal{N}\cap\overline{U})$.

\item [$iv)$]   Let $y_0=Du(x_0)$. Then $y_0\in \partial V\setminus \overline{\partial V\cap \Omega^*}\subset \partial\Omega^*.$
Moreover, there exists a constant $r$ depending on $\text{dist}(x_0, \partial \Omega),$ such that $B_r(y_0)\cap \Omega^*\subset V.$

\end{itemize}
\end{theorem}

\subsection{Sub-level sets} \label{s23}
To study higher order regularity of the potentials $u, v$,
we introduce the (centred)  sub-level sets as in  \cite{C92b,C96}.
Note that from $iii)$ and $ iv)$ of Theorem \ref{CMCL}, the function $v$ is locally strictly convex 
near $Du(\mathcal{F}) \subset \partial V\setminus \overline{\partial V\cap \Omega^*}$, 
which (as a portion of $\partial\Omega^*$) is convex as well.

\begin{definition}\label{defS}
Let $y_0\in \overline V$ and $h>0$ be a small constant. We denote by
\beq\label{sect}
	S^c_{h}[v](y_0) := \left\{y\in\R^n \,:\, v(y)< v(y_0) + (y-y_0)\cdot \bar{p} + h\right\}
\eeq
the centred sub-level set of $v$ with height $h$, where $\bar{p}\in \R^n$ is chosen 
such that the centre of mass of $S^c_{h}[v](y_0)$ is $y_0$.
We denote by
\beq\label{sub}
S_h[v](y_0) : =\left\{y\in V \,:\, v(y) < \ell_{y_0}(y) + h\right\}
\eeq
the sub-level set of $v$ with height $h$, where $\ell_{y_0}$ is a support function of $v$ at $y_0$. 
\end{definition}

Note that in the above definition, 
$S_h[v](y_0)$ is a subset of $V$ but $S^c_{h}[v](y_0)$ may not be contained in $\Om^*$.
In the following we will write $S_h[v](y_0)$ and $S^c_{h}[v](y_0)$ as 
$S_h[v]$  and $S^c_{h}[v]$ when no confusion arises.

\begin{remark} \label{uniest11}
%\emph{
Suppose $v(0)=0, v\geq 0.$
Let $L$ be the affine function such that $S^c_h[v](0)=\{v<L\}$.  
Since $(L-v)(0)=h$,  $L=v$ on $\partial S^c_h[v](0),$ $L\geq v\geq  0$ in $S^c_h[v](0),$ 
and $S^c_h[v](0)$ is balanced around $0$,  we have that 
 \begin{equation}\label{secrela}
v\leq L\leq Ch\ \ \ \text{in}\  S^c_h[v](0)
\end{equation}
 for a constant $C$ depending only on $n.$ 
  %{\color{red} Indeed, given any $e\in \mathbb{S}^{n-1},$ the line $\{te: t\in\mathbb{R}\}$ intersects $\partial S^c_h[v](0)$ at two points $p, q.$ 
 %By John's Lemma \cite[Lemma 2.1]{C96},
 %(see \cite{LW} for a simple proof), 
 %there is an ellipsoid $E$ centred at $0$ such that
 %\begin{equation}\label{ima3}
 %E\subset S^c_h[v]\subset C(n) E.
 %\end{equation}
 %Hence, $S^c_h[v](0)$ is balanced around $0,$ which implies that $1/C<\frac{|p|}{|q|}<C$ for some positive constant $C$ depending only on $n.$ Hence $0=cp+(1-c)q$ for some constant $0<c<1$ depending only on $n.$ Now, since $L$ is an affine function, we have  $h=L(0)=cL(p)+(1-c)L(q),$ which implies the desired estimate 
 %\eqref{secrela}.
 %}
 Indeed, assume that $L(te)=\sup_{S^c_h[v](0)}L$ at $te\in\partial S^c_h[v](0)$ for some $e\in \mathbb{S}^{n-1}$ and $t>0$.
 Let $-t'e\in\partial S^c_h[v](0)$ for some $t'>0$ be the boundary point along the opposite direction $-e$. 
 By its definition, the centre of mass of the convex set $S^c_h[v](0)$ is $0$, hence $t'\approx t$, namely $C^{-1}<t'/t<C$ for some constant $C$ depending only on $n$. 
 Since $L$ is an affine function, we have
 $$ h = L(0) = \frac{t}{t+t'}L(-t'e) + \frac{t'}{t+t'}L(te) \geq \frac{t'}{t+t'}L(te). $$
 Therefore, $L(te)\leq Ch$. 
The same property also holds if $v$ is replaced by $u.$ 
\end{remark}

%Let $L$ be the affine function such that $S^c_h[v](y_0)=\{v<L\}$.   Since $(L-v)(y_0)=h$,  $L=v$ on $\partial S^c_h[u](y_0)$, $L\geq v$ in $S^c_h[v](y_0)$, and $S^c_h[v](y_0)$ is balanced around $y_0$,  we have 
% \begin{equation}\label{secrela}
%0\leq L-v \leq Ch\quad\text{ in }S^c_h[v](y_0),
%\end{equation}
%for some constant $C$ depending only on $n$.
%In particular, if $v(y_0)=0, v\geq 0,$ then 
%\begin{equation}
%v \leq Ch\quad\text{ in }S^c_h[v](y_0)

For any $x_0\in \mathcal{F}$,  we have  $y_0:=Du(x_0)\in \partial\Om^*$.
When $h>0$ is sufficiently small, by \cite[Lemma 7.11]{CM} we have 
\begin{equation}\label{localise}
	S_h^c[v](y_0) \cap \Omega^*\subset V \quad\text{and}\quad S_h^c[v] (y_0) \cap \overline{\Omega}=\emptyset. 
\end{equation}
By \cite[Theorem 7.13]{CM} we have furthermore the strict convexity 
 \begin{equation}\label{strictconvex}
 v(y)\geq v(y_0)+Dv(y_0)\cdot (y-y_0)+ C|y-y_0|^{1+\beta}\quad \forall\, y\in \overline{V}\text{ near } y_0
 \end{equation}
for some constant  $\beta>1$, which in turn implies $u\in C^{1,\alpha'}$ as in part $ii)$ of Theorem \ref{CMCL}. 
% {\color{purple} (later, we need to make $\alpha, \beta, \alpha'$ consistent throughout the paper. )}

\begin{lemma}[Uniform density] \label{ud1} 
Let $\Omega, \Omega^*$ be as in Theorem \ref{main1}. 
Suppose that the densities $f, g$ satisfy
$\lambda^{-1}<f, g<\lambda$ for a positive constant $\lambda$.
 Let $x_0\in \mathcal{F}$,  and $y_0:=Du(x_0)\in \partial\Omega^*$.
Then for any $h>0$ small, we have 
\beq\label{UniDen}
	\frac{|S_h^c[v](y_0) \cap V|}{|S_h^c[v](y_0) |} \geq \delta ,
\eeq
where $\delta$ is a positive constant depending on $n, \lambda,\Omega^*,$
but independent of $h$.  
 \end{lemma} 
 
The above uniform density was proved in \cite[Theorem 3.1]{C96} under the condition that the source domain is polynomial convex and the target domain is convex.
Here we consider the potential $v$ in the domain $V$, and $V$ is uniformly convex near $y_0$, 
which is stronger than the polynomial convexity. 
But the target $U$ may not be convex near $x_0=Dv(y_0) \in\mathcal{F}$.
Thanks to the $C^{1,\alpha'}$ regularity of $\mathcal F$ in $ii)$ of Theorem \ref{CMCL}, 
we are able to work out a proof based on that in \cite{C96}.

 \begin{proof}
Without loss of generality, we may assume that $y_0=0$ and write $S_h^c[v](y_0)$ as $S_h^c[v]$ for brevity.  
By $iv)$ of Theorem \ref{CMCL}, we have $0\in \partial V\setminus \overline{\partial V\cap \Omega^*}\subset \partial\Omega^*.$
By John's Lemma \cite[Lemma 2.1]{C96},
 %(see \cite{LW} for a simple proof), 
 there is an ellipsoid $E$ centred at $0$ such that
 \begin{equation}\label{ima2}
 E\subset S^c_h[v]\subset C(n) E,
 \end{equation}
 where $\alpha E$ denotes the $\alpha$-dilation with respect to the centre of $E$, and the constant $C(n)$ depends only on $n$. By taking $h$ small enough, we may assume \eqref{localise} hold,
which implies that $S_h^c[v] \cap V=S_h^c[v] \cap \Omega^*$ is a convex set.
Since $S_h^c[v]$ is centred at $0\in\partial V$, for any $y \in V\cap S^c_h[v]$, we have
$\frac{1}{C(n)}y \in V\cap \frac{1}{C(n)}S^c_h[v]$. Hence,
 \begin{equation}\label{diamcompare}
 \text{diam}\left(V\cap \frac{1}{C(n)}S^c_h[v])\right)\geq \frac{1}{C(n)}\text{diam}\left(V\cap S^c_h[v])\right).
 \end{equation}

 %By 
 %the geometric decay of centered sections \cite[Lemma 7.6]{CM},  we may assume the diameter of 
 %$S^c_h[v]$ is so small such that
% $$C(n)S_h^c[v] \cap \Omega^*\subset V \ \text{and}\  C(n)S_h^c[v]  \cap \overline{\Omega}=\emptyset.$$
 %This implies that $C(n)S_h^c[v] \cap V=C(n)S_h^c[v] \cap \Omega^*$ is a convex set.
 Since $V$ is uniformly convex near $0$ and $v$ is strictly convex in $V$ near $0$,  we have
\begin{equation}\label{e20211}
 	\frac{|V\cap E|}{|E|} \geq C\left(\frac{\text{diam}(V\cap E)}{\text{diam}(E)}\right)^n.
 \end{equation}
For a proof of \eqref{e20211}, see \cite[Lemma 3.2]{C96}.
Note that the proof of \eqref{e20211} in \cite{C96} does not use the convexity of the target domain.

Suppose to the contrary that \eqref{UniDen} is not true.
Then by \eqref{ima2}, \eqref{diamcompare} and \eqref{e20211},
the quantity  $\frac{\text{diam}(V\cap S^c_h[v])}{\text{diam}(S^c_h[v])}$ is very small. 
Let $\lambda_1\ge\cdots\ge\lambda_n$ 
be the lengths of semi-axes of $E$ in the corresponding principal directions $\hat e_1, \cdots, \hat e_n$.
Let $L_h$ be the affine function such that $S_h^c[v]=\{v<L_h\}$.
Denote $x_h:=DL_h$. By \cite[Corollary 2.2]{C96} we have 
  \begin{equation}\label{ima1}
  \tilde{E}\subset Dv(S_h^c[v])\subset C \tilde{E},
  \end{equation}
where $C$ is a constant depending only on $n$, the constant $\lambda$ in \eqref{bdfg} but independent of  $v$ and $h$,
and $\tilde{E}$ is an ellipsoid with centre $x_h$, principal directions $\hat e_i$,
and lengths of semi-axes $\tilde\lambda_i \approx \frac{h}{\lambda_i}$,  $i=1,\cdots,n$. 
By \eqref{globlip}, we have  $Dv(S_h^c[v])\subset \overline\Omega.$
By Property $(ii)$ in \S\ref{s21},  
\begin{equation}\label{quad11}
v=\frac{1}{2}|y|^2+C\  \text{in any connected component of}\ \Omega\setminus\overline{U}
\end{equation}
and $S^c_h[v]\cap \Omega=\emptyset$  for $h$ small (see \eqref{localise}).
 Since $v\in C^1(\mathbb{R}^n)$ and $Dv(0)=x_0\in \Omega,$ we have that
 $Dv(B_r(0))\subset \Omega$ for $r$ sufficiently small. By the geometric decay of sections \cite[Lemma 7.6]{CM}, we have that $S_h^c[v]\subset B_r(0)$ provided $h$ is sufficiently small. Hence 
 $Dv(S^c_h[v])\subset \Omega.$
For any $y\in S^c_h[v],$ if $x:=Dv(y)\in\Omega\backslash\overline{U},$ then by \eqref{quad11} we have $Dv(x)=x=Dv(y),$ which implies
that the convex function $v$ is flat along the segment connecting $x$ and $y.$ This contradicts to \eqref{quad11}. 
Therefore 
\begin{equation}\label{cont121}
Dv(S_h^c[v])\subset \overline{U}\cap \Omega
\end{equation}
 provided $h$ is sufficiently small.
 %We can further deduce that $Dv(S_h^c[v])\subset \overline{U}$. 
 %{\color{blue} Indeed for $h$ small enough, if there exists a point $y\in S^c_h[v]$ such that $x:=Dv(y)\in \Omega\setminus\overline{U},$ then by \eqref{quad11} one has $Dv(x)=Dv(y),$ which implies
%that the convex function $v$ is flat along the segment connecting $x$ and $y.$ This contradicts to \eqref{quad11}. }
%% 
\renewcommand{\figurename}{Fig.}
\renewcommand{\captionlabeldelim}{}
\begin{figure}[h]
	\centering
	\includegraphics[width=1.0\textwidth]{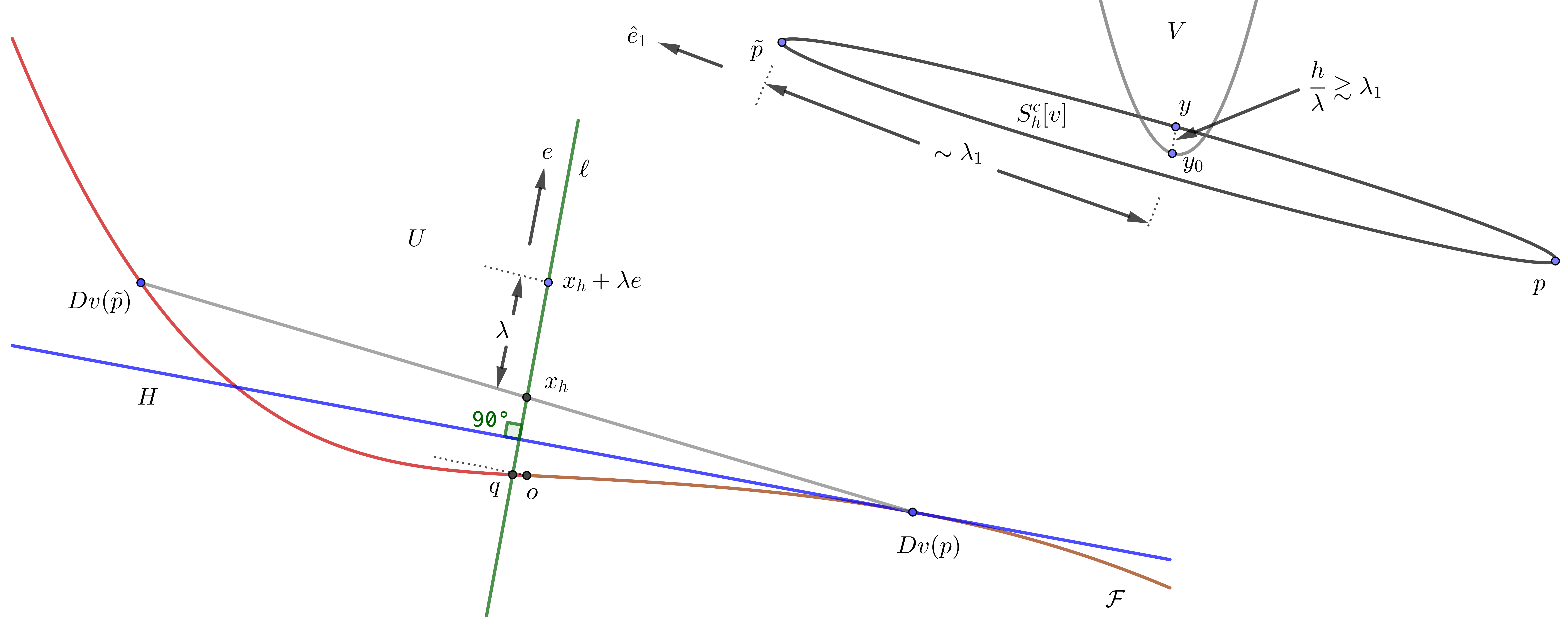}
	\caption{}
    \label{figuni}
\end{figure}

Let $p, \tilde{p}$ be the points on $\partial S_h^c[v]$ such that 
\begin{align}\label{pp}
   p\cdot \hat e_1 &=\inf\{y\cdot \hat e_1 \,:\, y\in S_h^c[v]\},\\
   \tilde{p}\cdot \hat e_1 &=\sup\{y\cdot \hat e_1 \,:\,  y\in S_h^c[v]\}. \nonumber
\end{align}
Since $\lambda_1$ is the longest axis of $E$ 
and $\frac{\text{diam}(V\cap S^c_h[v])}{\text{diam}(S^c_h[v])}$ is sufficiently small, 
we must have $p, \tilde{p}\in \mathbb{R}^n\setminus \overline{V}$, and hence $Dv(p), Dv(\tilde{p})\in \mathcal{F}$, (see Fig. \ref{figuni}). Indeed, by the same argument for the proof of \eqref{cont121}, we have that $Dv(p), Dv(\tilde{p})\in \Omega\cap\overline{U}.$  Suppose to the contrary that $Dv(p)\notin \mathcal{F},$ then $Dv(p)$ must be in the interior of $U.$ Since $Dv$ is $1-1$ from $\overline{V}$ to $\overline{U},$ there exists $q\in \overline V$ such that $Dv(q)=Dv(p):=x.$ Since $u$ is the Legendre dual of $v,$ we have that $\partial u(x)$ contains at least two points $q$ and $p,$ contradicting to the $C^1$ regularity of $u$ at $x.$ 
Hence $Dv(p)\in \mathcal{F}.$ The same argument works for $ Dv(\tilde{p})\in \mathcal{F}.$

%{\color{blue} Indeed, by the same argument as for the proof of \eqref{cont121}, we have $Dv(p)\in \Omega\cap\overline{U}.$  Suppose to the contrary that $Dv(p)\notin \mathcal{F},$ then $Dv(p)=:x$ must be in the interior of $U.$ Since $Dv$ is $1-1$ from $\overline{V}$ to $\overline{U},$ there exists a point $y\in \overline V$ such that $Dv(y)=x=Dv(p)$. Since $u$ is the Legendre dual of $v$ on $\overline U$ (see Remark \ref{uvu*}), we see $\partial u(x)$ contains two different points $y$ and $p$, which contradicts to the $C^1$ regularity of $u$ at $x.$ 
%Hence $Dv(p)\in \mathcal{F}.$ The same argument works for $ Dv(\tilde{p})\in \mathcal{F}$ as well. }

%%% From \eqref{ima1} we know that $Dv(S_h^c[v])$ is similar to the ellipsoid $\tilde E$ centred at $x_h$. 
From \eqref{pp} we know that $D(v-L_h)(p)$ and $D(v-L_h)(\tilde p)$ are parallel to $\hat e_1$, namely $Dv(p)$, $Dv(\tilde{p})$ and $x_h$ lie on a straight line. By \eqref{ima1}, 
	\begin{equation}\label{pplength}
		|Dv(p)-x_h|\approx |Dv(\tilde{p})-x_h|\leq C \frac{h}{\lambda_1}.
	\end{equation}
Let $H$ be the tangent plane of $\mathcal{F}$ at $Dv(p)$, and $\ell$ be the straight line passing through $x_h$ and perpendicular to $H$. Denote $q:=\ell\cap\mathcal{F}$ and $e:=\frac{x_h-q}{|x_h-q|}$.
From $ii)$ of Theorem \ref{CMCL}, $\mathcal{F}$ is locally a $C^{1,\alpha'}$ graph in the direction $e$.
Since the points $Dv(p), Dv(\tilde p), q$ lie on $\mathcal{F}$, 
by \eqref{pplength} and the Lipschitz continuity of $\mathcal F$, we obtain
 $$ |x_h-q|\leq C \frac{h}{\lambda_1}$$ 
for some constant $C$ independent of $h$.   

Let $\lambda'$ be the largest number such that $x_h+\lambda' e\in \overline{Dv(S_h^c[v])}$.   
For $h>0$ small, we have $x_h+\lambda' e\in U.$
From \eqref{ima1}, $Dv(S_h^c[v])$ is ``centred" about $x_h.$ 
Note that by \eqref{ima1} and \eqref{cont121} we have
\begin{equation} \label{liebelow}
\tilde{E}\subset Dv(S_h^c[v])\subset \overline{U}\cap \Omega.
\end{equation}
%{\color{blue} and $\tilde E$ is centred at $x_h$. From \eqref{liebelow} we see that $x_h$ lies strictly above the free boundary. Geometrically, $x_h+\lambda' e$ is a boundary point of $Dv(S_h^c[v])$ opposite to the point $q$, (see Fig. \ref{figuni}).}
From \eqref{liebelow} we see that $x_h,$ the centre of $\tilde{E},$  strictly lies above the free boundary.
It follows that $q$ is outside $\tilde{E}.$
Denote by $\tilde{q}$ the intersection of the segment $\overline{x_h{q}}$ with $\partial\tilde{E}.$ Then,  by \eqref{ima1} 
we have that $\tilde{q}$ and $x_h+\lambda' e$ are balanced around $x_h,$ namely, $|x_h-\tilde{q}|\approx |x_h+\lambda' e-x_h|=\lambda'.$ Hence $\lambda'\leq C|x_h-\tilde q|\leq C|x_h-q|.$
Thus by \eqref{ima1} and \eqref{liebelow} we have  
	\begin{equation}\label{eslab}
		\lambda'\leq C|x_h-q|\leq C \frac{h}{\lambda_1}.
	\end{equation}
Let  $y\in V$ be the point such that $Dv(y)=x_h+\lambda e$.   
By the definition of $\lambda'$,  we have $y\in V\cap \partial S_h^c[v].$
By the convexity of $v$, we have 
	$$ |y|\cdot |D(v-L_h)(y)| \geq |(v-L_h)(0)| = h. $$
Since $D(v-L_h)(y) = \lambda' e$, we obtain $\lambda|y|\geq h$. 
Hence from \eqref{eslab}
	$$|y|\geq  \frac{h}{\lambda'}\geq  \frac{1}{C}\lambda_1 $$ 
for some constant  $C$ independent of $h$.
That is $ \frac{|y|}{\lambda_1}\ge C^{-1}$, which contradicts to the assumption 
  $\frac{\text{diam}(V\cap S^c_h[v])}{\text{diam}(S^c_h[v])}$ is very small.
 \end{proof}
 
In this paper,  
the notation $a\lesssim b$ (resp. $a\gtrsim b$) means that there exists a constant $C>0$ independent of $h$ and the potential functions $u$ and $v$,
such that $a\leq Cb$ (resp. $a\geq Cb$), and the notation $a\approx b$ means that $C^{-1}a\leq b\leq Ca$, where $a, b$ are both positive constants.
Given a convex domain $D\subset\R^n$, 
we say that $D$ has a good shape if the eccentricity of its minimum ellipsoid is uniformly bounded.

\begin{corollary}\label{co21}
Under the conditions in Lemma \ref{ud1}, we have
\begin{itemize}
\item[$(i)$] \emph{Volume estimate:}
	\begin{equation}\label{nm1}
 		|S_h[v](y_0)|\approx |S_h^c[v](y_0)\cap V|\approx |S_h^c[v](y_0)|\approx h^{\frac{n}{2}}.
	\end{equation} 
Moreover, for any given affine transform $\mathcal{A}$, 
if one of $\mathcal{A}(S_h^c[v](y_0))$ and $\mathcal{A}(S_h[v](y_0))$ has a good shape, so is the other one.

\item[$(ii)$] \emph{Tangential $C^{1,1-\epsilon}$ regularity for $v$:} Assume in addition that $f\in C(\overline{\Omega}),\ g\in C(\overline{\Omega^*})$. Let $\mathcal{H}$ be the tangent hyperplane of $\partial\Omega^*$ at $y_0$.   
  Then $\forall\,\epsilon>0$,  $\exists \,C_\epsilon$ such that 
 \begin{equation}\label{tanc2}
 B_{C_\epsilon h^{\frac{1}{2}+\epsilon}}(y_0) \cap \mathcal{H} \subset S_h^c[v](y_0) \ \ \ \text{for $h>0$ small.} 
 \end{equation}
%{\color{red} here $0$ should also be replaced by $y_0$?} 
%\item[$(iii)$] The estimate \eqref{tanc2} also holds in the sub-level $S_h[v]$, namely 
%\beq
 %B_{C_\epsilon h^{\frac{1}{2}+\epsilon}}(0)\cap \partial V \subset S_h[v]\quad \ \text{for $h>0$ small.} 
 %\eeq
\end{itemize}
\end{corollary}

\begin{proof}  
As in the proof of Lemma \ref{ud1},
let us assume that $y_0=0\in  \partial V\setminus \overline{\partial V\cap \Omega^*}\subset \partial\Omega^*$ 
and write $S_h^c[v](0), S_h[v](0)$ as $S_h^c[v], S_h[v]$ for brevity. 
By the strict convexity estimate of $v$  in $\overline{V}$ (see \eqref{strictconvex}) and the fact that $S^c_h[v]$ is balanced around $0$,
we have an equivalence relation between $S_h[v]$ and $S^c_{h}[v]$: 
\beq\label{equi0} 
S^c_{b^{-1}h}[v] \cap V \subset S_h[v] \subset S^c_{bh}[v] \cap V\ \ \ \forall\, h>0 \text{ small},
\eeq
where $b\geq1$ is a constant independent of $h$. 
For a proof of \eqref{equi0}, we refer the reader to  \cite[Lemma 2.2]{CLW1}.
%{\Small\color{blue}{(the strict convexity of $v$ seems not a main reason for \eqref{equi0} )}}

From Lemma \ref{ud1} and \eqref{equi0}, the volume estimate \eqref{nm1} 
can be deduced similarly as in \cite[Corollary 3.1]{C96}. 
Note that by \eqref{localise} we have that $\det\, D^2v=\tilde{f}(y)\chi_{S^c_h[v]\cap\Omega^*}$
in $S_h^c[v],$ where $\tilde{f}(y)=\frac{g(y)}{f(Dv(y))}\in C(S^c_h[v]\cap \overline{\Omega^*}).$ 
Then, the proof of tangential $C^{1,1-\epsilon}$ estimate  is the same as in \cite[Lemma 4.1]{C96}.
\end{proof}

%\begin{notation}\label{uvxy}
%For clarity, we will use $x$ for variables of $u$ and $y$ for variables of $v$. 
%Accordingly, we will use $x$ to denote points in $S_h[u]$ and $y$ for points in $S_h[v]$.
%But $x, y$ are in the same coordinates of  $\R^n$.
%We also use $p, q,$ etc to denote a point in $\R^n,$ and use $e_i$ to denote the unit vector on the positive $x_i$-axis for $x=1,\cdots, n.$
%\end{notation}

\section{$C^{1,1-\epsilon}$ regularity of $\mathcal{F}$}\label{S5}

In this section, we establish the  $C^{1,1-\epsilon}$ regularity of the free boundary $\mathcal{F}$ for any $\epsilon>0$.
%To do this, we assume the unit inner normal $\nu_{_U}(x_0)$ of $U$ at the point $x_0\in\mathcal{F}$ and the unit normal $\nu_{_V}(y_0)$ of $V$ at its image $y_0=Du(x_0)$ satisfy the ``obliqueness" condition, i.e.
%\beq \label{obq-n00}
%\nu_{_U} (x_0)\cdot \nu_{_V} (y_0)>0.
%\eeq
To do this, we assume that the ``obliqueness" property holds, namely at any point $x_0\in\mathcal{F}$ and its image $y_0=Du(x_0)$,
\beq \label{obq-n00}
\nu_{_U} (x_0)\cdot \nu_{_V} (y_0)>0,
\eeq
where $\nu_{_U}(x_0)$ is the unit inner normal of $U$ at $x_0$ and $\nu_{_V}(y_0)$ is the unit inner normal of $V$ at $y_0$. 
This assumption will be verified in the last  section \S\ref{Sob}.
Under the condition \eqref{obq-n00}, the boundary value problem \eqref{MAEv}
% and \eqref{bdy-u}
is locally an oblique derivative problem of the Monge-Amp\`ere equation.
%{\color{purple} Again as in the introduction, I don't think this sentence is correct ...}

%Later in \S\ref{Sob}, we will show that this is always the case.   

\begin{theorem}\label{C11free}
Assume that $\Om, \Om^*\subset\mathbb{R}^n$ are uniformly convex domains with $C^2$ boundaries,
$f\in C(\bom)$, $g\in C(\overline{\Omega^*})$ are positive and continuous, and \eqref{obq-n00} holds.
Then $\mathcal{F}$ is $C^{1,1-\epsilon}$ smooth, for any small $\epsilon\in (0,1)$.
\end{theorem}

\begin{remark}\label{r3.1}
There is no $C^{1,1}$ estimate for the oblique derivative problem of the Monge-Amp\`ere equation.
Indeed, let 
$u(x)=(1+x_n^2)\, \big(\sum_{i=1}^{n-1} x_i^2\big)^{1-\frac 1n}$, $n\ge 3$.
Then in $\Om:=B(0, 1/n)$, $u$ satisfies
$$\det\,(D^2 u)=(4-4/n)^{n-1}(1+x_n^2)^{n-2}(1-2/n-(3-2/n)x_n^2)>0.$$
On the boundary $\p\Omega\cap\{\sum_{i=1}^{n-1}x_i^2<n^{-2}\}$, let
$$\beta (x)=\big(\beta_1(x),\cdots, \beta_n(x)\big)
=\Big(\frac{n}{n-1}\frac{x_1x_n}{1+x_n^2},\cdots,
	      \frac{n}{n-1}\frac{x_{n-1}x_n}{1+x_n^2}, -1 \Big). $$ 
Then $\beta(x)$ is smooth and
$$\frac{\p u}{\p\beta}(x)=\sum_{i=1}^n\beta_i\frac{\p u}{\p x_i}(x)=0.$$
 Let $\mathcal{N}_r:=\partial\Omega\cap\{\sum_{i=1}^{n-1}x_i^2<r^2\}\cap\{x_n>0\}$ for $r<n^{-1}$. Then 
	$$ \beta(x) \cdot \nu(x) > 0\quad \forall\, x\in\mathcal{N}_r, $$
where $\nu(x)$ is the unit inner normal vector at $x\in\mathcal{N}_r$. 
However, $u$ is not  $C^{1,\alpha}$ at $\hat x$ for any $\alpha>1-2/n$, where $\hat x=(0,\cdots,0,n^{-1})\in\mathcal{N}_r$ is the north pole. 
This function $u$ is Pogorelov's counter-example to the interior regularity of the Monge-Amp\`ere equation.
In \cite{U2}, an additional condition is imposed to obtain the $C^{1,1}$ a priori estimate.
\end{remark}

By \eqref{normalformular}, it suffices to show that $Du$ is $C^{1-\epsilon}$ along the free boundary $\mathcal{F}$.
For any $x_0\in \mathcal{F}$, we have $y_0=Du(x_0)\in   \partial V\setminus \overline{\partial V\cap \Omega^*}\subset \partial\Omega^*.$   
First we show that under the hypothesis \eqref{obq-n00}, there exists an affine transform $A$ with $\det\, A=1$ such that $\nu_{_U} (x_0)$ and $\nu_{_V} (y_0)$ become parallel. 
Indeed, by \eqref{obq-n00} without loss of generality we assume 
$$\nu_{_U} (x_0)=e_n=(0,\cdots,0,1) \quad\mbox{ and }\quad \nu_{_V} (y_0)=(0,\cdots,0,\sin\theta,\cos\theta)$$ 
for a $\theta\in(-\pi/2,\pi/2)$.
Let
\begin{equation}\label{matA}
	A =
\left(
\begin{array}{c|cc}
\mathbf{1}_{n-2}  & & \\
\hline 
 & 1 & c    \\
 & 0 & 1   
\end{array}
\right), \quad \tilde x=Ax, \quad \tilde y = (A^t)^{-1} y,
\end{equation}
where $\mathbf{1}_{n-2}$ is the $(n-2)\times(n-2)$ identity matrix, and the constant $c=-\tan\theta$.
By calculation,  
	$$ \tilde\nu = \frac{(A^t)^{-1}\nu_{_{\tilde U}} (\tilde x_0)}{|(A^t)^{-1}\nu_{_{\tilde U}} (\tilde x_0)|}=e_n \quad\mbox{ and }\quad \tilde\nu^*=\frac{A\nu_{_{\tilde V}}(\tilde y_0)}{|A\nu_{_{\tilde V}}(\tilde y_0)|}=e_n $$ 
are the unit inner normals of $\tilde{U}:=AU$ at $\tilde x_0$ and $\tilde{V}:=(A^t)^{-1}V$ at $\tilde y_0$, respectively. 
See \cite[(4.7)]{CW1} for more details. 
Denote $\tilde u(\tilde x) = u(A^{-1}\tilde x)$, $\tilde f(\tilde x) = f(A^{-1}\tilde x)$, $\vec f(\tilde x)=f(A^t\tilde x),$
$\tilde v(\tilde y) = v(A^t\tilde y),$ $\tilde g(\tilde y) = g(A^t\tilde y)$ and $\vec g(\tilde{y})= g(A^{-1}\tilde y).$
Then correspondingly, \eqref{push} becomes 
\begin{equation}\label{newpush}
\begin{split}
 (D\tilde{u})_\# \left(\tilde{f}\chi_{\tilde{U}}+\vec{g}\chi_{_{A({\Omega}^*\setminus {V})}}\right) &=\tilde{g}\chi_{\tilde\Omega^*}, \\
 (D\tilde{v})_\# \left(\tilde{g}\chi_{\tilde{V}}+\vec{f}\chi_{_{(A^t)^{-1}({\Omega}\setminus {U})}}\right) &=\tilde{f}\chi_{\tilde{\Omega}},
\end{split}
\end{equation}
where $\tilde{\Omega}=A\Omega$ and  $\tilde\Omega^*=(A^t)^{-1}\Omega^*.$

Next, we make the translations by letting 
\begin{equation}\label{pingyi}
\begin{split}
	\hat x &= T_1(\tilde x) = \tilde x - \tilde x_0,\\ 
	\hat y &= T_2(\tilde y) = \tilde y-\tilde y_0,
\end{split}
\end{equation} 
and define
\begin{align*}
	\hat u (\hat x) &= \tilde u(\tilde x) - \tilde x\cdot\tilde y_0 \\
	\hat v (\hat y) &= \tilde v(\tilde y).
\end{align*}
By subtracting a constant and change of coordinates, we may assume that $\hat u(0)=\hat v(0)=0$, and $\hat u, \hat v\geq0$. 
Denote $\hat f(\hat x)=\tilde f(\hat x+\tilde x_0)$, $\check{f}(\hat x)=\vec{f}(\hat x+\tilde y_0)$, $\hat g(\hat y)=\tilde g(\hat y+\tilde y_0)$ and $\check{g}(\hat y)=\vec{g}(\hat y+\tilde x_0)$.
Denote also $ \hat{\mathcal{F}}=A\mathcal{F}-\{\tilde{x}_0\}, \hat{\Omega}=\tilde\Omega-\{\tilde{x}_0\}, \hat{\Omega}^*=\tilde\Omega^*-\{\tilde{y}_0\},$ $\hat{U}=\tilde U-\{\tilde{x}_0\}$ and $\hat{V}=\tilde V-\{\tilde{y}_0\}$.
%{\color{purple} or better to write like? Denote $\hat f(\hat x)=\tilde f(T_1^{-1}\hat x)$, $\check{f}(\hat x)=\vec{f}(T_2^{-1}\hat x)$, $\hat g(\hat y)=\tilde g(T_2^{-1}\hat y)$ and $\check{g}(\hat y)=\vec{g}(T_1^{-1}\hat y)$.}
Then correspondingly, \eqref{newpush} becomes 
\begin{equation}\label{pushpush}
\begin{split}
 (D\hat {u})_\# \left(\hat{f}\chi_{\hat{U}}+\check{g}\chi_{_{T_1(A({\Omega}^*\setminus {V}))}}\right) &=\hat{g}\chi_{\hat\Omega^*}, \\
 (D\hat{v})_\# \left(\hat{g}\chi_{\hat{V}}+\check{f}\chi_{_{T_2((A^t)^{-1}({\Omega}\setminus {U}))}}\right) &=\hat{f}\chi_{\hat{\Omega}}.
\end{split}
\end{equation}
Note that $\hat{u}, \hat{v}, \hat{\mathcal{F}}, \hat{\Omega}, \hat{\Omega}^*, \hat{U}$ and $\hat{V}$ have the same regularity as $u, v, \mathcal{F}, \Omega,\Omega^*, U$ and $V.$ For simplicity of notations we still denote them by $u, v, \mathcal{F}, \Omega,\Omega^*, U, V.$ 

%Let $\tilde{u}(x)=u(A^{-1}x)$, $\tilde{v}(y)=v(A^ty)$ and $\tilde{\mathcal{F}}=A\mathcal{F}$. 
%$\tilde{\Omega}=A\Omega$ and $\tilde{\Omega}^*=(A^t)^{-1}\Omega^*.$ 
%Let $\tilde{f}(x)=f(A^{-1}x)$ and $\tilde{g}(y)=g(A^ty).$
%Note that the above translation and transformation is merely for simplifying the notations. 
%In particular, $\tilde{u}, \tilde{v}$ and $\tilde{\mathcal{F}}$ have the same regularity of $u, v$ and $\mathcal{F}$ as in the original partial transport problem \eqref{push}. 
%Hence, we still denote them by $u, v, \mathcal{F}$.

%Then correspondingly, \eqref{push} becomes 
%\begin{align*}
%& (D\tilde{u})_\# \left(\tilde{f}\chi_{\tilde{U}}+\tilde{g}\chi_{\tilde{\Omega}^*\setminus \tilde{V}}\right)=\tilde{g}\chi_{\tilde\Omega^*}, \\
%& (D\tilde{v})_\# \left(\tilde{g}\chi_{\tilde{V}}+\tilde{f}\chi_{\tilde{\Omega}\setminus \tilde{U}}\right)=\tilde{f}\chi_{\tilde{\Omega}}.
%\end{align*}
%The properties $(i)$--$(iii)$ following \eqref{push} also hold, particularly $\tilde{u}^*=\tilde{v}$ in $\tilde{V}$ and $\tilde{u}^*$ is strictly convex in $\tilde{\Omega}^*.$ 
%Note that $\tilde{u}, \tilde{v}, \tilde{\mathcal{F}}, \tilde{\Omega}, \tilde{\Omega}^*, \tilde{U}$ and $\tilde{V}$ have the same regularity as $u, v, \mathcal{F}, \Omega,\Omega^*, U$ and $V.$ For simplicity of notations we still denote them by $u, v, \mathcal{F}, \Omega,\Omega^*, U, V.$ 

%%
%\renewcommand{\figurename}{Fig.}
%\renewcommand{\captionlabeldelim}{}
%\begin{figure}[h]
%	\centering
%	\includegraphics[width=0.7\textwidth]{pic6}
%	\caption{}
%    \label{figtan}
%\end{figure}

By the above transformation and change of coordinates, we can assume that $\nu_{_U} (0)=\nu_{_V}(0)=e_n$,  
and locally near $0$, $\partial U$ and $\p V$ are represented as
\begin{align*}
 \partial U & = \left\{ x \,:\, x_n=\rho(x'),\ \ x'=(x_1,\cdots,x_{n-1}) \right\},\\
 \partial V & = \left\{ y \,:\, y_n=\rho^*(y'),\ \ y'=(y_1,\cdots,y_{n-1}) \right\}, 
 \end{align*}
where the function $\rho$ satisfies $\rho(0)=0$, $D\rho(0)=0$.
By $ii)$ of Theorem \ref{CMCL} and the interior ball property of $\mathcal{F}$,  we have
 \begin{equation}\label{ziyou1}
 -C|x'|^{1+\alpha'}\leq \rho(x')\leq C|x'|^2 \quad\text{ for some } \alpha'\in(0,1).
 \end{equation}
Meanwhile, the function $\rho^*$ satisfies $\rho^*(0)=0$, $D\rho^*(0)=0$;
and by the $C^2$ regularity and uniform convexity of $\p\Om^*$, we also have
  \begin{equation}\label{ziyou2}
  \frac{1}{C}|y'|^2\leq \rho^*(y')\leq C|y'|^2.
\end{equation}

In the following we aim to prove Theorem \ref{C11free}, or equivalently the $C^{1,1-\eps}$ regularity of $u$. 
Due to the lack of convexity and regularity of the free boundary $\mathcal{F}$,
we need careful analysis of the local geometry of the functions $u, v$.

\begin{lemma} \label{tc12}
For any $\epsilon>0$ small, there exists a constant $C_\epsilon$ such that
\beq\label{ustconv}
	u(x)\geq C_\epsilon |x'|^{2+\epsilon}\ \ \ \text{for}\ x\in U\  \text{near} \ 0.
\eeq  
\end{lemma}

\begin{proof}
Let $x=(x',x_n)\in U$ be a point near the origin and $|x'|\neq0$. (For $|x'|=0$, \eqref{ustconv} is trivially true.)
Denote $e:=\frac{(x',0)}{|(x',0)|}$ a unit vector in $\text{span}\{e_1,e_2,\cdots, e_{n-1}\}$, 
such that $x=|x'|e+x_ne_n$. 
Consider $z=te+\rho^*(te)e_n\in \partial V$ for some small $t>0$ to be determined. 

Given any $\epsilon>0$ small,
by \eqref{tanc2} and \eqref{secrela}, we have $v(te)\leq C_\epsilon t^{2-\epsilon}$.   
Since $Dv(\mathbb{R}^n)\subset \Omega$ is bounded, from \eqref{ziyou2} we have
\begin{align*}
v(z) &\leq v(te)+|v(z)-v(te)|\\
&\leq v(te)+C\rho^*(te)\\
&\leq C_\epsilon t^{2-\epsilon}+Ct^2\leq 2C_\epsilon t^{2-\epsilon} .
\end{align*} 
By the duality and noting that $u^*=v$ in $V$ (see Remark \ref{uvu*}), we then obtain 
% {\Small\color{blue} (is $u$ the duality of $v$?  $u, u^*, v, v^*$ in this paper can be rather confusing.)}
\begin{align*}
u(x)&= \sup_{y\in V}\left\{x\cdot y-v(y)\right\}\\
&\geq x\cdot z-v(z)\\
&\geq x\cdot \big(te+\rho^*(te)e_n\big)-C_\epsilon |t|^{2-\epsilon}\\
&\geq t|x'|-C|x_n|t^2-C_\epsilon |t|^{2-\epsilon}.
\end{align*}
Since $x\in U$ is close to $0$, by choosing $t=|x'|^{1+3\epsilon}$,  we thus obtain
\begin{align*}
	u(x) &\geq |x'|^{2+3\epsilon} - C|x'|^{2+6\epsilon}-C_\epsilon |x'|^{2+5\epsilon-3\epsilon^2} \\
		&\geq C|x'|^{2+3\epsilon}
\end{align*}
provided $|x|$ is sufficiently small. 
Hence we have the desired estimate.
\end{proof}

\begin{lemma}\label{normal1} 
For any $\epsilon>0$ small, there exists a constant $C_\epsilon$ 
 such that 
 $$ u(te_n) \leq C_\epsilon |t|^{2-\epsilon} \quad  \text{ for }  |t| \ \text{small}.$$
\end{lemma}

\begin{proof}
Let $q\in \partial S_h[v]$ be the point such that 
\begin{equation}\label{topq}
q_n=\sup\left\{y_n \,:\, y\in S_h[v]\right\}.
\end{equation}
By \eqref{equi0}, $q\in S_{bh}^c[v].$
By \eqref{nm1} and \eqref{tanc2}, we have 
$$q_n\leq C_\epsilon\frac{|S_{bh}^c[v]|}{ h^{(\frac{1}{2}+\epsilon)(n-1)}}
   \leq C_\epsilon\frac{h^{\frac{n}{2}}}{h^{(\frac{1}{2}+\epsilon)(n-1)}}
       =C_\epsilon h^{\frac{1}{2}-(n-1)\epsilon}.$$
Let $y\in \Omega^*$ be a point near the origin such that $v(y)=h$. 
The above estimate implies that 
$$y_n\leq q_n\leq C_\epsilon h^{\frac12-\epsilon}$$ 
for any given $\epsilon>0$ small. Hence we have
\begin{equation}\label{lb110}
v(y) \geq C_\epsilon |y_n|^{2+\epsilon}\ \ \text{for $y\in \Omega^*$ near the origin. }
\end{equation}
By properties ($i$) and ($iii$) before Remark \ref{uvu*} we then have
\begin{equation}\label{lb11}
u^*(y) \geq C_\epsilon |y_n|^{2+\epsilon}\ \ \text{for all}\ y\in \Omega^*.
\end{equation}
By duality and \eqref{lb11}, we then obtain
\begin{align*}
u(te_n)&= \sup_{y\in\Om^*}\left\{ te_n\cdot y - u^*(y) \right\} \\
&\leq  \sup_{y\in\Om^*}\left\{ t y_n - C_\epsilon |y_n|^{2+\epsilon} \right\} \\
&\leq  \sup_{y_n\in \mathbb{R}}\left\{ t y_n - C_\epsilon |y_n|^{2+\epsilon} \right\} \\
&\leq C_\epsilon |t|^{2-\epsilon}
\end{align*} 
for $|t|$ small.  
 % {\Small\color{red}(using the properties (i) and (iii) before remark 2.1, we can prove the estimates for all $t$ small, no matter its positive or negative)}
\end{proof}

Similarly to \eqref{sect} and \eqref{sub}, we can define the sub-level sets $S^c_h[u](x_0)$ and $S_h[u](x_0)$ for $u$.
Note that $S^c_h[u](x_0)$ is always convex but $S_h[u](x_0)$ may not be convex if  the free boundary $\F$ is not convex near $x_0$.
%%% By the estimates in Lemmas \ref{tc12} and \ref{normal1}, we are able to derive the uniform density property for $u$ as follows.

\begin{lemma} \label{ud3}
%Let $x\in\mathcal{F}$. 
For any $h>0$ small, we have
$$\frac{|S^c_h[u]\cap U|}{|S^c_h[u]|}\geq \delta_0$$ 
for a constant $\delta_0>0$ independent of $h$, 
where $S^c_h[u]=S^c_h[u](0)$.
\end{lemma}

\begin{proof}
%By a change of coordinates, we assume that $x=0$. 
Let $z=se_n, \tilde{z}=-\tilde{s}e_n$ be the intersections of $\partial S_h^c[u]$ and the $x_n$-axis, where $s, \tilde{s}>0.$
Since $S_h^c[u]$ is balanced around $0$,  we have $s\approx \tilde{s},$
and either $u(z)\geq Ch$ or $u(\tilde{z})\geq Ch.$
Then by Lemma \ref{normal1}, we obtain
\begin{equation}\label{gs1}
s\approx \tilde{s}\geq C_\epsilon h^{\frac{1}{2}+\epsilon}
\end{equation}
for any given $\epsilon>0$ small.

%By the strict convexity of $u$ in $\overline U$, the equivalence relation \eqref{equi0} also holds for sub-level sets of $u$. 
By Remark \ref{uniest11} and Lemma \ref{tc12} we have
\begin{equation}\label{gs2}
S^c_h[u]\cap U\subset S_{Ch}[u]\cap U\subset \big\{x \,:\, |x'|<C_\epsilon h^{\frac{1}{2}-\epsilon} \big\}.
\end{equation}
Recall that  $\rho(x')\leq C|x'|^2$ from \eqref{ziyou1}. 
Given any $x$ in the closure of $S_h^c[u]\cap \{x \,:\, x_n\geq C'h^{1-2\epsilon}\}\cap \overline{U},$ by \eqref{gs2}
we have that $|x'|<C_\epsilon h^{\frac{1}{2}-\epsilon},$ which implies $\rho(x')< C' h^{1-2\epsilon}\leq x_n,$
 where $C'=2CC_\epsilon^2.$
Hence $x\in U.$ This implies that 
\begin{equation}\label{cptinc121}
S_h^c[u]\cap \{x \,:\, x_n\geq C'h^{1-2\epsilon}\}\cap \overline{U}\subset\subset U.
\end{equation}
Now, if there is some $x\in S_h^c[u]\cap \{x \,:\, x_n\geq C'h^{1-2\epsilon}\}\backslash U,$ the segment connecting 
$x$ and $z$ will intersect $\partial U$ at some point $y.$ 
Since $x, z \in  S_h^c[u]\cap\{x \,:\, x_n\geq C'h^{1-2\epsilon}\},$
by convexity of $u,$ we have that $z\in S_h^c[u]\cap \{x \,:\, x_n\geq C'h^{1-2\epsilon}\}\cap \partial U,$ 
which contradicts to \eqref{cptinc121}.
Hence, we have   
\begin{equation}\label{gs3}
S_h^c[u]\cap \{x \,:\, x_n\geq C'h^{1-2\epsilon}\}\subset U.
\end{equation}
This implies that a large portion of $S^c_h[u]$ is contained in $U$, see Fig. \ref{figud}. 
%{\color{blue}
%In fact, to see \eqref{gs3} we suppose to the contrary that there exists a point $x\in S_h^c[u]\cap \{x \,:\, x_n\geq C'h^{1-2\epsilon}\}$ but not belong to $U$. Then the segment connecting $x$ and $z$ will intersect $\partial U$ at a point $y\in S_h^c[u]\cap \mathcal{F}$. Since $x_n\geq C'h^{1-2\epsilon}$, by \eqref{gs1} we have $y_n\geq C'h^{1-2\epsilon}$. Choosing $C'=4CC_\epsilon^2$, from \eqref{ziyou1} we obtain 
%	$$ |y'|\geq(C'/C)^{1/2}h^{\frac12-\epsilon}=2C_\epsilon h^{\frac12-\epsilon}, $$
%which contradicts \eqref{gs2}, and thus the claim \eqref{gs3} holds. 
%}

\renewcommand{\figurename}{Fig.}
\renewcommand{\captionlabeldelim}{}
\begin{figure}[h]
	\centering
	\includegraphics[width=0.7\textwidth]{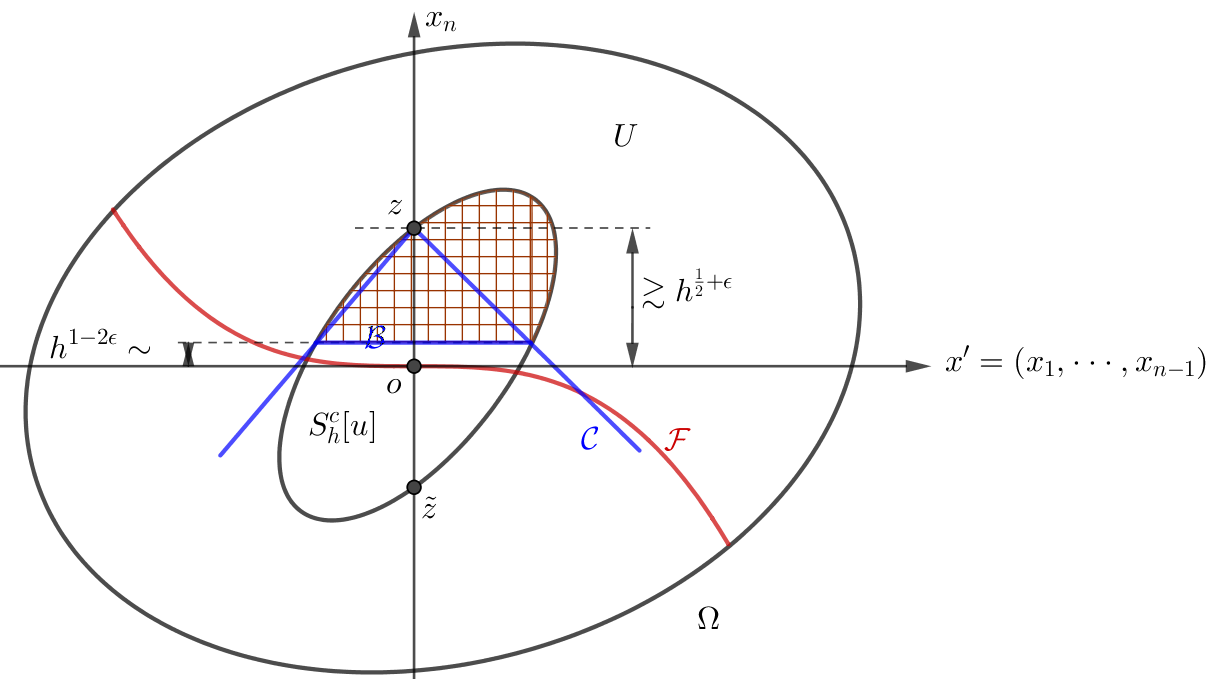}
	\caption{}
    \label{figud}
\end{figure}

By John's Lemma, there exists an ellipsoid $E$ centred at $0$, 
such that $E\subset S_h^c[u]\subset CE$ for a constant $C$ depending only on $n.$
From \eqref{gs1}, $s \gg h^{1-2\epsilon}$ for $h$ small.
By the convexity of $S_h^c[u]$ and \eqref{gs3}, we have
	\begin{equation}\label{ud11}
	\begin{split} 
		\left| S_h^c[u]\cap U \right| &\geq \left| S_h^c[u]\cap \{x_n\geq C'h^{1-2\epsilon}\} \right| \\
		&\geq |E\cap  \{x_n\geq C'h^{1-2\epsilon}\} |\\
		&\geq c\frac{s-h^{1-2\epsilon}}{s}|E|\\
				&\geq \frac{1}{2}c|S_h^c[u]|,
	\end{split}
	\end{equation}
where the constant $c>0$ only depends on $n.$
 %and $|S'_h|$ denotes the volume of $S'_h$ in $\R^{n-1}$. 
%Again, by the convexity of $S_h^c[u]$, we also have
%\begin{equation}\label{ud22} 
%	|S^c_h[u]| \leq C_1 s |S_h'|,
%\end{equation}
%for a constant $C_1$ depending only on $n$.
%Combining inequalities \eqref{ud11} and \eqref{ud22}, we then obtain
Hence $\frac{|S^c_h[u]\cap U|}{|S^c_h[u]|}\geq c/2$.   
\end{proof}

\begin{remark}
Since $z_n=s \gg h^{1-2\epsilon}$ for $h$ small,  by \eqref{gs3} and the strict convexity of $u$ in $U$,  
we see that $S^c_h[u]$ converges to $\{0\}$ as $h\rightarrow 0$. 
\end{remark}

\begin{corollary}\label{volume1}
We have the following estimates for $h>0$ small. 
\begin{itemize}
\item[$(i)$] \emph{Doubling property:}
$|\frac{1}{2}S^c_h[u]\cap U|\geq C |S^c_h[u]\cap U|$.
\item[$(ii)$] \emph{Volume estimate:}
$|S_h^c[u]| \approx |S_h^c[u]\cap U|\approx h^{\frac{n}{2}}.$
\end{itemize}
\end{corollary}

\begin{proof}
The doubling property follows from the proof of Lemma \ref{ud3}.     
Indeed, let $E, s$ be defined as above.  
Similarly to \eqref{ud11},   
\begin{align*}
\Big|\frac{1}{2}S^c_h[u]\cap U\Big| &\geq \Big|\frac{1}{2}E\cap  \{x_n\geq C'h^{1-2\epsilon}\}\Big |\\
	&\geq c\frac{\frac{1}{2}s-h^{1-2\epsilon}}{\frac{1}{2}s}|E|\\
				&\geq \frac{1}{2}c|S_h^c[u]|
\end{align*}
for a constant $c$ depending only on $n.$
%By the convexity of $S^c_h[u]$, for \eqref{ud22} we actually have
% $$ C_1s|S'_h| \leq |S^c_h[u]| \leq C_2s|S'_h|,$$
%where the constants $C_1, C_2$ depends only on $n$. 
Hence we obtain $|\frac{1}{2}S^c_h[u]\cap U|\geq C |S^c_h[u]\cap U|$.
 
Since the above doubling property is invariant under linear transforms of coordinates, 
similarly as in \cite[Corollary 2.1]{C96}, we can obtain
$$ \frac{\left|S_h^c[u]\right| \cdot \left|S_h^c[u]\cap U\right|}{h^n} \approx 1. $$
Therefore, by the uniform density of Lemma \ref{ud3}, we have the desired volume estimate.  
\end{proof}

In order to normalise the sub-level set $S^c_h[u]$, we need to strengthen estimate \eqref{gs2} to
\begin{equation}\label{cone444}
S^c_h[u] \subset \big\{x\in\R^n \,:\, |x'|\leq C_\epsilon h^{\frac{1}{2}-\epsilon} \big\}
\end{equation}
for any given $\epsilon>0$ small. 
The inclusion \eqref{cone444} can be proved as follows. 
Let $z=se_n$ be as in the proof of Lemma \ref{ud3}. 
From \eqref{gs3} and \eqref{gs2}, one sees that
\begin{equation}\label{cone222}
S^c_h[u]\cap \big\{x  \,:\, x_n\geq C'h^{1-2\epsilon}\big\}\subset \big\{x\in\R^n \,:\, |x'|\leq C_\epsilon h^{\frac{1}{2}-\epsilon} \big\}.
\end{equation}
Denote the intersection  $S^c_h[u]\cap\{x_n=C'h^{1-2\epsilon}\}=:\mathcal{B}$ 
with the same constant $C'$ in \eqref{cone222}.
Let $\mathcal{C}$ be the convex cone with vertex $z$ and base $\mathcal{B}$, namely
$$\mathcal{C} = \{z+t(x-z) \,:\, t\geq 0,\  x \in \mathcal{B}\}.$$ 
By convexity, we have (see Fig. \ref{figud})
\begin{equation}\label{cone111}
S^c_h[u]\cap \big\{0\leq x_n\leq C'h^{1-2\epsilon} \big\} \subset \mathcal{C}.
\end{equation}
From \eqref{gs1}, $s \gg h^{1-2\epsilon}$.  
Then by \eqref{cone222} and \eqref{cone111} we have 
\begin{equation}\label{cone333}
S^c_h[u]\cap\big\{x \,:\, 0\leq x_n\leq C'h^{1-2\epsilon} \big\}
  \subset \big\{x\in\R^n \,:\, |x'|\leq C_\epsilon h^{\frac{1}{2}-\epsilon} \big\}.
\end{equation}
From \eqref{cone222}, \eqref{cone333} and the property that $S^c_h[u]$ is balanced around $0$, 
we obtain \eqref{cone444}.

Next we normalise the sub-level set $S^c_h[u]$. 
Recall that from John's lemma, analogously to \eqref{ima2} there is an ellipsoid $E\subset\R^n$ such that 
$$S^c_h[u]\sim E =\Big\{x\in\R^n \,:\, \sum_{i=1}^{n-1}\frac{(x_i-k_ix_n)^2}{a_i^2}+\frac{x_n^2}{a_n^2}\leq 1\Big\} $$
in the sense that $E \subset S^c_h[u]\subset C_nE$.
For any $\epsilon>0$ small,  by \eqref{gs1} and \eqref{cone444} we have
\begin{align}\label{ab1}
 \hskip30pt  & a_i\leq C_\epsilon h^{\frac{1}{2}-\epsilon}\ \text{ for}\ i=1,\cdots, n-1,\\
  & a_n\geq C_\epsilon h^{\frac{1}{2}+\epsilon}. \nonumber
\end{align}
Moreover, since $z=se_n\in S^c_h[u]\subset C_nE$, from \eqref{gs1} and \eqref{ab1} we have 
\begin{equation}\label{k1} 
  \hskip20pt  |k_i|\leq C_n \frac{a_i}{s}\leq C_\epsilon h^{-2\epsilon} \ \ \text{for $i=1,\cdots,n-1.$}
\end{equation}
Let $A_h : x\mapsto \hat x$ be the affine transformation 
\begin{align}\label{nmAh}  
 \hskip30pt  & \hat x_i=\frac{x_i-k_ix_n}{a_i}\quad\text{for } i=1,\cdots,n-1;\\
  & \hat x_n=\frac{x_n}{a_n}, \nonumber
\end{align}
which normalises $S^c_h[u]$ such that $A_h(E)=B_1$. 

 Let $x=(x', \rho(x'))\in \partial U$ with $|x'| = h^{\frac{1}{2}-2\epsilon}$. By a rotation of coordinates, we may assume that $x'=(h^{\frac{1}{2}-2\epsilon},0,\cdots,0)$.
By \eqref{ziyou1}, \eqref{ab1} and \eqref{k1} we have
\begin{align}\label{dada}
	|\hat x_1| &= \Big| \frac{ h^{\frac{1}{2}-2\epsilon}- k_1 \rho(x')}{a_1} \Big| \geq C_\epsilon h^{-\epsilon} \rightarrow +\infty,	\\
	|\hat x_n| &= \Big| \frac{\rho(x')}{a_n} \Big| \leq \frac{Ch^{(\frac12-2\epsilon)(1+\alpha')}}{C_\epsilon h^{\frac{1}{2}+\epsilon}}\rightarrow 0, \nonumber
\end{align}
as $h\to0$ provided $\epsilon>0$ is small enough. 
Similarly, for any $x=(x', \rho(x'))\in \partial U$ with $|x'| \leq h^{\frac{1}{2}-2\epsilon},$
we have $|\hat x_n| \rightarrow 0$ as $h\to0$ provided $\epsilon>0$ is small enough. 
Hence, for any given constant $N>0,$ we have
\begin{equation}\label{pin121}
\partial A_h(U)  \cap B_N(0)\subset \{x: |x_n|\leq c_h\}
\ \text{ for 
some constant  }\ c_h\rightarrow 0\ \text{as}\ h\rightarrow 0.
\end{equation}

Now, denote $\hat S_h :=A_h (S_h^c[u])$ and $\hat U_h :=A_h(U).$
Then \eqref{pin121} implies the volume
\begin{equation}\label{closea}
\left| \left(\hat S_h\cap \{x_n\geq 0\} \right) \, \triangle\, \left(\hat S_h\cap \hat U_h\right) \right| \rightarrow 0
\end{equation}
uniformly as $h\rightarrow 0$, 
where $A\, \triangle\, B =  (A-B)\cup(B-A)$ for two sets $A, B$.
%{\Small\color{blue}(before ''More precisely" it is by \eqref{ab1}, after ''More precisely", it is by  \eqref{dada}?)}
%{\Small\color{red}(right, both of them are by \eqref{dada}}

\begin{lemma}\label{tc14}
For any given $\epsilon>0$,  there exists a constant $C_\epsilon>0$ such that 
\begin{equation}\label{t666}
	B_{C_\epsilon h^{\frac{1}{2}+\epsilon}}(0)\cap\{x_n=0\}\subset S_h^c[u].
\end{equation}
\end{lemma}

\begin{proof}
%Similar to the observation $b)$ in the proof of \cite[Lemma 4.1]{C96}, by the convexity of $u$ and the fact that $S^c_h[u]$ is balanced around $0$, we can obtain that for any large constant $M>1$, there exists a constant $C>0$ depending only on $n$ such that
%\begin{equation}\label{einc} 
%	\frac{1}{CM} S^c_{h}[u]\subset S^c_{\frac{h}{M}}[u].
%\end{equation}
%We will prove \eqref{t666} by using an iteration argument. 
%First, we \emph{claim} that for any large constant $M>1$, there exists $h_0>0$ such that $\forall\,h\in(0, h_0]$,
%\begin{equation}\label{iterationa}
%\frac{1}{C}M^{-\frac{1}{2}}S^c_h[u]\cap \{x_n=0\}\subset S^c_{\frac{h}{M}}[u],
%\end{equation}
%where $C>0$ is a constant depending only on $n$. 
%%% {\Small\color{blue} (I feel that \eqref{iterationa} is too good to be true)}

%{\color{purple}
We will prove \eqref{t666} by an iteration argument. 
First, we \emph{claim} that there exists a constant $C_*>0$ depending only on $n$, such that 
for any large constant $M>1$, there exists $h_0>0$ such that $\forall\,h\in(0, h_0]$,
\begin{equation}\label{iterationa}
\frac{1}{C_*}M^{-\frac{1}{2}}S^c_h[u]\cap \{x_n=0\}\subset S^c_{\frac{h}{M}}[u].
\end{equation}
%}

Assuming \eqref{iterationa} for the moment, we can obtain \eqref{t666} as follows.
For any given $\epsilon>0$ small, let $M=C_*^{1/\epsilon}.$
For any $h\in (0, h_0),$ there exists an integer $k$ and a height $\bar{h}\in [\frac{h_0}{M}, h_0]$ such that $h=\frac{\bar{h}}{M^k}.$
%{\Small{\color{blue} (why don't say   ''let $M=C^{1/\eps}$'' directly ?)}} {\Small\color{red}(yes, changed)}. 
By iterating \eqref{iterationa}, we obtain
 \begin{equation}\label{ith0}
 	\frac{1}{C_*^k} M^{-\frac{k}{2}}S^c_{\bar h}[u]\cap\{x_n=0\}\subset S^c_{\frac{\bar{h}}{M^k}}[u]\ \ \text{for all}\ k\geq  1.
\end{equation}
%$\forall\,h<h_0$, let $k\geq1$ such that $\frac{h_0}{M^{k}} \leq h < \frac{h_0}{M^{k-1}}$, 
Since $k=\log_M({\bar h}/h)$, a straightforward computation shows that
$\frac{1}{C_*^k}M^{-\frac{k}{2}} =({h}/{\bar h})^{\frac12+\epsilon}.$

%namely $\log_M(\frac{h_0}{h}) \leq k < 1+\log_M(\frac{h_0}{h})$.
%Thus,  
%\begin{equation*}
%\begin{split}
%	\frac{1}{C^k}M^{-\frac{k}{2}} &= (CM^{\frac{1}{2}})^{-k} \\
%	&> (CM^{\frac{1}{2}})^{-\log_M(\frac{Mh_0}{h})} 
%			= (CM^{\frac{1}{2}})^{\log_M(\frac{h}{Mh_0})}\\
%	& = \left(M^{\log_M(\frac{h}{Mh_0})}\right)^{\log_M(CM^{\frac12})} 
%			= \Big(\frac{h}{Mh_0}\Big)^{\frac12+\epsilon}.
%\end{split}
%\end{equation*}
Recall that $u\in C^1(\R^n)$ and globally Lipschitz (see \eqref{globlip} and Theorem \ref{CMCL} $(i)$). 
It implies that for the  $\bar h>0$, $B_{r_0}(0) \subset S^c_{\bar h}[u]$ 
%for $r_0 = \frac{h_0}{2M\|Du\|_\infty}\leq \frac{\bar h}{2\|Du\|_\infty}$.  
for $r_0 = \frac{\bar h}{2\|Du\|_\infty} \geq \frac{h_0}{2M\|Du\|_\infty}$. 
 Indeed, suppose $S^c_{\bar h}[u]=\{u<L\}$ for some affine function $L,$ then 
$(u-L)(0)=-\bar h$, $u-L=0$ on $\partial S^c_{\bar h}[u],$ and $|D(u-L)|\leq 2\|Du\|_\infty,$ hence for any 
$e\in \mathbb{S}^{n-1}$ and $0\leq t< r_0$,  we have $(u-L)(te) < -\bar h+2r_0\|Du\|_\infty = 0$, which implies $te\in S^c_{\bar h}[u]$.

Hence by \eqref{ith0}, we obtain  
 $$B_{C_\epsilon h^{\frac{1}{2}+\epsilon}}(0)\cap\{x_n=0\}\subset S_h^c[u],$$
 where $C_\epsilon=r_0(h_0)^{-\frac12-\epsilon}$.
Therefore \eqref{t666} is proved. 

It remains to prove the claim \eqref{iterationa}.
Let $A_h$ be the transformation in \eqref{nmAh}. Let
	$$ u_h(x)=\frac{1}{h}u(A_h^{-1}x). $$
%Then $u_h$ satisfies the Monge-Amp\`ere equation
%	$$ \det\, D^2 u_h=\hat f\chi_{\hat S_h\cap \hat U} \ \text{in}\ \hat S_h
%	\quad \text{ with } \hat f=c_h \frac{f}{g\circ Du},$$
%	where $c_h:=\frac{|\det\,A^{-1}_h|^2}{h^n}\frac{f(0)}{g(0)}$ is a constant.	
%Since $\hat S_h=A_h(S^c_h[u]) \sim B_1$, from part $(ii)$ of Corollary \ref{volume1}, 
%$|\det\,A^{-1}_h|\approx|S^c_h[u]|\approx h^{n/2}$. 
%Hence $c_h\approx 1$ and $|\hat f-c_h|_{L^\infty(\hat S_h)} \rightarrow 0$ as $h\rightarrow 0.$
% in $\hat S_h\cap \hat U_h$ for some constant $c_0>0$, uniformly as $h\rightarrow 0$. 
%Without loss of generality, we assume $c_0=1$.
%{\color{purple}
Then $u_h$ satisfies the Monge-Amp\`ere equation
	$$ \det\, D^2 u_h=\hat f\chi_{\hat S_h\cap \hat U} \ \text{ in }\ \hat S_h
	\quad \text{ with } \hat f=\frac{|\det\,A^{-1}_h|^2}{h^n} \frac{f}{g\circ Du}\circ A_h^{-1},$$
	where $\hat S_h=A_h(S^c_h[u]) \sim B_1$ and $\hat U=A_h(U)$. 
Let the constant $c_h:=\frac{|\det\,A^{-1}_h|^2}{h^n}\frac{f(0)}{g(0)}$.	
From $(ii)$ of Corollary \ref{volume1}, 
$|\det\,A^{-1}_h|\approx|S^c_h[u]|\approx h^{n/2}$. 
Hence $c_h\approx 1$ and $|\hat f-c_h|_{L^\infty(\hat S_h)} \rightarrow 0$ as $h\rightarrow 0.$
%}
Under the above normalisation, the claim \eqref{iterationa} is equivalent to 
 \begin{equation}\label{nmclaim}
   \frac{1}{C_*} M^{-\frac{1}{2}}S^c_1[u_h] \cap \{x_n=0\} \subset S^c_{\frac{1}{M}}[u_h]. 
 \end{equation}
 
We shall prove \eqref{nmclaim} by approximating $u_h$ by $w_h$,
where $w_h$ is the convex solution to 
\begin{align} \label{Drwh}
\det\, D^2w_h  &=  c_h\chi_{\hat S_h\cap \{x_n\geq 0\}}\quad  \text{ in }\, \hat S_h, \\
	w_h  &= u_h \quad  \text{ on }\, \partial \hat S_h. \nonumber
\end{align}

Since $\hat S_h$ is centered at $0$ and $|\hat S_h|\approx 1,$ we have that $|\hat S_h\cap \{x_n\geq 0\}|\approx 1.$
Let $L_h$ be the affine function such that $\hat S_h=\{u_h <L_h\}.$ Note that $u_h(0)-L_h(0)=-1.$
Let $w'_h:=w_h-L_h,$ then $w'_h$ satisfies the same equation as $w_h$ does, and $w'_h=0$ on 
$\partial \hat S_h.$ Then, by \cite[Lemma 2.4]{C96},  we have $|w'_h(0)|\approx |\inf w'_h|\approx 1$ in $\hat S_h,$
$$\text{dist}\left(\partial \{w'_h\leq 0\},  \partial \{w'_h\leq \frac14 w'_h(0)\}\right)\geq c_1$$ and
$$\text{dist}\left(\partial \{w'_h\leq \frac14 w'_h(0)\},  \partial \{w'_h\leq \frac12 w'_h(0)\}\right)\geq c_1$$ for some positive constants $C, c_1>0$ depending only on $n.$
By convexity of $w'_h$ and \cite[Corollary A.23]{Fig3}, it follows that $\|Du\|_{L^\infty( \{w'_h\leq \frac14 w'_h(0)\})}\leq C$ for some constant $C$ depending only on $n.$ Note that by convexity of $w'_h$ we also have $\frac{1}{2}\hat S_h\subset  \{w'_h\leq \frac12 w'_h(0)\}.$
%{\color{blue} Note that $\hat S_h$ is normalised, and $w_h=L_h$ on $\partial\hat S_h$ for some affine function $L_h$.
%One can obtain (see for example \cite[Lemma 2.4]{C96}) that $|w_h|\leq C,  |Dw_h| \leq C$ in $\frac12\hat S_h$ for some constant $C$ depending only on $n.$
%}
Note also that the right hand side of equation \eqref{Drwh} is independent of $x_i$ for $i=1,\cdots, n-1$. 
Hence by Pogorelov's interior second derivative estimate (see \cite[Corollary 1.1]{C96}), 
we have 
\beq\label{Pogo}
	|D_{ii}w_h|=|D_{ii}w'_h|\leq C_1\ \text{ in } \frac{1}{2}\hat S_h,\quad i=1,\cdots, n-1
\eeq
for a constant $C_1$ depending only on $n.$ 
Hence, for any large constant $M>1$, 
\begin{equation}\label{lev11}
B_{\frac{1}{C_2}M^{-\frac{1}{2}}}(0) \cap \{x_n=0\}\subset \Big\{x \,:\, w_h(x)\leq w_h(0)+Dw_h(0)\cdot x+\frac{1}{2M}\Big\},
\end{equation}
where $C_2>0$ is a constant depending only on $n.$
Thanks to \eqref{closea}, by the comparison principle (see \cite[Lemma 1.3]{C96}), we have 
\begin{equation}\label{uniformc}
\delta_h:=\|u_h-w_h\|_{L^\infty(\frac{1}{2}\hat S_h)}\rightarrow 0\quad\text{ as }\ h\to0.
\end{equation} 

Recall that $u_h(0)=0$, $u_h\geq 0.$
Similarly to \eqref{secrela}, we have $u\leq Ch$ in $S^c_h[u]$.
Thus $0\leq u_h\leq C$ in $\hat S_h$.
%{\color{blue}
%Hence by \eqref{uniformc} and the $C^{1,1}$ estimate for  $w_h$ in $x_i$ for $i=1, \cdots, n-1$,
%\begin{equation}\label{gradest}
%-\delta_h \leq u_h-\delta_h \leq w_h(\delta_h^{\frac{1}{2}}e)\leq w_h(0)+Dw_h(0)\cdot \delta_h^{\frac{1}{2}}e+C_1\delta_h
%\end{equation}
%for any unit vector in $e\in \{x_n=0\}.$
%we have $w_h(0)\rightarrow 0$ and $Dw_h(0)\cdot x\rightarrow 0$ uniformly for $x\in \frac{1}{4}\hat S_h\cap\{x_n=0\}$
%as $h\rightarrow 0.$
%From \eqref{gradest} we deduce that $-D\omega_h(0)\cdot e\leq (C_1+2)\delta_h^{\frac{1}{2}},$ for any unit vector in $e\in \{x_n=0\}.$ Replacing $e$ by $-e,$ we have that $|D\omega_h(0)\cdot e|\leq (C_1+2)\delta_h^{\frac{1}{2}},$ for any unit vector in $e\in \{x_n=0\}.$
%Hence we have $w_h(0)\rightarrow 0$ and $Dw_h(0)\cdot x\rightarrow 0$ uniformly for $x\in \frac{1}{4}\hat S_h\cap\{x_n=0\}$
%}
%{\color{purple}
Let $e\in \{x_n=0\}$ be a unit vector. 
By \eqref{Pogo} and \eqref{uniformc}, we have
\begin{equation}\label{gradest}
-\delta_h \leq u_h(\delta_h^{{1}/{2}}e)-\delta_h \leq w_h(\delta_h^{{1}/{2}}e)\leq w_h(0)+Dw_h(0)\cdot \delta_h^{{1}/{2}}e+C_1\delta_h,
\end{equation}  
and thus 
$$-Dw_h(0)\cdot e\leq (C_1+2)\delta_h^{{1}/{2}}.$$ 
Replacing $e$ by $-e,$ we then obtain 
$$|Dw_h(0)\cdot e|\leq (C_1+2)\delta_h^{{1}/{2}}\qquad \text{$\forall$ unit vector $e\in \{x_n=0\}$}. $$
Hence we have $w_h(0)\rightarrow 0$ and $Dw_h(0)\cdot x\rightarrow 0$ uniformly for $x\in \frac{1}{4}\hat S_h\cap\{x_n=0\}$ as $h\to0$.
%}
By \eqref{lev11} and \eqref{uniformc}, it then follows that  
for any $M>1$, there exists $h_0>0$ such that $\forall\,h\in(0,h_0]$,
\begin{equation}\label{nminc}
	B_{\frac{1}{C_2}M^{-\frac{1}{2}}}(0) \cap \{x_n=0\}\subset \Big\{x\,:\,u_h(x)\leq \frac{1}{M}\Big\}.
\end{equation}
 
We now show that \eqref{nmclaim} follows from \eqref{nminc}. 
Recall that $S_{1/M}^c[u_h]=\{u_h<L\}$ for some affine function $L$ with $L(0)=\frac{1}{M}$.
For a unit vector $e\in\{x_n=0\}$,
replacing $e$ by $-e$ if necessary,
we may assume that $L$ is non-decreasing in the direction $e$,
thus by \eqref{nminc}, $\frac{1}{C_2}M^{-\frac{1}{2}}e\in S_{1/M}^c[u_h]$.
%%% {\Small\color{blue} (why ``$L$ is non-decreasing in $e$" 
%%% implies ``$\frac{1}{C_2}M^{-\frac{1}{2}}e\in S_{1/M}^c[u_h]$''? by \eqref{nminc}?)}
As $S_{1/M}^c[u_h]$ is balanced around $0$,
it implies that  $-\frac{1}{C_3}M^{-\frac{1}{2}}e\in S_{1/M}^c[u_h]$ for a different  constant $C_3>C_2$ depending only on $n$.
%{\Small\color{blue} (you cannot say $C_3$ depends only on $n$
 %if it also depends on $C_2$, unless you have shown that $C_2$ depends only on $n$)} 
 %{\Small\color{red}($C_2$ indeed depends only on $n.$)}
Therefore,%Note that by the strict convexity of $u_h$, the equivalence relation \eqref{equi0} also holds for $u_h$, namely there is a universal constant $b\geq1$ such that $\left\{u_h < \frac{1}{bM}\right\} \subset S_{\frac{1}{M}}^c[u_h]$. 
%Hence, from \eqref{nminc} there exists a constant $C_3>C_2$ depending only on $n$ such that
$$ B_{\frac{1}{C_3}M^{-\frac{1}{2}}}(0) \cap \{x_n=0\}\subset  S_{\frac{1}{M}}^c[u_h]. $$

Then, recall that $S^c_1[u_h]=A_h(S^c_h[u]) \sim B_1$ is normalised. 
Therefore, we conclude that for any $M>1$, there exists $h_0>0$ such that $\forall\,h\in(0,h_0]$, 
 \begin{equation}\label{inistep}
	\frac{1}{C_4}M^{-\frac{1}{2}}S^c_1[u_h]\cap\{x_n=0\}\subset B_{\frac{1}{C_3}M^{-\frac{1}{2}}}\cap \{x_n=0\}\subset S^c_{\frac{1}{M}}[u_h],
 \end{equation}
where the constant $C_4$ depends only on $n$. 
Rescaling back, the claim \eqref{nmclaim} is proved. 
\end{proof}

We are now in a position to prove the $C^{1,1-\epsilon}$ regularity of $u$. 

\begin{corollary}\label{alpha}
For any $\epsilon>0$ small, there exists a constant $C_\epsilon$ such that 
 \begin{align}
 & u(x)\leq C_\epsilon |x|^{2-\epsilon} \quad\text{ for }\ x\in B_{r_0}(0),  \label{gs7}\\
 & u(x)\geq C_\epsilon |x|^{2+\epsilon} \quad\text{ for }\ x\in U\cap B_{r_0}(0), \label{u11d}
 \end{align}
where $r_0>0$ is a small constant.   Moreover, we have
\begin{equation}\label{gest1}
|Du(x)|\leq C_\epsilon |x|^{1-\epsilon} \quad\text{ for }\ x\in B_{\frac{r_0}{2}}(0).
\end{equation}
\end{corollary}

\begin{proof}
By \eqref{gs1}, Lemma \ref{tc14}, and the property that $S^c_h[u]$ is balanced around $0$,  we have 
$$B_{C_\epsilon h^{\frac{1}{2}+\epsilon}}(0)\subset S^c_h[u].$$
%From the equivalence relation \eqref{equi0} for $u$, and again the fact that $S^c_h[u]$ is balanced around $0$, 
%we have $S^c_h[u]\subset \{u<Ch\}$.
By Remark \ref{uniest11} it implies that $u<Ch$ in $B_{C_\epsilon h^{\frac{1}{2}+\epsilon}}$.   
Hence $u(x)\leq C_\epsilon |x|^{2-\epsilon}$ near the origin, and so  \eqref{gs7} is proved. 

Estimate \eqref{u11d} generalises Lemma \ref{tc12} in the sense that $u$ also has a lower bound along the $x_n$ direction. 
Let $q\in \partial S_h[u]$ be the point such that 
$q_n=\sup\{x_n \,:\, x\in S_h[u]\}$.   
By \eqref{gs1} and the first inclusion of \eqref{gs2}, we have $q_n\geq C_\epsilon h^{\frac{1}{2}+\epsilon}.$
By \eqref{ziyou1} and \eqref{gs2}, we also have 
$$\tilde{D}:=S_h[u]\cap \{x_n\geq C_\epsilon h^{1-2\epsilon}\}\subset U.$$
Note that $\frac{1}{C}\leq \det D^2u \leq C$ in $\tilde{D}$ and $0\leq u\leq h$ on $\tilde D$. 
The uniform estimate for the Monge-Amp\`ere equation  \cite{GT} implies that $|\tilde D|\leq Ch^{\frac{n}{2}}$.    
On the other hand, by \eqref{gs7},
 \begin{equation}\label{gs71}
B_{C_\epsilon h^{\frac{1}{2}+\epsilon}}(0)\cap\{x_n= C_\epsilon h^{1-2\epsilon}\} \subset \tilde{D}.
\end{equation}
Hence we obtain  
$|\tilde D|\geq C_\epsilon h^{(\frac{1}{2}+\epsilon)(n-1)}(q_n-C_\epsilon h^{1-2\epsilon}),$  which implies
\begin{equation}\label{gs8}
q_n\leq C_\epsilon h^{\frac{1}{2}-(n-1)\epsilon}. 
\end{equation}
By \eqref{gs8} and Lemma \ref{tc12}, we then obtain 
$S_{h}[u]\subset B_{C_\epsilon h^{\frac{1}{2}-\epsilon}}(0)\cap U$,  and so \eqref{u11d} follows.

The gradient estimate \eqref{gest1} follows from \eqref{gs7} and the convexity of $u$. 
\end{proof}

\begin{proof}[Proof of Theorem \ref{C11free}]
By Corollary \ref{alpha},  $Du$ is $C^{1-\epsilon}$ along the free boundary $\mathcal{F}$, for any $\epsilon>0$ small. 
By \eqref{normalformular}, it follows that $\mathcal{F}$ is $C^{1,1-\epsilon}$, for any $\epsilon>0$ small. 
\end{proof}

\section{$C^{2,\alpha}$ regularity}\label{S6}

In this section, we adopt the method recently developed in \cite{CLW1} 
to prove the $C^{2,\alpha}$ regularity of $u$ up to the free boundary
$\mathcal{F}.$
Let $u, v, \Omega,\Omega^*, U, V, \rho, \rho^*$ be as in \S\ref{S5}. 
Suppose the obliqueness \eqref{obq-n00} holds, and the densities $f\in C^\alpha(\overline{\Omega})$, $g\in C^\alpha(\overline{\Omega^*})$ for some $\alpha\in(0,1)$.
%%% And keep in mind that in this section, $\epsilon>0$ is as small as we want. 

First we construct an approximate solution of $u$ in $S_h[u]$ as follows. 
Denote 
$$D_h^+=S_h[u]\cap \left\{x_n\geq h^{1-3\epsilon} \right\}. $$  
Note that by Corollary \ref{alpha},  
\begin{equation}\label{diamlevel}
\text{diam}(S_h[u])\leq C_\epsilon h^{\frac{1}{2}-\epsilon}.
\end{equation}
By Theorem \ref{C11free}, we have
\begin{equation}\label{freeregu1}
|\rho(x')|\leq C_\epsilon |x'|^{2-\epsilon}\leq C_{\epsilon}h^{1-\frac{5}{2}\epsilon}\quad  \forall\, x\in \mathcal{F}\cap\partial S_h[u],
\end{equation}
 where $x'=(x_1,\cdots, x_{n-1}).$ 
 %{\color{blue}{\Small(point out clearly the notations $x', \tilde x$ etc at the beginning of \S2)} \color{red}{we did it at the beginning of \S4 for discussion in higher dimensions.}}
Hence for $h>0$ sufficiently small,  we have $D_h^+\Subset U$, see Fig. \ref{figref} below.

Let $D_h^-$ be the reflection of $D_h^+$ with respect to the hyperplane $\left\{x_n=h^{1-3\epsilon}\right\}$.  
Denote 
\begin{equation}\label{domDh}
	D_h:=D_h^+\cup D_h^-. 
\end{equation}
Since $Du(D_h^+)\subset \Omega^*\subset \{y_n\geq 0\},$ 
we have ${u_n}\geq 0$ in $D_h^+$, 
which implies that $D_h$ is a convex set. 
Moreover, by \eqref{freeregu1} and Corollary \ref{alpha}, it is straightforward to check that 
\begin{equation}\label{dh1}
B_{\frac{1}{C_\epsilon}h^{\frac{1}{2}+\epsilon}}(0)\subset D_h\subset B_{C_\epsilon h^{\frac{1}{2}-\epsilon}}(0).
\end{equation}

\renewcommand{\figurename}{Fig.}
\renewcommand{\captionlabeldelim}{}
\begin{figure}[h]
	\centering
	\includegraphics[width=0.8\textwidth]{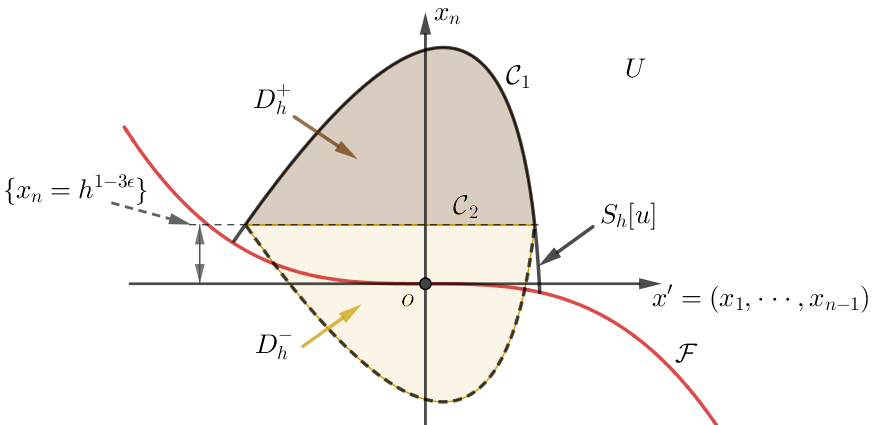}
	\caption{}
    \label{figref}
\end{figure}

Let $w$ be the solution to
\begin{equation}
\label{eq:MA v}
\begin{cases}
\det\,D^2w=1&\mbox{in $D_h$},\\
w=h&\mbox{on $\partial D_h$}.
 \end{cases}
 \end{equation}
Our proof relies on the following comparison estimate.
By the standard Alexandrov estimate for Monge-Amp\`ere equation \cite[Proposition 4.4]{Fig3} and \eqref{dh1}, we have that
$|w-h|\leq C|D_h|^{\frac{2}{n}}\leq C_\epsilon h^{1-2\epsilon}.$
Hence 
\begin{equation}\label{west199}
|w|\leq C_{\epsilon} h^{1-2\epsilon}\quad\mbox{ in } D_h.
\end{equation}

\begin{lemma}\label{diff1}
Assume that
$$\left|\frac{f}{g\circ Du}-1\right|\leq  Ch^\tau \ \ \mbox{in}\ \  D_h\cap U  $$
for a constant $\tau\in (0,1/2).$  
Then we have  the estimate 
\begin{equation}\label{coreLL}
	\|u-w\|_{L^\infty(D_h\cap U)}\leq C_1 h^{1+\tau'}
\end{equation} 
for some constant $\tau'\in (0, \tau)$ and  some constant $C_1$ independent of $h.$
\end{lemma}
\begin{remark}
Later, one can see that by Remark \ref{tau'} the exponent $\tau'$ can be improved to the same $\tau.$
\end{remark}

\begin{proof}
The boundary $\partial D_h^+ = \mathcal{C}_1\cup\mathcal{C}_2$ consists of two parts, 
where $\mathcal{C}_1\subset \left\{x_n> h^{1-3\epsilon}\right\}$ and $\mathcal{C}_2\subset \left\{x_n= h^{1-3\epsilon}\right\}$.  
We have $u=w$ on $\mathcal{C}_1$,
and by symmetry, $D_nw=0$ on $\mathcal{C}_2$.
We \emph{claim} that $0\leq D_nu \leq C_\epsilon h^{1-4\epsilon}$ on $\mathcal{C}_2$ 
for any given small $\epsilon>0$.  

To see this, for any $x=(x',h^{1-3\epsilon})\in \mathcal{C}_2 $,  let $z=(x', \rho(x')) \in \mathcal{F}$. 
By \eqref{diamlevel} and \eqref{freeregu1}, we have
$$|z-x|\leq h^{1-3\epsilon}+C_\epsilon h^{(\frac{1}{2}-\epsilon)(2-\epsilon)}\leq C_\epsilon h^{1-3\epsilon},$$
for $h$ small. 
By \eqref{gest1}, we have 
$$|Du(z)|\leq C_\epsilon |z|^{1-\epsilon}\leq C_\epsilon h^{(\frac{1}{2}-\epsilon)(1-\epsilon)}.$$
Since $Du(z) \in \partial \Om^*,$ by \eqref{ziyou2}  
we obtain 
$$ D_nu(z) \leq C_\epsilon h^{2(\frac{1}{2}-\epsilon)(1-\epsilon)}\leq  C_\epsilon h^{1-4\epsilon}. $$  
On the other hand, by Corollary \ref{alpha},
\begin{equation*}
|D_nu(x)-D_nu(z)|\leq C_\epsilon |x-z|^{1-\epsilon} \leq C_\epsilon h^{(1-3\epsilon)(1-\epsilon)}\leq C_\epsilon h^{1-4\epsilon}.
\end{equation*}
 Hence $0< D_nu(x) \leq C_\epsilon h^{1-4\epsilon},$ and the claim is proved.

Let 
\begin{align*}
 \widehat w  & = (1-h^\tau)^{1/n}w-(1-h^\tau)^{1/n}h+h, \\
\widetilde w & = (1+h^\tau)^{1/n}w-(1+h^\tau)^{1/n}h+h+2C_\epsilon(x_n-Ch^{1/2-\epsilon})h^{1-4\epsilon}. 
\end{align*}
By \eqref{eq:MA v} and choosing $C$ large,  we have
	\begin{align*}
		\det\,D^2\widehat w <\det\,D^2 u < \det\,D^2\widetilde w  & \quad\mbox{ in }  D_h^+,\\
			\widetilde w\leq u=\widehat w  = h &\quad\mbox{ on }\mathcal{C}_1, \\
		D_n\widehat w =0 < D_n u <D_n\widetilde w&\quad\mbox{ on }\mathcal{C}_2.
	\end{align*}
By the comparison principle, it follows that 
\begin{equation}\label{ineqmod1}
\widehat w \geq u\geq \widetilde w
\end{equation}
 in $D_h^+$.  
By the first inequality of \eqref{ineqmod1} and \eqref{west199} we have that   
\begin{align*}
u&\leq(1-h^\tau)^{1/n}w-(1-h^\tau)^{1/n}h+h\\
&\leq (w-h)(1-\frac{2}{n}h^{\tau})+h\\
&\leq w+\frac{2}{n}h^{1+\tau}+C_\epsilon h^{1+\tau-2\epsilon}\\
&\leq w+C_\epsilon  h^{1+\tau'}\qquad \text{ in } \ D_h^+,
\end{align*}
provided $h$ is sufficiently small and $\tau'<\tau-2\epsilon.$
By the second inequality of \eqref{ineqmod1} and \eqref{west199} we have that   
\begin{align*}
u&\geq (1+h^\tau)^{1/n}w-(1+h^\tau)^{1/n}h+h+2C_\epsilon(x_n-Ch^{1/2-\epsilon})h^{1-4\epsilon}\\
&\geq (w-h)(1+\frac{2}{n}h^{\tau})+h-2CC_\epsilon h^{3/2-5\epsilon} \\
&\geq w-\frac{2}{n}h^{1+\tau}-C_\epsilon h^{1+\tau-2\epsilon}-2CC_\epsilon h^{3/2-5\epsilon} \\
&\geq w-C_\epsilon  h^{1+\tau'}\qquad \text{ in } \ D_h^+,
\end{align*}
provided $h$ is sufficiently small and $\epsilon$ is chosen small enough.

Therefore, by choosing $\epsilon$ sufficiently small, we have
	\begin{equation}\label{coreL}
		|u-w| \leq C_\epsilon h^{1+\tau'} \qquad\mbox{in } D_h^+. 
	\end{equation}

\vspace{5pt}

Next, we estimate $|u-w|$ in $ D_h^-\cap U$.  
For $x=(x',x_n)\in D_h^-\cap U$, we have
 \begin{align}\label{xnest}
 h^{1-3\epsilon}\geq x_n
 &\geq \rho(x')   \geq -C_\epsilon |x'|^{2-\epsilon}   \nonumber  \\
  & \geq  -C_\epsilon h^{(\frac{1}{2}-\epsilon)(2-\epsilon)}\geq -C_\epsilon h^{1-3\epsilon}.
  \end{align}
    Note that the third inequality in \eqref{xnest} follows from Theorem \ref{C11free}.
Let $$z=(x', 2h^{1-3\epsilon}-x_n) \in D_h^+.$$
Then by \eqref{xnest} we have  $|x-z| \leq C_\epsilon h^{1-3\epsilon}$.  
From \eqref{coreL}, $|u(z)-w(z)| \leq Ch^{1+\tau'}$.
Since $w$ is symmetric with respect to $\left\{x_n=h^{1-3\epsilon}\right\}$, we have  $w(x)=w(z)$.
By \eqref{gest1}, we also have
	\begin{align*} 
		|u(x)-u(z)| & \leq \|Du\|_{L^\infty(D_h)} |x-z|\\
		 & \leq C_\epsilon h^{(\frac12-\epsilon)(1-\epsilon)+(1-3\epsilon)}\leq C_\epsilon h^{\frac{3}{2}-5\epsilon}
	 \end{align*}
for $\epsilon>0$ small.
Therefore, for the given constant $\tau<1/2$, when $\epsilon>0$ is sufficiently small,
	\begin{equation*}
		|u(x)-w(x)| \leq |u(x)-u(z)| + |u(z)-w(z)| \leq Ch^{1+\tau'}.
	\end{equation*}
Combining with \eqref{coreL} we thus obtain the desired $L^\infty$ estimate \eqref{coreLL}. 
\end{proof}

With Lemma \ref{diff1}, 
we can use the perturbation argument \cite{JW} to prove that $u\in C^{2,\alpha}({B_{\delta_0}\cap \overline{U}})$.
See also \cite[Theorems 5.1 and 5.3]{C96}, \cite[\S6]{CLW1}. 
Consequently by \eqref{normalformular}, we obtain $\mathcal{F}$ is $C^{2,\alpha}.$ 
For the reader's convenience, we outline the proof here.

Without loss of generality, assume
 $f(0)=g(0)=1$. 
 By \eqref{dh1}, the $C^\alpha$ regularity of $f,g$, and the $C^{1,\alpha'}$ regularity of $u,$ 
 we have  
	\begin{equation}\label{5002}
		\omega_f(h) := \sup_{x\in D_h}\Big|\frac{f(x)}{g(Du(x))}-1\Big| \leq Ch^\tau
	\end{equation}
for some $\tau \in (0,\frac{1}{2}).$
To proceed further, let us first quote a lemma from \cite{JW}.

\begin{lemma}\cite[Lemma 2.2]{JW}\label{pl3}
Let $u_i$, $i=1,2$, be two convex solutions of $\det\, D^2u=1$ in $B_1(0)$.
Suppose $\|u_i\|_{C^4}\leq C_0$.
Then if $|u_1-u_2|\leq\delta_1$ in $B_1(0)$ for some constant $\delta_1>0$, we have, for $1\leq k\leq 3$,
	\[ |D^k(u_1-u_2)| \leq C\delta_1 \quad\mbox{in } B_{1/2}(0). \]
\end{lemma}

%To {\color{blue}pursue (?)} the perturbation argument, we need to apply the following normalisation transform. 
Let $D_h, w$ be as in \eqref{domDh}, \eqref{eq:MA v}.
Given any $h>0,$ let $A$ be a unimodular affine transformation such that $\hat D_h:=h^{-\frac{1}{2}}A(D_h)$ has a good shape 
%{\Small\color{blue}(``good shape'' is vague, affine transform for ``good shape'' is not unique)}, 
in the sense that 
\begin{equation}\label{gsha}
	B_r(z) \subset \hat D_h \subset C_nB_r(z)
\end{equation}
for some $r>0$ and some point $z\in \hat D_h$,  where $C_n$ is a constant depending only on $n$. 
%{\Small\color{red} I deleted one sentence}
%From \eqref{dh1}, we have $h^\epsilon \lesssim r \lesssim h^{-\epsilon}$ for any $\epsilon>0$ small.  

We claim that $r\approx 1$.
Indeed,  let $\bar{w}(x):=\frac{1}{h}w(h^{\frac{1}{2}}A^{-1}x).$ Then, $\bar w$ is a convex solution of 
\begin{equation}\label{eqbarw}
\begin{cases}
\det\,D^2 \bar w=1&\mbox{in $\hat D_h$},\\
\bar w=1 &\mbox{on $\partial \hat D_h$}.
 \end{cases}
 \end{equation}
 By Lemma \ref{diff1} and since $0\leq u\leq h$ in $D_{h}^+,$
 we have
  \begin{align}\label{west1}
  & -Ch^{1+\tau'}\leq w \leq h \quad \text{in  }D_{h}^+, \nonumber \\
  & w(0)\leq u(0) + Ch^{1+\tau'} = Ch^{1+\tau'}. 
  \end{align}
By the symmetry of $w$, \eqref{west1} also holds in $D_h.$
Hence, 
\begin{align}\label{west2}
    -Ch^{\tau'} &\leq \bar{w}\leq 1\quad  \text{ in }  \hat D_h, \nonumber \\ 
    \bar w(0) & \leq Ch^{\tau'}.
\end{align}
From \eqref{eqbarw}, by Alexandrov's estimate \cite{C96, Fig3} we have
$$ |1 - \inf \bar w|^n \approx |\hat D_h|. $$
By \eqref{west2} it follows that $ |\hat D_h|\approx 1$ for $h$ small.  
Hence by \eqref{gsha}, we obtain $r\approx 1$.
By \eqref{west2}, we have that $|\bar w(0)-1|\approx 1.$ Hence by the Alexandrov maximum 
principle \cite[Theorem 2.8]{Fig3}, we have that 
$$\text{dist}(0, \partial\hat D_h)\geq c\frac{|\bar w(0)-1|^n}{\text{diam}(\hat D_h)^{n-1}\mu_{\bar{w}-1}(\hat D_h)}$$
for some constant $c$ depending only on $n$, where $\mu_{\bar w-1}$ is the Monge-Amp\`ere measure defined in \eqref{MAmeas}. 
Note that by \eqref{eqbarw} we have $\mu_{\bar{w}-1}(\hat D_h)\approx |\hat D_h|\approx 1.$
Hence $\text{dist}(0, \partial\hat D_h)\geq \frac{1}{C_n}$ for some constant $C_n$ depending only on $n.$
Therefore,
\begin{equation}\label{nnmm1}
 B_{1/C_n}(0)\subset h^{-\frac{1}{2}}A(D_h)\subset B_{C_n}(0) .
 \end{equation}
In particular, it implies that 
\beq\label{Dhbala}
\text{the set $D_h$ is balanced around $0$ for $h$ small.}
\eeq

%From \eqref{nnmm1}, we see that the transform $h^{-\frac12}AD_h$ has a good shape.
Next, we \emph{claim} that $h^{-\frac12}A(D_{h/4})$ also has a good shape. 
%simultaneously normalises $D_{h/4}$ as well. 
In fact, as in \eqref{domDh}, we can similarly define $D_{h/4}$ that is symmetric with respect to $\left\{x_n=(\frac{h}{4})^{1-3\epsilon}\right\}$. 
Note that $D_{h/4}$ may not be a subset of $D_h$, see Fig. \ref{fig4h}.

%{\Small\color{blue} (If you mean $D_{h/4}$ has a good shape, 
%then I think it is automatically true for a different shape constant.
%This is because that both $D_{h/4}$ and $D_h$ are convex and $|D_{h/4}|\ge c|D_h|$.) \color{red} we need the shape constant is independent of $h$, but here $D_{h/4}$ may not be a subset of $D_h$.}

\renewcommand{\figurename}{Fig.}
\renewcommand{\captionlabeldelim}{}
\begin{figure}[h]
	\centering
	\includegraphics[width=0.8\textwidth]{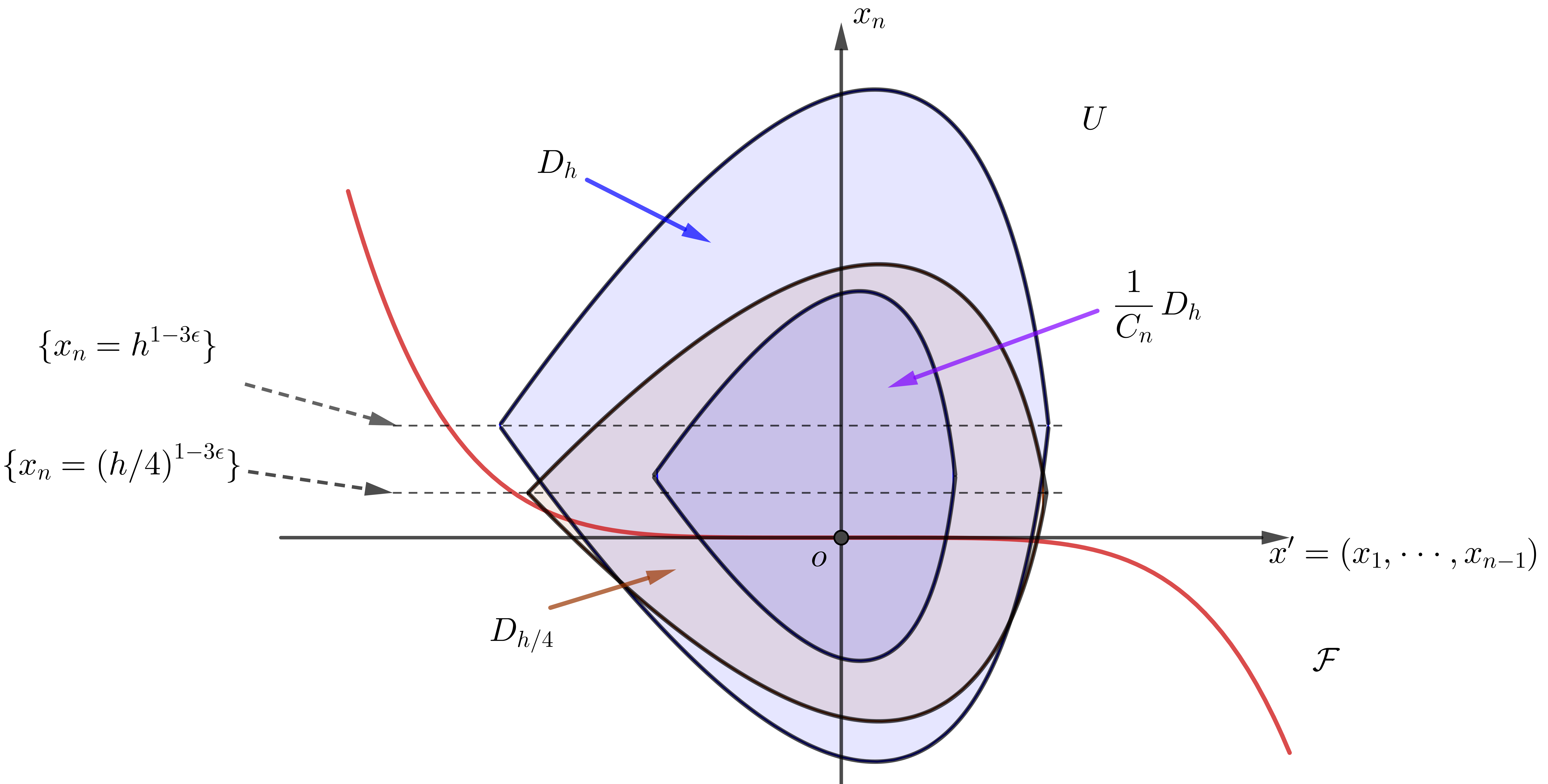}
	\caption{}
    \label{fig4h}
\end{figure}

By \eqref{dh1}, the width of $D_{h/4}$ in $e_n$ direction is greater than
 $C_\epsilon h^{\frac{1}{2}+\epsilon} \gg h^{1-3\epsilon}$ for $h$ small. Then, by convexity and symmetry,
 we have $|D_{h/4}\cap \{x_n\geq h^{1-3\epsilon}\}|\approx |D_{h/4}|\approx h^{n/2}.$ 
Hence  
\begin{equation}\label{dhprop1}
\left| h^{-\frac{1}{2}}A\left(D_{h/4}\cap \{x_n\geq h^{1-3\epsilon}\}\right) \right| \approx \left| h^{-\frac{1}{2}}A(D_{h/4})\right|\approx 1.
\end{equation}
Note that the set $h^{-\frac{1}{2}}A\left(D_{h/4}\cap \{x_n\geq h^{1-3\epsilon}\}\right)$ is uniformly bounded, since from \eqref{nnmm1}
\begin{equation}\label{dhprop2}
h^{-\frac{1}{2}}A(D_{h/4}\cap \{x_n\geq h^{1-3\epsilon}\})\subset h^{-\frac{1}{2}}A(D_h) \subset B_{C_n}(0).
\end{equation}
Hence, due to \eqref{dhprop1} the set $h^{-\frac{1}{2}}A\left(D_{h/4}\cap \{x_n\geq h^{1-3\epsilon}\}\right)$ also includes a ball inside, that is
\begin{equation}\label{dhprop3}
B_{\frac{1}{C_1}}(p)\subset h^{-\frac{1}{2}}A(D_{h/4}\cap \{x_n\geq h^{1-3\epsilon}\})\subset h^{-\frac{1}{2}}A(D_{h/4}) 
\end{equation}
for some point $p$, where the constant $C_1$ depends only on $n.$ 
By \eqref{dhprop1} and \eqref{dhprop3}, we have 
\begin{equation}\label{dhprop4}
h^{-\frac{1}{2}}A(D_{h/4})\subset B_{C_2}(0)
\end{equation}
 for some constant $C_2$ depending only on $n.$  
Finally, since $D_{h/4}$ is balanced around $0$ by \eqref{Dhbala}, from \eqref{dhprop1} and \eqref{dhprop4}
%by symmetry and the fact that  $C_\epsilon h^{\frac{1}{2}+\epsilon} \gg h^{1-3\epsilon}$ for $h$ small, 
we see that 
 $h^{-\frac12}A(D_{h/4})$ has a good shape,
namely 
\begin{equation}\label{nnmm2}
B_{1/C_3}(0)\subset h^{-\frac{1}{2}}A(D_{h/4}) \subset B_{C_3}(0)
\end{equation}
 for some constant $C_3$ depending only on $n.$

\begin{remark}\label{tau'}
Note that by \eqref{nnmm1} we have $|D_h|\approx h^{\frac{n}{2}},$ and then can improve the estimate
\eqref{west199} to $|w|\leq Ch.$ Hence, by examining the proof of Lemma \ref{diff1}, we can replace $\tau'$ by $\tau$
in the estimate \eqref{coreLL}.
\end{remark}

\begin{proof}[Proof of Theorem \ref{main1}]
%For any $\alpha'<\alpha$, let $h_0>0$ small enough such that \eqref{5002} is satisfied for $\tau=\frac{\alpha'}{2}$. 
%{\Small\color{blue} (there is no need to introduce $\alpha'$, you have already assumed that $\tau\in (0, \frac \alpha 2)$)}
Denote $h_k=4^{-k}h_0.$
Let $u_k$, $k=0,1,\cdots$, be the convex solution of
	\begin{align}
		\det\,D^2 u_k = 1 &\qquad\mbox{in }D_{h_k},\\
		u_k = h_k& \qquad\mbox{on }\partial D_{h_k}. \nonumber
	\end{align}
By rescaling back \eqref{nnmm1} and \eqref{nnmm2}, we see that
  $D_{h_k}$ is comparable to $D_{h_{k+1}},$ that is 
  $\frac{1}{C_n}D_{h_k}\subset D_{h_{k+1}}$ 
  for some constant $C_n$ depending only on $n$, 
  (see Fig. \ref{fig4h}). 
 % {\Small\color{blue} (Why do you use $M$ instead of $C_n$? you have used $M$ for a large constant)}
  
 Let $x=(x', x_n)\in \frac{1}{C_n}D_{h_k}$. 
 If $x_n\geq h_{k+1}^{1-3\epsilon},$ by \eqref{freeregu1} we have $x\in D_{h_k}\cap D_{h_{k+1}}\cap U.$ Then, by Lemma \ref{diff1}
  we obtain 
  \begin{equation}\label{gest1111}
  \begin{split}
  	|u_k(x)-u_{k+1}(x)| &\leq |u_k(x)-u(x)|+|u(x)-u_{k+1}(x)| \\
		&\leq C_1h_k^{1+\tau}+C_1h_{k+1}^{1+\tau}\leq Ch_k^{1+\tau}.
\end{split}
\end{equation}
  If $x=(x', x_n)\in \frac{1}{C_n}D_{h_k}$ with $x_n< h_{k+1}^{1-3\epsilon},$ by symmetry we have $\bar x := (x', 2h_k^{1-3\epsilon}-x_n)\in D^+_{h_k}$ and $\tilde x :=(x', 2h_{k+1}^{1-3\epsilon}-x_n)\in D^+_{h_{k+1}}$. 
 Since $u_k$, $u_{k+1}$ are symmetric with respect $\{x_n=h_k^{1-3\epsilon}\}$, $\{x_n=h_{k+1}^{1-3\epsilon}\}$, respectively, we have 
  \begin{align*}
  	|u_k(x)-u_{k+1}(x)| &= |u_k(\bar x) - u_{k+1}(\tilde x)| \\
		&\leq |u_k(\bar x) - u(\bar x)| + |u(\bar x) - u(\tilde x)| + |u(\tilde x) - u_{k+1}(\tilde x)|.
  \end{align*}
 From Lemma \ref{diff1}, $|u_k(\bar x) - u(\bar x)| \leq C_1h_k^{1+\tau}$ and $|u(\tilde x) - u_{k+1}(\tilde x)| \leq C_1h_{k+1}^{1+\tau} \leq C_1h_k^{1+\tau}$.
 To estimate the term $|u(\bar x) - u(\tilde x)|$, note that by \eqref{diamlevel} and Corollary \ref{alpha} we have 
 \begin{equation*}
   \|Du\|_{L^\infty(S_{h_k}[u])}\leq C_\epsilon h_k^{(\frac{1}{2}-\epsilon)(1-\epsilon)}\leq C_\epsilon h_k^{\frac{1}{2}-2\epsilon}.
 \end{equation*}
 Since $\bar x, \tilde x \in S_{h_k}[u]$ and $\tau<\frac12$, we thus obtain 
\begin{equation*}
\begin{split}
	|u(\bar x) - u(\tilde x)| &\leq \|Du\|_{L^\infty(S_{h_k}[u])}|\bar x-\tilde x|  \\
		& \leq  C_\epsilon h_k^{\frac{1}{2}-2\epsilon}|2h_k^{1-3\epsilon}-2h_{k+1}^{1-3\epsilon}| \\
		& \leq C_\epsilon h_k^{\frac{1}{2}-2\epsilon+1-3\epsilon} \leq Ch_k^{1+\tau}
\end{split}
\end{equation*}
   for some constant $C$ independent of $k$,  provided $\epsilon$ is small enough.
  Therefore, $|u_k(x)-u_{k+1}(x)|\leq Ch_k^{1+\tau}$.  Together with \eqref{gest1111}, we then conclude that 
\begin{equation}\label{kLinfy}
  	\|u_k-u_{k+1}\|_{L^\infty(\frac{1}{C_n}D_{h_k})}\leq C h_k^{1+\tau}
\end{equation} 
for some constant $C$ independent of $k.$

Let $A$ be the affine transformation such that $|\det\,A|=1$ and $\hat{D}_k:=h_k^{-\frac{1}{2}}A(D_{h_k})$ is normalised, namely $B_{\frac{1}{C}}(0)\subset \hat D_k \subset B_C(0)$ for some constant $C$ depending only on $n.$ 
Define   
   $$\bar{u}_k(x):=\frac{1}{h_k}u_k(h_k^{\frac{1}{2}}A^{-1}x),\quad\text{and}\quad \bar{u}_{k+1}(x):=\frac{1}{h_k}u_{k+1}(h_k^{\frac{1}{2}}A^{-1}x).$$ 
By \eqref{kLinfy}, we have
\begin{equation}\label{nmkLinfy}
	\|\bar{u}_k-\bar{u}_{k+1}\|_{L^\infty(\frac{1}{C_n}\hat{D}_{k})}\leq C h_k^{\tau}.
\end{equation}
Note that from \eqref{nnmm2},  $\hat{D}_{k+1}:=h_{k}^{-\frac{1}{2}}A(D_{h_{k+1}})$ is also normalised, thus both $\bar u_k$ and $\bar u_{k+1}$ have interior regularity \cite[Section 17.6]{GT}. Hence, by Lemma \ref{pl3}, we have 
$$ |D^2\bar u_k - D^2\bar u_{k+1}| \leq C h_k^{\tau} \quad\text{ in } \frac{1}{2C_n}\hat D_k. $$
Rescaling back and noticing that $\|A\|, \|A^{-1}\|\leq C_{\epsilon}h_k^{-\epsilon}$ due to \eqref{dh1}, we obtain  
    \begin{equation}\label{in2}
		|D^2u_k - D^2u_{k+1}| \leq C_\epsilon h_k^{\tau-2\epsilon} \quad\text{ in } \frac{1}{2C_n}D_{h_k},
	\end{equation}
and particularly $$\|D^2u_{k}(0)\|\leq \|D^2u_0(0)\|+\sum_{i=0}^{k-1}\|D^2u_{i+1}(0)-D^2u_{i}(0)\|\leq C+\sum_{i=0}^{k-1}C_\epsilon h_i^{\tau-2\epsilon}\leq C_2,$$
provided we choose $\epsilon$ sufficiently small, where $C_2$ is a universal constant independent of $k$. 
Since $\det\, D^2u_k=1,$ we also have $D^2u_k(0)\geq C_3I$ for some constant $C_3$ independent of $k.$

Now we \emph{claim} that 
\begin{equation}\label{goodshape111}
B_{\frac{1}{C_4}h_k^{\frac{1}{2}}}(0)\subset D_{h_k}\subset B_{C_4h_k^{\frac{1}{2}}}(0) \quad \forall\, k=1,2,\cdots
\end{equation}
  for some constant $C_4$ independent of $k$. 
Suppose the claim fails. Then the above affine transformation $A^{-1}$ must have a large norm. 
On the one hand, by Pogorelov estimate (see \cite[Section 17.6] {GT} or \cite[Theorem 3.10]{Fig3}), we have $\|D^2\bar{u}_k(0)\|\leq CI$ for some constant $C$ depending only on $n.$
 On the other hand $\|D^2\bar{u}_k(0)\|=\|(A^t)^{-1}D^2u_k(0)A^{-1}\|\geq C_3\|A^{-1}\|^2$ is very large, which is a contradiction.
 Hence \eqref{goodshape111} is proved.

 Since in \eqref{goodshape111} the constant $C_4$ is independent of $k,$ 
  we have that 
 \begin{equation}\label{goodshape222}
B_{\frac{1}{C_4}(4^{-1}h_k)^{\frac{1}{2}}}(0)\subset D_{4^{-1}h_k}\subset B_{C_4(4^{-1}h_k)^{\frac{1}{2}}}(0).
\end{equation}
 Denote by 
 $d_1:=\sqrt{(C_4^{-1}(4^{-1}h_k)^{\frac{1}{2}})^2-h_k^{2-6\epsilon}}.$ By a direct computation we have 
 that $C_4^{-1}(2^{-2}h_k)^{\frac{1}{2}}\leq d_1 \leq C_4^{-1}(2^{-1}h_k)^{\frac{1}{2}},$ provided $\epsilon$ is small and $k$ is large.
First, by the definition of $D_{h}$ we have that 
$$B_{d_1}(0)\cap U\cap \{x_n> h_k^{1-3\epsilon}\}\subset S_{4^{-1}h_k} \subset S_{h_k}.$$  
Then, for any $x=(x', x_n) \in B_{d_1}(0)\cap U\cap \{x_n\leq h_k^{1-3\epsilon}\},$ 
since $\mathcal{F}$ is $C^{1, 1-\epsilon},$ we have that
 $$h_k^{1-3\epsilon}\geq x_n\geq -C_{\epsilon}|x'|^{2-\epsilon}\geq 
 -C_{\epsilon}d_1^2\geq -C_\epsilon (C_4^{-2}2^{-1}h_k)^{\frac{2-\epsilon}{2}}.$$
 Hence $|x_n|\leq |h_k^{1-3\epsilon}|$ provided $k$ is large and $\epsilon$ is chosen small initially.
 Note that $(x', h_k^{1-3\epsilon})\in S_{4^{-1}h_k}.$
 Recall that by \eqref{gest1} we have that for any $x\in B_{d_1}(0)$ we have that $|Du(x)|\leq C_\epsilon |x|^{1-\epsilon}.$
 Now, 
 \begin{eqnarray*}
 u(x)&\leq& u(x', h_k^{1-3\epsilon})+C_\epsilon |d_1|^{1-\epsilon}(h_k^{1-3\epsilon}-x_n)\\
 &\leq& 4^{-1}h_k+2+2C_\epsilon (C_4^{-1}(2^{-1}h_k)^{\frac{1}{2}})^{1-\epsilon}h_k^{1-3\epsilon}\\
 &\leq& \frac{1}{2}h_k
 \end{eqnarray*}
 provided $\epsilon$ is small and $k$ is large. Hence $B_{d_1}(0)\cap U\subset S_{h_k}[u]$ for $k$ large.
 
 Let $z=(0, z_n)$ be the intersection of $\{te_n: t\geq 0\}$ and $\partial S_{h_k}[u],$ by \eqref{goodshape111} we have
 that $\frac{1}{C_4}h_k^{\frac{1}{2}}\leq z_n\leq C_4 h_k^{\frac{1}{2}}.$
 For any $x=(x', x_n) \in S_{h_k}[u]\cap \{x_n< h_k^{1-3\epsilon}\},$  by \eqref{diamlevel} we have that
 $|x'|\leq C_\epsilon h_k^{\frac{1}{2}-\epsilon}.$ Then, by the $C^{1,1-\epsilon}$ regularity of $\mathcal{F}$ we have that
 $x_n\geq -C_\epsilon h^{(\frac{1}{2}-\epsilon)(2-\epsilon)}.$ Hence 
 \begin{equation}\label{xn1}
 |x_n|\leq  C_\epsilon h_k^{1-3\epsilon}.
 \end{equation}
 Let $y=(y', h_k^{1-3\epsilon})$ be the intersection of the segment $xz$ and the hyperplane $\{x_n=h_k^{1-3\epsilon}\}.$ 
 By convexity of $u$ we have that $u(y)<h_k.$ Observe that $|y'|^2\leq C^2_\epsilon h_k^{1-2\epsilon}<h_k^{1-3\epsilon} $ provided $k$ is large. Hence $y\in D_{h_k},$ and by \eqref{goodshape111} we have that $|y'|\leq C_4h_k^{\frac{1}{2}}.$
 Now, 
 \begin{equation}\label{xxest}
 |x'|=\frac{|z_n-x_n|}{|z_n-y_n|}|y'|\leq C_4h_k^{\frac{1}{2}}
 \frac{C_4h_k^{\frac{1}{2}}+h_k^{1-3\epsilon}}{C^{-1}_4h_k^{\frac{1}{2}}-h_k^{1-3\epsilon}}\leq C_5h_k^{\frac{1}{2}},
 \end{equation}
 provided $k$ is large, for some constant $C_5$ depending only on $n.$
 By \eqref{xn1} and \eqref{xxest} we have that $S_{h_{k}}[u]\subset  B_{2C_4h_k^{{1}/{2}}}(0).$

From the above discussion, one has 
	$$B_{N^{-1}h_k^{{1}/{2}}}(0) \cap U \subset S_{h_{k}}[u]\subset  B_{Nh_k^{{1}/{2}}}(0)$$ 
for a constant $N$ independent of $k$, which implies that $u$ is $C^{1,1}$ at $0.$
Once having $u$ is $C^{1,1}$ at $0$, we deduce that $\epsilon=0$ in \eqref{dh1}, and since $f, g$ are $C^{\alpha}$ near $0,$ we can choose $\tau=\frac{\alpha}{2}$ in \eqref{5002}. 
Define 
	$$P_k(x):=u_k(0)+Du_k(0)\cdot x+\frac{1}{2}D^2u_k(0)x\cdot x.$$  
Let $r_k:=\frac14\min\{\frac{1}{C_4}(h_k)^{1/2}, \frac{1}{N}(h_k)^{1/2}\}$, where $C_4$ is in \eqref{goodshape111}, and $\hat B_k:=B_{r_k}(0)$. 
By applying Lemma \ref{pl3} to $\bar{u}_i, \bar{u}_{i+1}$ and then rescaling back, we have
\begin{align*}
\|D^3u_k\|_{L^\infty(\hat B_k)}&\leq \|D^3u_0\|_{L^\infty(\hat B_k)}+\sum_{i=0}^{k-1}\|D^3u_{i+1}-D^3u_i\|_{L^\infty(\hat B_k)}\\
&\leq C(1+\sum_{i=0}^{k-1}h_j^{\tau-\frac{1}{2}}) \leq Ch_k^{\tau-\frac{1}{2}}.
\end{align*}
Hence,
$$\|u_k-P_k\|_{L^\infty(\hat B_k)}\leq C\|D^3u_k\|_{L^\infty(\hat B_k)}h_k^{\frac{3}{2}}\leq Ch_k^{1+\tau}.$$
Therefore, by Lemma \ref{diff1} again, as $\tau=\frac{\alpha}{2}$, we have
\begin{equation}\label{approx1}
\begin{split}
|u(x)-P_k(x)| &\leq |u(x)-u_k(x)|+|u_k(x)-P_k(x)| \\
	&\leq C_1 h_k^{1+\tau} +  Ch_k^{1+\tau} \leq Cr_k^{2+\alpha}
\end{split}
\end{equation}
 for $x\in \hat B_k\cap U=B_{r_k}(0)\cap U.$ 
%Since $\mathcal{F}$ is $C^{1,\beta},$ we may find a ball $B_{\frac{1}{K}2^{-k}}(z_k)\subset B_{\frac{1}{4N}h_0^{\frac{1}{2}}2^{-k}}(0)\cap U$ for some point $z_k$ with $|z_k|\leq C2^{-k},$
 %where $K$ is a constant depending only on $N, h_0.$ 
Then, by \eqref{approx1} we have 
\begin{equation}\label{poly1}
\|P_k-P_{k-1}\|_{L^\infty\left(\hat B_k\cap U\right)}\leq 2Cr_k^{2+\alpha}.
\end{equation}
Denote $a_k=u_k(0)$, $b_k=Du_k(0)$, $c_k=\frac{1}{2}D^2u_k(0).$ Then
 $P_k(x)=a_k+b_k\cdot x +c_kx\cdot x.$ 
By \eqref{poly1}, we obtain
\begin{equation}\label{polydecay}
 \|c_k-c_{k-1}\|\leq Cr_k^{\alpha},\ \ \|b_k-b_{k-1}\|\leq Cr_k^{1+\alpha},\ \text{ and }\  |a_k-a_{k-1}|\leq  Cr_k^{2+\alpha}.
 \end{equation}
 Recall that $h_k=h_04^{-k}$, so $r_k\approx h_0^{1/2}2^{-k}$. 
 %{\color{purple} better to write $r_k \approx h_0^{1/2}2^{-k}$, because $\approx$ is up to a universal constant. } 
 Hence, $a_k, b_k, c_k$ converge to some $a_\infty, b_\infty, c_\infty$, respectively. 
 Let $P(x)=a_\infty+b_\infty\cdot x +c_\infty x\cdot x$. 
By \eqref{approx1}, \eqref{poly1} and \eqref{polydecay}, we obtain that
$|u(x)-P(x)|\leq C|x|^{2+\alpha}$, when $x\in B_{r_0}(0)\cap U$ for a small constant $r_0>0.$
%Then by standard argument we can pass from this pointwise estimate to a full $C^{2,\alpha}$ estimate, see for instance \cite[Remark 7.7]{Sa}.
 %$P_k=a_k+b_k\cdot (x-z_k) +c_k(x-z_k)\cdot (x-z_k),$ 
\end{proof}

\begin{remark}
By using the  strategy in this paper and the techniques developed in \cite[Section 4.3]{CLW1}, in dimension two, the assumptions on domains in Theorem \ref{main1} can be relaxed. In fact, we only need to assume $\Omega, \Omega^*$ to be $C^{1,\alpha}$ and convex.
\end{remark}

%\begin{remark}
%\emph{
%The approach also works for the more general case when two convex domains have overlap considered by Figalli \cite{AFi2, AFi} and Indrei \cite{I}. In particular, the main result holds for the part of free boundary away from the closure of the common region.}
%\end{remark}

\begin{remark}\label{rmkhr}
 % Let $u, v, \mathcal{F}, \Omega,\Omega^*, U, V, \rho, \rho^*$ be as in the beginning of \S\ref{S5}. 
 Assume further that $\Omega, \Omega^*, f, g$ are smooth, then the higher regularity of $\mathcal{F}$ follows from the classical elliptic theory \cite{GT}. 
 For the reader's convenience, we give an outline of the argument.
 Let $x_0\in\mathcal{F}$ and $y_0=Du(x_0)$.
 By a change of coordinates, we can assume $y_0=0$ and locally near the origin
  $$ \partial V=\{(y',y_n) : y_n=\rho^*(y')\}\ \text{ for }\ y'=(y_1,\cdots,y_{n-1})  $$
 with a smooth, convex function $\rho^*$ satisfying $\rho^*(0)=0$ and $D\rho^*(0)=0$.
 Once having $u$ is $C^{2,\alpha}$ smooth up to $\mathcal{F},$ one has $v\in C^{2,\alpha}(\overline{V}\cap B_{r_1}(0))$ for some small constant $r_1>0$.
 Let $\eta(x)$ be the defining function of $\mathcal{F}$ such that $\eta\in C^{2,\alpha}(B_{r_0}(x_0))$ for a small $r_0>0$ satisfying $\eta(x)=0$ and $|D\eta(x)| \ne 0$ for $x\in B_{r_0}\cap\mathcal{F}.$ 
 Then the function $v$ satisfies 
 \begin{equation}\label{tolinear}
 \begin{split}
 	\det\, D^2v(y) &= \frac{g(y)}{f(Dv(y))} \quad \text{for } y\in B_{r_1}(0)\cap V, \\
	\eta(Dv(y)) &=0 \qquad\ \hskip25pt  \text{ for } y\in B_{r_1}(0)\cap\partial V.
\end{split}
 \end{equation} 
Make the following change of coordinates $y\to\tilde y$ to flatten the boundary $B_{r_1}(0)\cap \partial V$, 
	\begin{equation*}
		\tilde{y}' = y'; \qquad	 \tilde{y}_n = y_n-\rho^*(y') 
	\end{equation*} 
and let $\hat v(\tilde{y})=v(y).$ 
By differentiating \eqref{tolinear} in the $\tilde{y}_k$-variable for $k=1, 2,\cdots, n-1$, we can see that function $\hat{w}=\partial_{\tilde{y}_k}\hat v$ satisfies a linear uniformly elliptic equation with an oblique boundary condition
 \begin{equation}\label{linearised}
 \begin{split}
 	\mathcal{L}[\hat w] = a^{ij}D_{ij}\hat w+b^iD_i\hat w - \tilde f = 0  \quad &\text{ in } B_{r_1}(0)\cap \{\tilde y_n>0\}, \\
	\beta \cdot D\hat w = \tilde g \qquad\ \hskip25pt  &\text{ on } B_{r_1}(0)\cap \{\tilde y_n=0\},
\end{split}
 \end{equation} 
where the coefficients $a^{ij}\in C^\alpha$, $b^i\in C^{1,\alpha}$, the functions $\tilde f\in C^\alpha$, $\tilde g\in C^{1,\alpha}$, and $\beta$ is a $C^{1,\alpha}$ vector field on $B_{r_1}(0)\cap \{\tilde y_n=0\}$ satisfying 
$$ \beta(\tilde y) \cdot e_n > 0 \qquad \text{ for all $\tilde{y}\in \{\tilde{y}_n=0\}$ near $0.$ }$$
Then, one can apply \cite[Section 6.7]{GT} to conclude that $\hat w=\partial_{\tilde{y}_k}\hat v\in C^{2,\alpha}(B_{\frac{1}{2}r_1}\cap \{\tilde{y}_n\geq 0\})$ for $k=1,\cdots, n-1.$
 By using the equation \eqref{linearised}, we also have $\partial_{\tilde{y}_n}\hat v\in C^{2,\alpha}(B_{\frac{1}{2}r_1}\cap  \{\tilde{y}_n\geq 0\}).$ 
 Hence, $\hat v\in C^{3,\alpha}(B_{\frac{1}{2}r_1}\cap \{\tilde{y}_n\geq 0\}),$ which implies 
 $$v\in C^{3,\alpha}(B_{\frac{1}{2}r_1}\cap \bar V).$$ Since $D^2u=(D^2v)^{-1},$ it implies that $u$ is $C^{3,\alpha}$ near $0.$ Hence $\mathcal{F}$ is $C^{3,\alpha}$ near $0,$ which implies that $\eta$ is $C^{3,\alpha}$ near $0.$ Finally, by differentiating the equation and boundary condition repeatedly, we can show that $\mathcal{F}$ is $C^{k,\alpha}$ for any $k\geq 1.$
  \end{remark}

\section{Blow-up analysis}\label{Sbu}

The purpose of this section and the next section is to prove the obliqueness property \eqref{obq-n00}. 
 In this section, we assume that $\overline\Om, \overline{\Om^*} \subset \R^n$ are disjoint, uniformly convex domains with $C^2$ boundaries.
The densities $f\in C(\bom), g\in C(\overline{\Omega^*})$, and there is a positive constant $\lambda$ such that 
$\lambda^{-1}<f, g<\lambda$ in $\Omega, \Omega^*$, respectively. 
%Then, we will  study the limit profile of a normalisation sequence assuming the obliqueness fails.

%If the obliqueness fails at some point $x_0\in\mathcal{F}$, it is crucial to understand the limit profile of a normalisation sequence around $x_0$.
%This is analogous to the technique of blow-up analysis introduced by Caffarelli \cite{C77} to the classical free boundary problem. 
Let $x_0\in\mathcal{F}$, $y_0=Du(x_0)\in \partial V\setminus \overline{\partial V\cap \Omega^*}$, and $\nu_{_U}(x_0), \nu_{_V}(y_0)$ be the 
unit inner normals of $U, V$, respectively.
By the convexity of $u$, it always holds that $\nu_{_U}(x_0) \cdot \nu_{_V} (y_0)\geq0.$ 
Suppose  \eqref{obq-n00} fails at $x_0$, then 
\beq\label{oblifail}
	\nu_{_U}(x_0) \cdot \nu_{_V} (y_0)=0.
\eeq
By a translation of coordinates, we may assume that $x_0$ is the origin. Then, by subtracting a  constant, we may assume $v(y_0)=0, v\geq 0.$ Hence $Dv(y_0)=0.$
 Denote  
\beq\label{Vhat} 
	\hat V = \left\{y-y_0 : y\in V \right\}. 
\eeq
	
The main result of this section is the following
\begin{proposition} \label{blowuppic}
Suppose \eqref{oblifail} occurs. 
Then, there exists a sequence of $h_k\rightarrow 0,$ and a sequence of affine transformations $A_{k}$ such that as $k\to\infty$,
$$v_k(y):=\frac{1}{h_k}v(A_k^{-1}(y+y_0))\quad\text{for } y\in \mathbb{R}^n$$ 
locally uniformly converges to a global convex function $v_0$.
Meanwhile, $A_k(\hat V)$ locally uniformly converges to a convex set $V_0$ as $k\to\infty$.  
There satisfies  
$$ \det\, D^2v_0=c_0\chi_{_{V_0}}\quad\text{ in } \R^n$$
for some constant $c_0>0$.

Let $U_0:=\text{interior of}\ Dv_0(\mathbb{R}^n)$.  Then, $U_0$ is a convex set.  
Under a proper coordinate system, we have the following limit profiles. 
\begin{itemize}
\item[($i$)] When $n=2,$  we have
	$$ V_0 =\big\{(y_1,y_2)\in\R^2 \,:\, y_1> \rho_0^*(y_2) \big\}, $$ 
	 where $\rho_0^*(t)=at^2$ for some constant $a>0$, and  
	$$ U_0 = \big\{(x_1,x_2)\in\R^2 \,:\, x_2> \rho_0(x_1) \big\}, $$
	 where $\rho_0$ is a convex function satisfying $0\leq \rho_0(t)\leq Ct^2$ for a constant $C>0$,
	and $\rho_0(t)=\frac{1}{2r}t^2$ for $t<0$,
	where $r>0$ is a constant. 
%for some positive constants $a, r$ depending on the domains $\Om, \Om^*$. 

\item[($ii$)] When $n\geq 3,$ we have 
	\begin{align*}
		V_0 &= \{y\in\R^n : y_1>\rho_0^*(y_n)\}, \\
		U_0 &= \{x\in\R^n : x_n>\rho_0(x_1)\},
	\end{align*}
where $\rho_0^*,\rho_0$ are two convex functions defined near $0$ satisfying $\rho_0^*(0)=\rho_0(0)=0$, $\rho_0^*\geq 0$, $\rho_0\geq 0.$ Moreover, $\rho_0^*$ is smooth and uniformly convex.
\end{itemize}

\end{proposition}

\begin{remark}
By the discussion below \eqref{maeq111} we can see that 
 $v_0$ is $C^1$ and strictly convex in $\overline{V}_0.$ 
\end{remark}

\subsection{Blow-up in dimension two}\label{S5.1}
Assume \eqref{oblifail} that the obliqueness fails at $0\in\mathcal{F}$ and $y_0=Du(0)\in \partial V$.
By a translation and a rotation of coordinates,  we may assume that the unit inner normals are 
$\nu_{_U}(0)=e_2$, $\nu_{_V}(y_0 )=e_1$ (see Fig. \ref{fig1}).
Then by \eqref{normalformular}, we have 
$$y_0 =re_2\quad \text{ for some } r\geq \text{dist}(\Omega, \Omega^*)>0.$$  
By $ii)$ of Theorem \ref{CMCL},
there is a function $\rho \in C^{1,\alpha'}$ satisfying $\rho(0)=\rho'(0)=0$ such that
%\\
%{\Small\color{blue}(I change $\beta$ to $\alpha$, because in Theorem 2.1 it is $\alpha$ 
%and in this section $\alpha$ is not used for any other occasion)}
	\begin{equation}\label{parU}
		\mathcal{F}=\{(x_1,x_2) \,:\, x_2=\rho(x_1)\}\quad\text{ near 0}.
	\end{equation}
Since $\partial V\cap\partial\Omega^*$ is $C^2$ smooth and uniformly convex near $y_0$,  we may assume 
	\begin{equation}\label{parV}
		\partial V=\{(y_1,y_2) \,:\, y_1=\rho^*(y_2-r)\}\quad\text{ near }y_0,
	\end{equation}
and $\rho^*(t)=at^2+o(t^2)$ for some constant $a>0$.

%
% \begin{notation}\label{n3.1}
%For simplicity, we will choose  $x_0$ as the origin.
%By the above choice of coordinates, we have $y_0=re_n$.
%%In the latter case, we have $y_0=0$ and $x_0=-re_n$.
%%We will assume that $x_0$ is the origin except otherwise indicated.
%\end{notation}

\begin{lemma} \label {up1}
$\rho(x_1) > 0$ for $x_1<0$ near the origin.
\end{lemma}

\begin{proof}
Suppose to the contrary that there exists a point $-se_1\in U$ for some $s>0$. 
Then $Du(-se_1)\in V$. 
By the expression \eqref{parV} (the strict convexity of $\pom^*$), we have 
	$$ \left(Du(-se_1)-y_0\right) \cdot e_1 > 0. $$
On the other hand, since $u$ is convex and $y_0=Du(0)$, we have
	$$(-se_1-0)\cdot \left(Du(-se_1)-y_0\right)\geq 0,$$ 
which is a contradiction.  
\end{proof}

\renewcommand{\figurename}{Fig.}
\renewcommand{\captionlabeldelim}{}
\begin{figure}[h]
	\includegraphics[width=0.45\textwidth]{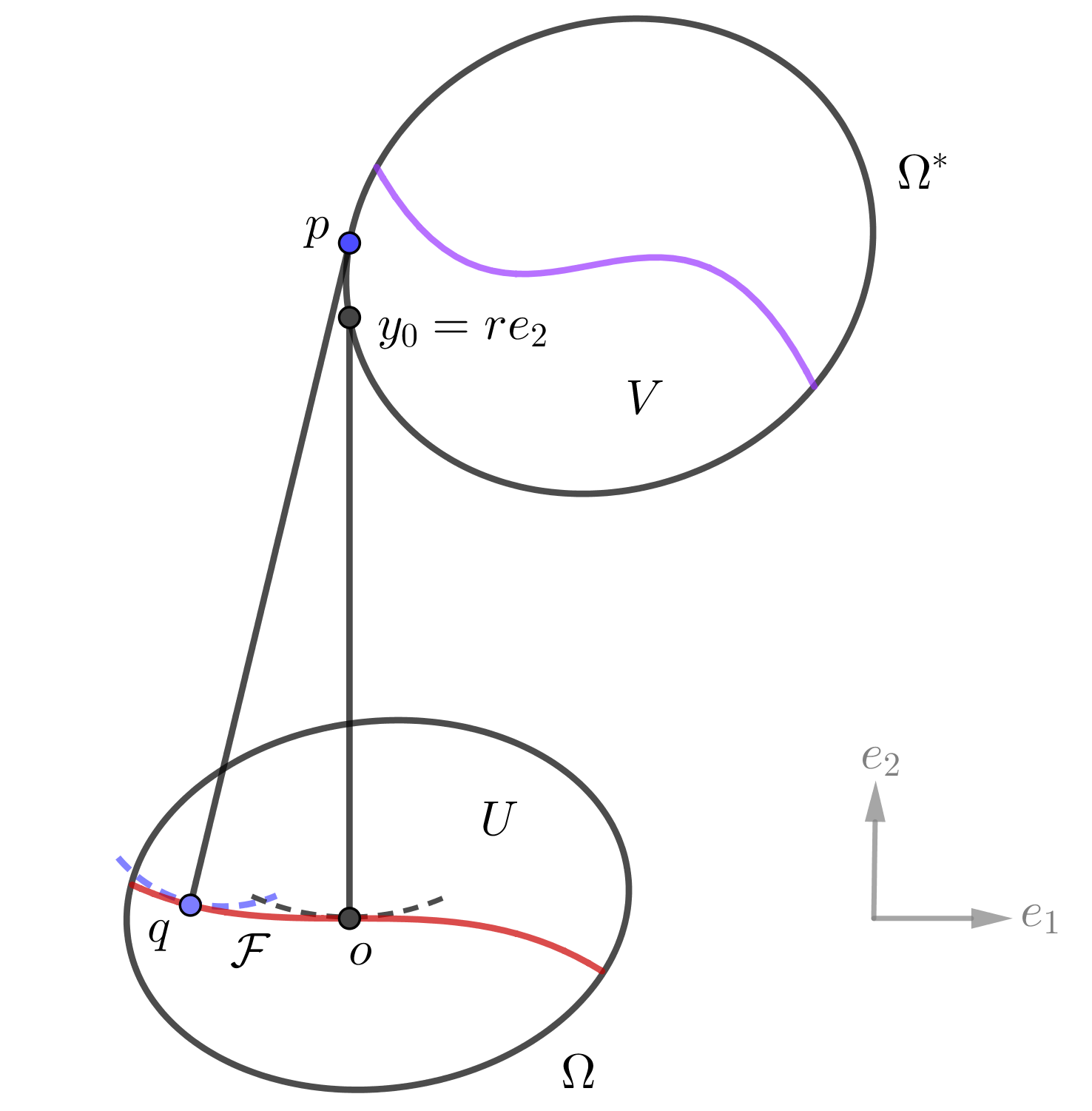}
	\caption{}
    \label {fig1} 
\end{figure}

The next lemma is a refinement of Lemma \ref{up1}. 
\begin{lemma}\label{keylemma}
$\rho(x_1)\leq Cx_1^2$ for $x_1$ close to $0$. Moreover,
$\rho(x_1)=\frac{1}{2r}x_1^2+o(x^2_1)$ for $x_1<0$ close to $0$.
\end{lemma}
\begin{proof}
First, by the interior ball property in Lemma \ref{intball}, $\mathcal{F}$ stays below the ball $B_r(y_0)$, 
which implies  that $\rho(x_1)\leq \frac{1}{2r}x_1^2+o(x^2_1)$ for $x_1$ close to $0.$    
Hence it suffices to prove $\rho(x_1)\geq \frac{1}{2r}x_1^2+o(x^2_1)$ for $x_1<0$ near the origin.

Consider a point $q=(q_1, \rho(q_1))\in\mathcal{F}$ for $q_1<0$ small.
Denote $p=Du(q)\in \partial \Om^*$.
By the interior ball property  again,  we have $B_{|p-q|}(p)\cap \Omega\subset U.$  
It implies $|p-q|\leq |p-0|,$ since otherwise $0$ would be an interior point of $U$ contradicting to the fact that $0\in \partial U.$
% {\color{blue} Indeed, if $|p-q|> |p-0|,$ then by the interior ball property we have $0\in B_{|p-q|}(p)\cap\Omega\subset U,$ namely, $0$ is an interior point of $U,$ which is a contradiction. }
Hence we have
\begin{equation}\label{ineq11}
\begin{split}
|p_2-\rho(q_1)|^2+(p_1+|q_1|)^2 &= |p-q|^2 \\
					           & \leq |p|^2=p_1^2+p_2^2.
\end{split}
\end{equation}
It follows that 
$$\rho(q_1)\geq \frac{1}{2p_2}q_1^2.$$ 
By the continuity of $Du$, we have $p_2\to r$ as $q_1\to 0$,
namely $p_2=r+o(1)$ as $q_1\to 0.$
Therefore,
 $$\rho(q_1)\geq \frac{1}{2r+o(1)}q_1^2\geq \frac{1}{2r}q_1^2+o(q_1^2).$$ 
\end{proof}

By our discussion in Section \ref{S2}, $v\in C^1(\mathbb{R}^2)$ and $Dv=\text{Id}$ in $\Omega\setminus U$.
Hence, as $0\in\mathcal{F}\subset\partial U$,
$$Dv(0)=0=Dv(y_0). $$  
By the convexity of $v$, we infer that
$$ Dv(te_2)=0\quad \forall\,t\in[0,r].$$ 
By subtracting a constant, we may assume that $v(y_0)=0$ and $v\geq 0$ on $\mathbb{R}^2$.
Then $v(te_2)=0$ for all $t\in[0,r]$ as well.  

Consider the point $p=(p_1, p_2)\in \partial \{v<h\}\cap \partial\Omega^*$ with $p_2<r$ (see Fig. \ref{f3.2}).
Since $0\in\{v<h\}$,  by the convexity of $\{v<h\}$ and $\Omega^*$,  the sub-level set
\begin{equation}\label{pin11}
S_h[v]=\{v<h\}\cap\Omega^*\ \text{is pinched between the rays}\ \overrightarrow{0y_0}\ \text{and}\ \overrightarrow{0p}.
\end{equation}
Denote $s:=r-p_2$.  From \eqref{parV}, $p_1=\rho^*(-s)=as^2+o(s^2)$.

\renewcommand{\figurename}{Fig.}
\renewcommand{\captionlabeldelim}{}
\begin{figure}[h]
	\centering
	\includegraphics[width=0.5\textwidth]{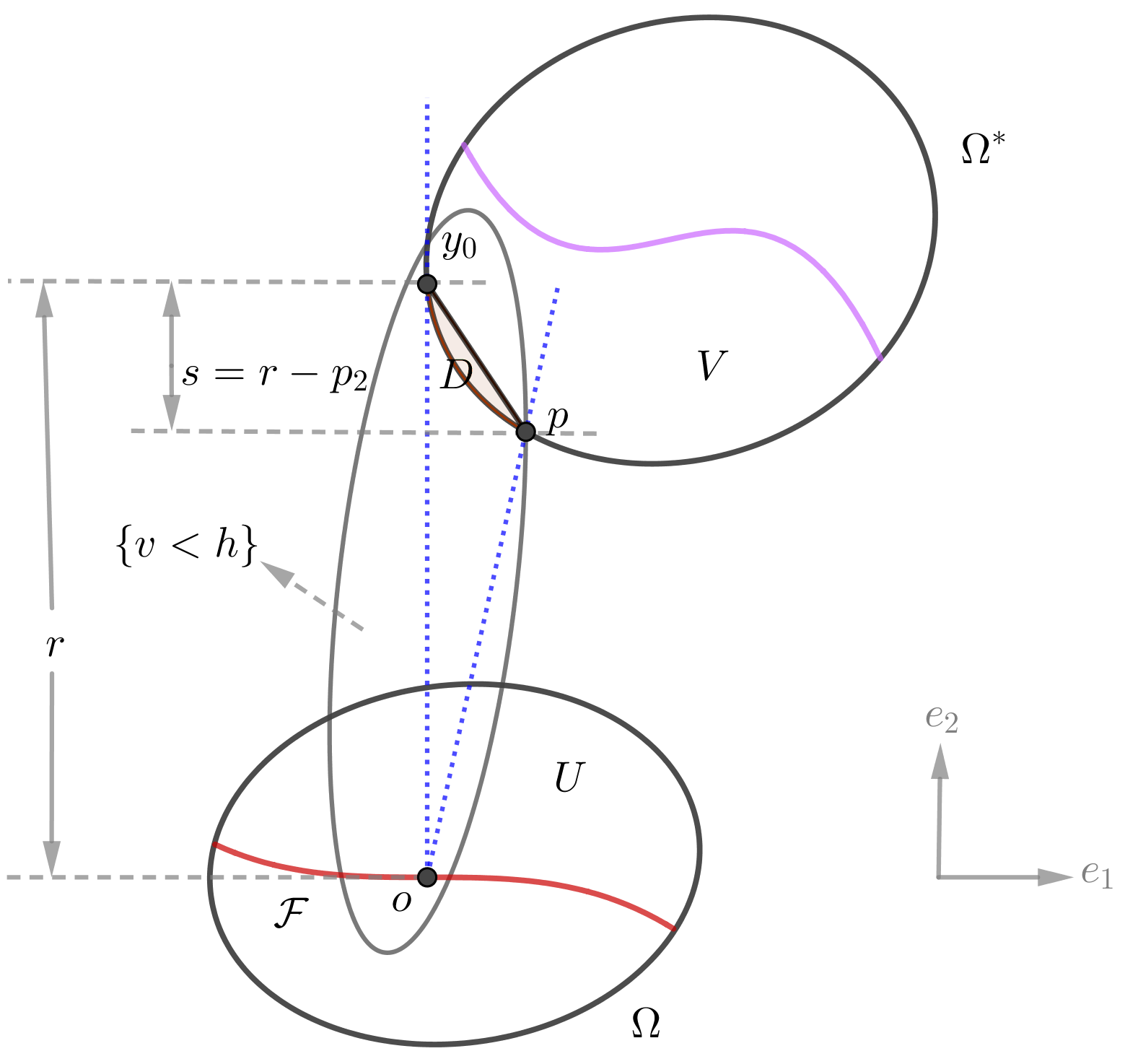}
	\caption{}
    \label{f3.2}
\end{figure}

\begin{lemma}\label{dlemma}
There exist positive constants $C_1, C_2$ depending on $\lambda$ and the domains $\Omega, \Omega^*$, but independent of $h$, such that 
	\begin{equation}\label{sh13}
		C_1h^{1/3}\leq s\leq C_2h^{1/3}.
	\end{equation}
\end{lemma}

\begin{proof}
Let $D\subset S_h[v]$ be the region enclosed by $\partial\Omega^*$ and the segment $\overline{y_0 p}$, (see Fig. \ref{f3.2}).
We have
\begin{equation}\label{volD}
\begin{split}
|D|&=\frac{1}{2}s\rho^*(-s)-\int_{0}^s\rho^*(-t)\,dt\\
&=\frac{1}{2}as^3+o(s^3)-\int_0^s(at^2+o(t^2))\,dt\\
&=\frac{1}{6}as^3+o(s^3).
\end{split}
\end{equation} 
By the volume estimate \eqref{nm1}, we also have $|D|\leq |S_h[v]|\approx h$.   
Hence, $s\leq C_2h^{1/3}$.   

For any given $y \in S_h[v]$,  by \eqref{pin11} we have $\frac{y_1}{y_2}\leq \frac{p_1}{p_2}.$
By the strict convexity of $v$,  we have $\text{diam}(S_h[v])\leq \frac{r}{3}$ for $h$ sufficiently small. 
Hence, $p_2\geq \frac{2}{3}r$ and $y_2\leq \frac{4}{3}r$,
thus we obtain
$$y_1\leq \frac{p_1}{p_2}y_2\leq Cs^2. $$    
From \eqref{parV} we also have 
$$y_1 \geq \rho^*(y_2-r) \geq \frac12a(y_2-r)^2.  $$  
Combining the above two inequalities, we obtain $|y_2-r|\leq Cs\ \forall\, y \in S_h[v]$. Hence
 $$S_h[v] \text{ is contained in the box } [0, Cs^2]\times [r-Cs, r+Cs].$$   
It follows that $h\approx |S_h[v]|\leq 2C^2s^3$,  which then implies $s\geq C_1h^{1/3}$.   
\end{proof}

%A crucial ingredient of the proof of Proposition \ref{oblique} is to estimate the shape of the centred sub-level set $S^c_h[v](y_0)$.
Thanks to Lemma \ref{dlemma}, we are able to give a precise description of the shape of the centred sub-level set $S^c_h[v](y_0)$ in the subsequent lemma. 

\begin{remark}\label{y0}
In order to simplify notations, we can translate $y_0$ to the origin by letting $\hat u(x)=u(x)-y_0\cdot x$ and $\hat v(y)=v(y+y_0).$
By subtracting a constant we may also assume
 $\hat u(0)=\hat v(0)=0$, and $D\hat u(0)=D\hat v(0)=0.$
Under the translation, $V$ becomes $\hat V$ defined by \eqref{Vhat} and 
\begin{align*}
 D\hat u(\mathbb{R}^2)=\hat \Omega^* = \left\{ y-y_0 : y \in \Omega^*\right\}.
\end{align*} 
By the properties $(i)$--$(iii)$ in \S\ref{s21}, it is also straightforward to check that 
$\hat{u}^*=\hat{v}$ in $\hat{V}$, and $\hat{u}^*$ is strictly convex in $\hat \Omega^*.$  
For simplicity, we  denote $\hat u, \hat v, \hat\Omega^*, \hat V$  by $u, v, \Omega^*, V$. 
We remark that the separation of $\Om^*$ and $\Om$ will not be used in the rest of this subsection.
\end{remark}
%{\Small\color{blue} (we can write $\hat u, \hat v$ as $u, v$ 
%only if in the rest of the section, we will not use the condition that $\Om^*$ and $\Om$ are separate anymore, 
%otherwise it will be a confusion for readers.
%If we will not use the separation of $\Om^*$ and $\Om$ in the rest of the section, please point out it clearly.)}

By Remark \ref{y0},
we may assume $y_0=0, v(0)=0$ and $Dv(0)=0.$
The following lemma characterises the shape of the centred sub-level set $S^c_h[v](0)$.

\begin{lemma}\label{normalv1}
There exists a positive constant $C$ independent of $h$ such that 
\begin{equation}\label{nm11}
B_{\frac{1}{C}}(0)\subset A_h\big( S^c_h[v] \big) \subset B_C (0),
\end{equation}
where  $A_h$ is a linear transform given by
\beq
A_h=\left(\begin{matrix}
 h^{-\frac{2}{3}} , &0  \\
0, & h^{-\frac{1}{3}}
 \end{matrix}\right).
 \eeq
\end{lemma}

\begin{proof}
Let $D$ be as in the proof of Lemma \ref{dlemma}. 
From \eqref{volD} and \eqref{sh13}, we have the volume estimate $|D|\approx h$.
Hence
	$$|A_h(D)|=\frac{1}{h}|D|\geq \frac{1}{C}$$ 
for some  $C>0$ independent of $h$.   
Since $D$ is contained in the rectangle $[0, Ch^{2/3}]\times [-Ch^{1/3}, 0]$,  
we see that $A_h(D)$ is bounded, and $A_h(D)\subset B_C(0)$ for a constant $C$ independent of $h$.   
Hence there exist a ball contained in $A_h(D)$, namely
	$$ B_{\frac{1}{C}}(q)\subset A_h(D) $$
for a point $q\in A_h(D)$ and a different constant $C$.
From the equivalence relation \eqref{equi0}, we thus conclude
	\begin{equation}\label{aee1}
		B_{\frac{1}{C}}(q)\subset A_h(D)\subset A_h(S_h[v]) \subset A_h(S_{bh}^c[v]),
	\end{equation}
where $b\geq1$ is a  constant independent of $h$. 

By the volume estimate \eqref{nm1}, we have
$|S_h[v]|\approx |S_h^c[v]|\approx h$,  hence 
\begin{equation}\label{aee2}
|A_h(S_{bh}^c[v])|\approx 1.
\end{equation}
By \eqref{aee1}, \eqref{aee2} and noting that $S_{bh}^c[v]$ is a convex set centred at $0$,   we obtain \eqref{nm11}.
\end{proof}

The proof of Lemma \ref{normalv1} also applies to the sub-level set $S_h[v]$.
%and \eqref{nm11} holds for $S_h[v]$ as well.
In fact, from \eqref{aee1}, $A_h(S_h[v])$ contains a ball $B_{1/C_1}(q)$. 
By John's lemma,
there exists an ellipsoid $E$ centered at $q',$ the center of mass of $A_h(S_h[v]),$ such that $E\subset A_h(S_h[v])\subset C(n)E.$ Let $r_1\leq r_2\leq\cdots r_n$ be the principal radii of $E.$
Similarly to \eqref{aee2},  we see that by \eqref{nm1}, the volume  $|A_h(S_h[v])|\approx 1,$ hence $r_1r_2\cdots r_n \approx 1.$
 Since $A_h(S_h[v])$ contains a ball $B_{1/C_1}(q),$ $r_1\geq \frac{1}{C_1C(n)}.$
Hence $r_n\lesssim \frac{1}{r_1^{n-1}}\leq C.$
Therefore, $A_h(S_h[v])$ also has a good shape, namely, $B_{1/C}(q')\subset A_h(S_h[v])\subset B_{C}(q'),$
for some positive constant $C$ independent of $h.$

\begin{proof}[Proof of Proposition \ref{blowuppic} when $n=2$]

Denote $V_h=A_h(V)$, $U_h=\frac{1}{h}A^{-1}_h(U)$.   
Locally near the origin, the boundary $\p U_h$ can be represented by
\begin{equation}\label{limitshape}
 \partial U_h =
     \Big\{(x_1,x_2)\in\R^2: \ x_2=\rho_h(x_1):=h^{-\frac23}\rho(h^{\frac13}x_1)\Big\}.
     \end{equation}
By Lemma \ref{keylemma}, we have $\rho_h(t)\leq Ct^2,$ and $\rho_h(t)=\frac{1}{2r}t^2+o(1)t^2$ for
$t<0.$

Similarly, by \eqref{parV}, the boundary $\p V_h$ can locally be represented by
\beq\label{pVh}
\partial V_h=\left\{(y_1,y_2)\in\R^2: \ y_1=\rho^*_h(y_2)=h^{-\frac23}\rho^*(h^{\frac13}y_2)=ay_2^2+o(1)y_2^2\right\},
\eeq
where $o(1)\rightarrow 0$ as $h\rightarrow 0$.

Denote 
\begin{equation}\label{buvh}
	v_h(y)=\frac{1}{h}v(A_h^{-1}y).
\end{equation}
We \emph{claim} that for $h> 0$ small, $v_h$ is locally uniformly bounded in  $\R^2$.
Note that by \eqref{localise} we have $S^c_h[v]\cap V=S^c_h[v]\cap \Omega^*$ is convex and $S^c_h[v]\cap \Omega=\emptyset$ for $h$ small.
Hence, by \eqref{Asolv} 
	$$C^{-1}\chi_{_{S^c_h[v]\cap V}}  \leq \det\, D^2 v \leq C \chi_{_{S^c_h[v]\cap V}}\quad\mbox{ in } S^c_h[v].$$
Therefore, the Monge-Amp\`ere measure $\mu_v$ is doubling for $S^c_h[v],$ when $h$ is small. Note also that, by the same reason the doubling property holds
for all centred sub-level sets $S^c_h[v](y)$ for $y\in \overline{V}$ close to the origin and $h$ small.
Then, for any $k>0$ large, 
by the geometric decay of sections (see \cite[Lemma 2.2]{C96} or \cite[Lemma 7.6]{CM}), there exists a constant $M_k$ such that 
	$$kS^c_h[v]\subset S^c_{M_kh}[v] \quad \text{ for $h>0$ small}. $$
On the other hand, by  \eqref{nm11} we have 
	$$B_{\frac{k}{C}}(0) \subset A_h\left(kS^c_h[v]\right) \subset A_h\left(S^c_{M_kh}[v]\right).$$ 
From \eqref{secrela}, we have $v\leq C_1M_kh$ in $S^c_{M_kh}[v]$ for a constant $C_1$ independent of $h$. 
Hence under the normalisation \eqref{buvh}, we obtain
	\begin{equation}\label{aac1}
	 0\leq v_h\leq C_1M_k \quad\text{in }B_{\frac{k}{C}}(0),
	 \end{equation}
where the constants $C, C_1$ are independent of $k, h$. 
As $k>0$ can be arbitrarily large, the claim is proved.
By \eqref{aac1}, \cite[Corollary A.23]{Fig3} and the convexity of $v_h,$ we have that
%{\color{blue} Since $v_h(0)=0$, by \eqref{aac1} and the convexity of $v_h$, we have }
\begin{equation}\label{aac2}
\|Dv_h\|_{L^\infty(B_{{k}/{2C}}(0))}\leq \frac{\|v_h\|_{L^\infty(B_{{k}/{C}}(0))}}{{k}/{2C}}\leq \frac{2CC_1M_k}{k}.
\end{equation}
Now, passing to a subsequence, by the above claim we may assume that 
$v_h$ converges to $v_0$ locally uniformly.
By the expression \eqref{pVh}, we may also assume that
$V_h, \rho^*_h$ converge to $V_0, \rho^*_0$ locally uniformly, 
and 
$$V_0:=\big\{y\in\R^2 \,:\, y_1> \rho_0^*(y_2)=ay_2^2 \big\} . $$
Moreover, 
\begin{equation}\label{maeq1}
\det\, D^2v_0=c_0\chi_{_{V_0}} \quad \text{ in }\ \mathbb{R}^2
\end{equation}
for some constant $c_0>0$. 

%{\color{red} Indeed, suppose $v_{h_i}$ converges to $v_0$ locally uniformly as $i \rightarrow \infty,$ where 
%$h_i\rightarrow 0,$ as $i\rightarrow \infty.$ 
%Fix any $k$ large, we have that $v_{h_i}$ satisfies $\det D^2v_{h_i}(x)=f_i(x) \chi_{_{B_{\frac{k}{C}}(0)\cap V_{h_i}}}$ in $B_{\frac{k}{C}}(0),$
%where $f_i(x):=\frac{g(A_{h_i}^{-1}x)}{f(Dv(A_{h_i}^{-1}x))}.$ Since $V_{h_i}$ converges to $V_0$ locally uniformly and $f, g$ are continuous near $0,$
% we have that $f_i(x) \chi_{_{B_{\frac{k}{C}}(0)\cap V_{h_i}}}$ converges to $\frac{f(0)}{g(0)}\chi_{_{B_{\frac{k}{C}}(0)\cap V_{0}}}.$
% By the weak continuity of Monge-Ampere measures, we have that $\det D^2v_{0}(x)=c_0 \chi_{_{B_{\frac{k}{C}}(0)\cap V_{0}}}%$ in $B_{\frac{k}{C}}(0),$ where $c_0=\frac{f(0)}{g(0)}.$
 % Let $k\rightarrow \infty,$ we get the desired equation \eqref{maeq1}.

%}

%{\color{blue}
Denote by $U_0$ the interior of $\partial v_0(\mathbb{R}^2)$.
Since $v_0$ is a convex function defined on entire $\mathbb{R}^2,$ 
$U_0$ is convex.  
First we need a property  that
 for any $\tau>0,$ there exists a constant $M_\tau>0$ independent of $h$ such that 
 \begin{equation}\label{ginl131}
	B_\tau(0)\cap U_h\subset Dv_h(B_{M_\tau}(0)\cap V_h) \quad\text{ for $h>0$ small}.
\end{equation}
This property will be proved for general dimension later, see Lemma \ref{goodinclu1} and its proof.

For any $k$ large, by \eqref{aac2} we also have  
 \begin{equation}\label{ginl132}
Dv_h(B_{k}(0)) \subset B_{_{C_k}}\cap \overline{U_h}  \quad\text{ for $h>0$ small}
\end{equation}
for some constant $C_k$ independent of $h.$
By \eqref{ginl131} and \eqref{limitshape},  we have that
$$\{x:  x_2>Cx^2_1\}\cap B_{\tau}(0)\subset Dv_h(B_{M_{\tau}}(0)).$$
%where $\rho_h(x_1)=\frac{1}{2r}t^2+o(1)t^2$ for $t<0.$
Let $h\rightarrow \infty,$ and then take $\tau\rightarrow \infty$ (also take $M_{\tau}\rightarrow \infty$) we have that
$$\{x\in \mathbb{R}^2:  x_2>Cx_1^2\}\subset \partial v_0(\mathbb{R}^2),$$ which implies
\begin{equation}\label{U0c}
\{x\in \mathbb{R}^2: x_2>Cx_1^2\}\subset U_0.
\end{equation}
By \eqref{U0c} and the convexity of $U_0,$ we have that $U_0\subset \{x: x_2\geq 0\}.$ Hence $U_0$ is the epigraph of some convex function $\rho_0$ with $\rho(0)=\rho'(0)=0,$ namely, $U_0=\{x: x_2>\rho_0(x)\}.$  
Replacing $\{x\in \mathbb{R}^2:  x_2>Cx_1^2\}$ by $\{x: x_2>\rho_h(x_1), x_1<0\}$ in  the above argument, we have that
$\{x\in \mathbb{R}^2: x_2>\frac{1}{2r}x_1^2, x_1<0\}\subset U_0,$ which implies
$\rho_0(x_1)\leq \frac{1}{2r}x_1^2$ for $x_1<0.$ Note that  $\rho_h(t)=\frac{1}{2r}t^2+o(1)t^2$ for $t<0.$

Then, for any $k$ large, since the convex functions $v_h$ locally uniformly converges to $v_0$ in $\mathbb{R}^n,$ and both
$v_h, v_0$ are $C^1$ in the interior of $B_k(0)\cap V_0$ (provided $h$ is sufficiently small), by convexity we have that $Dv_h$ converges to $Dv_0$ locally uniformly in $B_k\cap V_0.$ Hence, $Dv_h(x)\rightarrow Dv_0(x)$ for any $x\in V_0.$
Then, by \eqref{ginl132} and \eqref{limitshape}, taking limit $h\rightarrow 0,$  we have that
 $Dv_0(V_0)\cap\{x: x_1\leq 0\}\subset \{x: x_2\geq \frac{1}{2r}x_1^2\}.$ By \eqref{maeq1} we see that 
 $|\partial V_0(\mathbb{R}^2\backslash V_0)|=0,$ which implies that $|U_0\backslash Dv_0(V_0)|=0.$ Note that the boundary of convex set has Lebesgue measure $0$. From the above discussion we deduce that $U_0\cap\{x: x_1\leq 0\}\subset \{x: x_2\geq \frac{1}{2r}x_1^2\},$ which implies that $\rho_0(x_1)\geq \frac{1}{2r}x_1^2$ for $x_1<0$.

Therefore, we have
\begin{equation}\label{limU0} 
U_0 = \Big\{(x_1,x_2)\in\R^2 \,:\, x_2> \rho_0(x_1) \Big\}
\end{equation}
 where $ \rho_0$ is a convex function satisfying $0\leq \rho_0(t)\leq Ct^2$ and   
	$\rho_0(t)=\frac{1}{2r}t^2$ for $t<0.$
Hence $U_0\subset\{x_2\geq0\}$ and $\{x_2=0\}$ is a support plane of $U_0$ at $0$.%}
\end{proof}

\subsection{Blow-up in higher dimensions}

%In this section, we prove the obliqueness in higher dimensions  
%by further exploiting the interior ball condition of Lemma \ref{intball} and adapting some techniques from \cite{CLW1}. 
In this subsection we assume $n\geq 3$ and the obliqueness fails at $x_0\in\mathcal{F}.$
Similarly as in \S\ref{S5.1}, 
denote $y_0=Du(x_0)$, which is a point on $ \partial V\setminus \overline{\partial V\cap \Omega^*}\subset \partial\Omega^*$.
Denote still by $\nu_{_U}(x_0), \nu_{_V}(y_0)$ the unit inner normals of $U, V$ at $x_0, y_0$, respectively. 
By a change of coordinates, we assume that $x_0=0$, 
$\nu_{_U} (0)=e_n$, and $\nu_{_V} (y_0)=e_1$.   
By subtracting a constant we can also assume that $v\geq 0$ and $v(y_0)=0$.
From \eqref{normalformular}, $y_0=re_n$ for some $r>0$.  

Unless otherwise specified, we use the notations $x=(x_1,\cdots,x_n)\in\R^n$; $x'=(x_1,\cdots,x_{n-1})$, $\hat x=(x_2,\cdots,x_n) \in\R^{n-1}$; and $\tilde x=(x_2,\cdots,x_{n-1})\in\R^{n-2}$.
 
Similarly to \eqref{parU}, the free boundary $\mathcal F$ can locally be expressed by
	$$\mathcal{F}=\{x  \,:\, x_n=\rho(x_1,\tilde x)\}\ \ \ \text{near $0$} $$
for some function $\rho$. 
By Lemma \ref{intball}, $\mathcal{F}$ lies below the ball $B_r(y_0)$ near $0$.
Hence by $ii)$ of Theorem \ref{CMCL}, the function $\rho$ satisfies 
	\begin{equation}\label{domain1}
		-C(x_1^2+|\tilde{x}|^2)^\frac{1+\alpha'}{2}\leq\rho(x_1,\tilde{x})\leq C(x_1^2+|\tilde{x}|^2)
	\end{equation}
for some $\alpha'\in(0,1)$. 
Analogously to \eqref{parV}, we also have
$$\partial V= \left\{ y  \,:\, y_1=\rho^*(\tilde{y}, y_n-r) \right\} \quad  \text{ near } \ y_0  $$ 
for some $C^2$ smooth and uniformly convex function $\rho^*$, which can be expressed as
	\begin{equation}\label{defppp}
		\rho^*(\tilde{y},t)=P(\tilde{y},t)+o(|\tilde{y}|^2+t^2),
	\end{equation}
where $P$ is a quadratic polynomial satisfying 
 $$C^{-1}(|\tilde{y}|^2+t^2)\leq P(\tilde{y},t)\leq C(|\tilde{y}|^2+t^2) $$
for some positive constant $C$.

For brevity,  we write $S_h[v](y_0), S_h^c[v](y_0)$ simply as
$S_h[v], S^c_h[v]$ when no confusion arises. 
By $(ii)$  of Corollary \ref{co21},  for any given $\eps>0$, there exists $C_\eps$ such that
\begin{equation}\label{h1}
\begin{split}
& B_{C_\epsilon h^{\frac{1}{2}+\epsilon}}(y_0)\cap\{y_1=0\}  \subset S_h^c[v] .   \\
%& B_{C_\epsilon h^{\frac{1}{2}+\epsilon}}(y_0)\cap\partial V  \subset S_h[v], 
\end{split}
\end{equation}
A key estimate is the following

\begin{lemma}\label{hkey1}
For any given $\epsilon>0$ small, 
there exists a constant $C_\epsilon$ such that for all unit vector $e\in\text{span}\{e_2, e_3, \cdots, e_{n-1}\}$,
\beq\label{xe}
|(y-y_0) \cdot e|\leq C_\epsilon h^{\frac{1}{2}-\epsilon} \quad\forall\,y\in S_h^c[v].
\eeq 
\end{lemma}

Let $p=(p_1, 0,\cdots, 0, p_n)$ be a point on $\partial\{v<h\}\cap \partial \Omega^*$ with $p_n<r$ 
(see Fig. \ref{fighd}).
Denote $s=r-p_n$.    
Since $\partial \Omega^*$ is $C^2$ smooth and uniformly convex, 
we have  $p_1=as^2+o(s^2)$ for a positive constant $a.$
Lemma \ref{hkey1} is built upon the following estimate.

\begin{lemma}\label{hkey2}
For any $\epsilon>0$ small, there exist constants $C, C_\epsilon$ such that
\beq\label{hsh}
C h^{\frac{1}{3}}\leq s\leq C_ \epsilon h^{\frac{1}{3}-\epsilon}
\eeq
when $h>0$ is small, where $C>0$ is a constant independent of $\eps$.
\end{lemma} 

\begin{proof} 
Let $D \subset \text{span}\{e_1,e_n\}$ be 
a two-dimensional region enclosed by $\partial \Omega^*$ and  the segment $\overline{y_0p}$  (see Fig. \ref{fighd}).  
By \eqref{volD}, we have $|D|_{\mathcal{H}_2}=\frac{1}{6}as^3+o(s^3)$, where $|\cdot|_{\mathcal{H}_d}$ denotes the $d$-dimensional Hausdorff measure. 
From \eqref{equi0}, we have 
	\begin{equation}\label{doDD}
		D\subset S_h[v]\cap V\subset S_{bh}^c[v].
	\end{equation}
	By \eqref{h1} we have 
	\begin{equation}\label{abc1}
	C_\epsilon h^{\frac{1}{2}+\epsilon}e_i\subset S^c_{bh}\quad \text{for}\ i=2, \cdots, n-1.
	\end{equation}
Combining these estimates and using  \eqref {nm1} and the convexity of $S^c_{bh}$, we obtain
	\begin{equation*} 
	 h^{\frac{n}{2}}\approx |S_{bh}^c|_{\mathcal{H}_n}\geq C_\epsilon h^{(\frac{1}{2}+\epsilon)(n-2)} |D|_{\mathcal{H}_2} \geq C_\epsilon s^3h^{(\frac{1}{2}+\epsilon)(n-2)}. 
\end{equation*}
Hence the second inequality of \eqref{hsh} is obtained.

\renewcommand{\figurename}{Fig.}
\renewcommand{\captionlabeldelim}{}
\begin{figure}[h]
	\centering
	\includegraphics[width=0.6\textwidth]{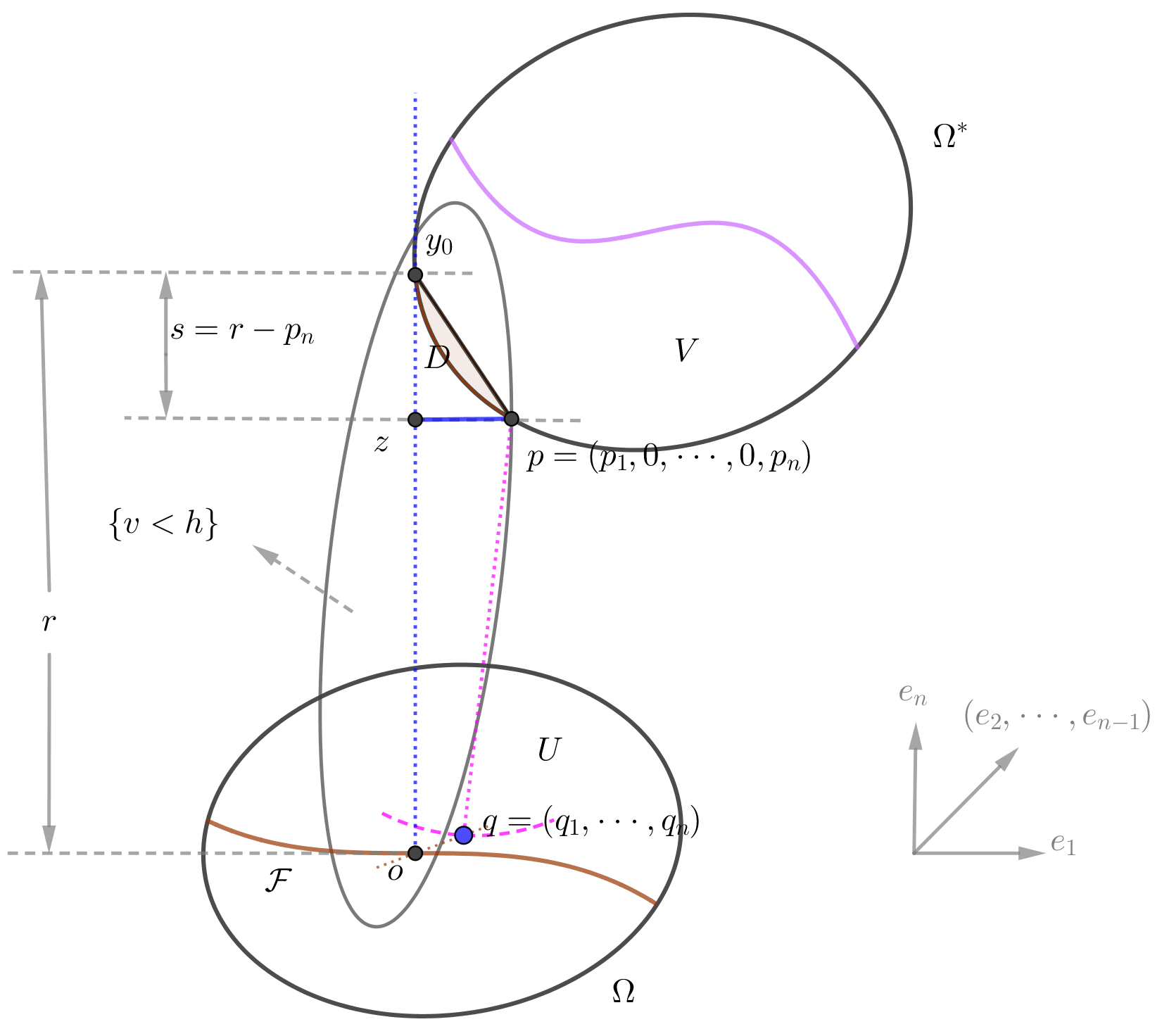}
	\caption{}
    \label{fighd}
\end{figure}

Next we show the first inequality of \eqref{hsh}.
By the reasoning before Lemma \ref{dlemma}, 
we may assume that $v\geq 0$ in $\mathbb{R}^n$, $v(0)=v(y_0)=0$,   and
$v=0$ on the segment $\overline{0y_0}$.
In particular, we have $v(z)=0$,
where $z=p_ne_n$ is the projection of $p$ on the $x_n$ axis.
Denote $q=(q_1,\cdots, q_n) =Dv(p)\in \mathcal{F}$.
By the convexity of $v$, we have 
\begin{equation}\label{h4}
q_1=Dv(p)\cdot e_1\geq \frac{v(p)-v(z)}{|p-z|}=\frac{h}{p_1}\geq C\frac{h}{s^2} .
\end{equation}
By the interior ball property  (Lemma \ref{intball}),  we have $B_{|p-q|}(p)\cap \Omega\subset U$. 
Hence
\begin{equation}\label{h5}
|p-q|^2\leq |p-0|^2.
\end{equation}
By the interior ball property again,  
the free boundary $\mathcal{F}$ lies below the ball $B_r(y_0)$.
It implies
\begin{equation}\label{h55}
	q_n\leq r-\sqrt{r^2-|q'|^2} ,
\end{equation}
where $q'=(q_1,q_2,\cdots,q_{n-1})$.

Note that when $h>0$ is sufficiently small, by strict convexity of $v$ in $V$, $S_h[v]$ will be
small, and then by the continuity of $Dv$, $|q|$ will be small, which ensures $|q'| <r$
and $p_n>r-\sqrt{r^2-|q'|^2}.$
Recall that $p_n=r-s$. By \eqref{h5} and \eqref{h55}, we have
$$|q'|^2+p_1^2-2p_1q_1+\big(r-s-(r-\sqrt{r^2-|q'|^2})\big)^2\leq p_1^2+(r-s)^2, $$
from which one infers that
\begin{align*}
2sr&\leq 2p_1q_1+2sr \Big( 1-\frac{|q'|^2}{r^2} \Big)^{\frac{1}{2}}\\
&\leq 2p_1q_1+2sr-\frac{s}{r}|q'|^2.
\end{align*}
Namely
$\frac{s}{r}|q'|^2\le 2p_1q_1$. 
Noting that $q_1\le |q'|$, we thus obtain  
\begin{equation*}
	\frac{s}{r}q_1 \leq 2p_1.
\end{equation*}
Recall that $p_1 \leq Cs^2+o(s^2)$. By \eqref{h4}, we then deduce
$$ \frac{h}{s r} \leq Cp_1\le Cs^2,$$
from which it follows that $s\geq Ch^{\frac{1}{3}}$. 
So the first inequality of \eqref{hsh} is proved.
\end{proof}

With Lemma \ref{hkey2}, we are now ready to prove Lemma \ref{hkey1}.

\begin{proof}[Proof of Lemma \ref{hkey1}]
Let $D$ be the region defined in the proof of Lemma \ref{hkey2}, (see Fig. \ref{fighd}).
By \eqref{doDD},  
	\begin{equation}\label{doDD1}
		D\subset S_{bh}^c[v].
	\end{equation}
From \eqref{volD} and thanks to \eqref{hsh}, we have
\beq\label{AD}
|D|_{\mathcal{H}_2}=\frac{1}{6}as^3+o(s^3)\geq Ch,
\eeq
provided $h>0$ is small enough.

Let $e\in \text{span}\{e_2,e_3,\cdots,e_{n-1}\}$ be a unit vector.
Denote by  $e^{\perp}$ the subspace orthogonal to $e$, 
passing through the point $y_0$.
Then by \eqref{h1}, \eqref{doDD1} and \eqref{AD}, 
we have  
$$ |S_{bh}^c[v]\cap e^{\perp}|_{\mathcal{H}_{(n-1)}} \geq C h^{(\frac{1}{2}+\epsilon)(n-3)} |D|_{\mathcal{H}_2} \ge C_\epsilon h^{1+(\frac{1}{2}+\epsilon)(n-3)}.$$
Hence, $\forall\,y\in S_{bh}^c[v]$,  by the convexity of $S_{bh}^c[v]$ and the volume estimate \eqref{nm1} we obtain
\begin{equation*}
\begin{split}
	h^{n/2}\approx |S_{bh}^c[v]|_{\mathcal{H}_n} &\geq |S_{bh}^c[v]\cap e^{\perp}|_{\mathcal{H}_{(n-1)}} \times  |(y-y_0)\cdot e|	\\
				& \geq C_\epsilon h^{1+(\frac{1}{2}+\epsilon)(n-3)}  |(y-y_0)\cdot e| ,
\end{split}
\end{equation*}
which implies that $|(y-y_0)\cdot e|\leq C_\epsilon h^{\frac{1}{2}-(n-3)\epsilon}$.   
Note that the constant $b$ in \eqref{equi0} is independent of $h$.
Replacing $h$ with ${h}/{b}$, we then obtain the desired estimate \eqref{xe}.
\end{proof}

%%% $$\text{||||||}$$

\begin{corollary}\label{coroabove}
For any $y\in S_h[v]$,  we have 
$$Dv(y)\cdot e_n\geq -C_{\epsilon}h^{1-\epsilon}\ \ \ \text{for $h>0$ small. } $$
\end{corollary}

\begin{proof} 
For any given $y\in S^c_h[v]\cap V$,  by Lemma \ref{hkey1},
\begin{equation}\label{ba1}
|y_i|\leq C_\epsilon h^{\frac{1}{2}-\epsilon}\   \ \ \text{for}\ i=2,\cdots, n-1.
\end{equation}
From \eqref{h1},  
$B_{C_\epsilon h^{\frac{1}{2}+\epsilon}}(y_0)\cap\{y_1=0\}\subset S^c_{h}[v].$
Hence by \eqref{nm1}, we have 
$$h^{\frac{n}{2}}\approx |S_{h}^c[v]| \geq C_\epsilon^{n-1}h^{(\frac{1}{2}+\epsilon)(n-1)}y_1 ,$$  
which implies
\begin{equation}\label{ba2}
y_1\leq C_\epsilon h^{\frac{1}{2}-(n-1)\epsilon} . 
\end{equation}

For any given $y\in S_h[v]$,
by the equivalence relation \eqref{equi0}, the above estimates \eqref{ba1} and \eqref{ba2} also hold.
By Lemma \ref{intball}, similarly to \eqref{h5} we have $|y-Dv(y)|\leq |y|$.   
Hence if $Dv(y)\cdot e_n<0$,  we have
 \begin{align}\label{ba3}
 |Dv(y)\cdot e_n|
   & \leq |y|-y_n=\sqrt{y_1^2+\cdots+y_{n-1}^2+y_n^2}-y_n  \nonumber \\
   & \leq  \frac{C}{y_n}(y_1^2+\cdots+y_{n-1}^2).
 \end{align}
When $h>0$ is small, $y$ is close to $y_0$ and $y_n\geq \frac{r}{2}$.
Combining \eqref{ba1}, \eqref{ba2} and \eqref{ba3}, we obtain $Dv(y)\cdot e_n\geq -C_\epsilon h^{1-\epsilon}.$
\end{proof}

%{\Small\color{blue} (Point out that in the rest of the section, 
%we will not use the condition that $\Om^*$ and $\Om$ are separate anymore,
%so that we can assume that $y_0=0$.)}
%{\Small\color{red} ok}

In the rest of the section, 
we will not use the condition that $\Om^*$ and $\Om$ are separate anymore.
By the changes in Remark \ref{y0},  we may assume $y_0=0$ for simplicity.
Let $T_h$ be an affine transformation such that $T_h(S_h^c[v]) \sim B_1(0)$.
Let $T_1:y \mapsto \bar y$ be the transform given by
\beq
\label{T1}
\left\{ 
{\begin{array}{ll}
  \bar y_1=h^{-\frac23}y_1,& \\
  \bar y_i= h^{-\frac12}y_i, & \ \ \ i=2,\cdots, n-1,\\
  \bar y_n=h^{-\frac13}y_n. &
  \end{array}}\right. 
\eeq
The following lemma shows that $T_h$ is close to $T_1$, 
in the sense that the norm of $T_2:=T_h\circ T_1^{-1}$ is bounded by $h^{-\epsilon}$ for any $\epsilon>0$, when $h>0$ small.  
It provides geometric estimates for the shape of the centred sub-level set $S^c_h[v]$.

\begin{lemma}\label{gse1}
For any $\epsilon>0$,
there exists a constant $C_\epsilon>0$ independent of $h$ such that 
\begin{equation}\label{ag1}
B_{\frac{1}{C_\epsilon}h^{\epsilon}}\subset T_1(S_h^c[v])\subset B_{C_\epsilon h^{-\epsilon}},
\end{equation}
and
 \begin{equation}\label{ag2}
 \|T_2\|+\|T_2^{-1}\|\leq C_\epsilon h^{-\epsilon}.
 \end{equation}
\end{lemma}

 \begin{proof} 
Let  $b$ be  the constant in  \eqref{equi0}. 
By \eqref{h1} we have that 
\begin{equation}\label {halo1}
B_{C_\epsilon h^{\frac{1}{2}+\epsilon}}(0)\cap\{y_1=0\}\subset S^c_{bh}[v].
\end{equation}
Let $D$ be domain in the $x_1x_n$-plane, given in the proof of Lemma \ref{hkey2}. 
Let $G$ be the convex hull of the set 
$D\cup \{C_\epsilon h^{\frac{1}{2}+\epsilon}e_i: i=2,\cdots, n-1\}.$
Since $D\subset S_h[v]\subset S_{bh}^c[v]$,   by \eqref{h1} we have $G\subset S_{bh}^c[v].$

By Lemma \ref{hkey2},  we have 
$$T_1(G)\subset B_{C_\epsilon h^{-c_{_1}\epsilon}}(0)\ \ \text{ and }\ \ |T_1(G)|\geq C_\epsilon h^{c_{_2}\epsilon}$$ 
for some constants $c_{_1}, c_{_2}>0$.  Note that the first inclusion uses $p_1=as^2+o(s^2)\lesssim h^{\frac{2}{3}-2\epsilon},$ and the second inequality use \eqref{volD},
the estimate on $|D|_{\mathcal{H}_2}.$

By convexity, it implies that there exists a ball 
\begin{equation}\label{lesia1}
B_{C_\epsilon h^{c_{_3}\epsilon}}(z) \subset T_1(G)
\end{equation}
for some point $z\in T_1(G)$ and some constant $c_{_3}>0$.   
As $\epsilon >0$ can be arbitrarily small,
we may simply assume that  $c_{_3}=1$.  

Since $T_1(G)\subset T_1(S_{bh}^c[v])$ and $|T_1(S_{bh}^c[v])|\approx 1$,   
by \eqref{lesia1} we have  $$\text{diam}(T_1(S_{bh}^c[v])) \leq C_\epsilon h^{-(n-1)\epsilon}.$$   
%Let $E$ be the minimum ellipsoid of $T_1S_{bh}^c[v]$,
 By John's Lemma \cite[Lemma 2.1]{C96}, there exists an ellipsoid $E$ centred at $0,$
such that $E\subset T_1(S_{bh}^c[v])\subset CE$ for some constant $C$ depending only on $n.$   
Let $r_1\leq \cdots \leq r_n$ be  the principal semi-axes of $E$.
Then we have $r_n\leq  C_\epsilon h^{-(n-1)\epsilon}$ and $r_1r_2\cdots r_n\approx |T_1(S_{bh}^c[v])|\approx 1$,   which implies $r_1\geq \frac{1}{C_\epsilon} h^{(n-1)^2\epsilon}.$
Therefore we obtain \eqref{ag1}.

Recall that $T_2\circ T_1(S_h^c[v])=T_h(S_h^c[v]) \sim B_1$. By \eqref{ag1} we have \eqref{ag2}.
\end{proof}

\begin{remark}\label{shapelev}
Note that since $T_h(S^c_h[v])\sim B_1(0),$ by \eqref{equi0} the equivalence relation between $S^c_h[v]$ and $S_h[v]$ 
we have that $T_h(S_h[v])$ also has a good shape and satisfies 
	\beq\label{shapegd}
		B_{\frac{1}{C}}(0)\cap T_h(V)\subset T_h(S_h[v]) \subset B_{C}(0)\cap T_h(V) 
	\eeq 
for some constant $C>0$ independent of $h.$
\end{remark}

With Lemma \ref{gse1} for the geometric estimate of the sub-level set $S^c_h[v]$, 
we can now carry out the normalisation process. 
Let
	\begin{equation}\label{novh}
		v_h(y):=\frac{1}{h}v(T_h^{-1}y).
	\end{equation}
Similarly to the claim following \eqref{buvh}, $v_h$ is locally uniformly bounded in $\R^n$ as $h\to0$. 
Hence by passing to a subsequence, $v_h\rightarrow v_0$, $T_h(V)\rightarrow V_0$ locally uniformly,
and $v_0$ satisfies 
\begin{equation}\label{maeq111}
\det\, D^2 v_0=c_0\chi_{_{V_0}}\quad\text{in } \R^n
\end{equation}
for a constant $c_0>0$.
Here by $T_{h}(V)\rightarrow V_0$ locally uniformly we mean that for any fixed $k>0$ large, $T_{h}(V)\cap B_k(0)$ converges to $V_0\cap B_k(0)$ as $h\rightarrow 0$ in Hausdorff distance. 
Note that $T_{h}(V)\cap B_k(0)$ is convex when $h$ is sufficiently small.
Since for any fixed $k>0,$
we have the diameter of $T_h(V)\cap B_k(0)$ is uniformly bounded for all $h$ small, hence by the Blaschke 
selection theorem that up to a subsequence we have $T_h(V)\cap B_k(0)$ converges to a convex set. Then by the
standard diagonal method, we can choose a subsequence such that $T_{h}(V)\rightarrow V_0$ locally uniformly.

Since $V_0$ is convex,  the doubling property holds for the centred sub-level sets of $v_0$, namely
$$ \Big|\frac{1}{2}S_h^c[v_0](y) \cap V_0 \Big| \geq C \Big|S_h^c[v_0](y)\cap V_0\Big|\quad\forall\ y\in \overline{V_0}, $$ 
where the constant $C$ depends only on $n$. 
%{\Small\color{blue}($C$ also depends on the principal curvature of $\p V_0$, which in turn depends in $\Om, \Om^* \cdots$?)}
As $v_0$ is a global convex function, $\partial v_0(\R^n)$ is also convex. 
Hence, by \eqref{maeq111} and Caffarelli's boundary regularity theory \cite{C92b},
 $v_0$ is strictly convex and $C^1$ smooth in $\overline{V_0}$. 
However, unlike \eqref{limU0} in dimension two, we do not have any further information 
on the regularity of  $\p U_0$, where $U_0$ is the interior of  $\partial v_0(\R^n)$.
Thus we cannot infer higher regularity of $v_0$ at the moment.
To overcome this difficulty, our strategy is to show 
that the blow-up limits $U_0, V_0$ have nice decomposition properties (Lemmas \ref{split1}--\ref{hkey6}).
%$U_0, V_0$ are flat in the $e_2,\cdots,e_{n-1}$ directions (Lemmas \ref{split1}--\ref{hkey6}).

Denote $V_h=T_h(V)$.
% and $V_0=\lim_{h\to0}V_h$.  
The following lemma shows that in the normalisation \eqref{novh}, 
the modulus of convexity and the $C^{1,\alpha'}$ norm of $v_h$ are locally uniformly bounded as $h\to0$. 

\begin{lemma}\label{uniformestimate1}

There exist constants $\alpha'\in(0,1]$ and $\beta' \ge 2$ such that
 \begin{equation}\label{ddes}
 C_1|y|^{\beta'} \leq  v_h(y)\leq C_2|y|^{1+\alpha'}\quad \text{ for }y\in B_1(0)\cap V_h,
\end{equation}
where  the positive constants $C_1$ and $C_2$ are independent of $h.$
\end{lemma}

\begin{proof}
Since $T_h(S_h^c[v])\sim B_1$, 
%by \eqref{nm1} and \eqref{equi0}, 
by Remark \ref{shapelev}
$T_h(S_h[v])$ has a good shape and 
\begin{equation}\label{noSh} 	
	B_{\frac{1}{C}}(0)\cap V_h\subset T_h(S_h[v]) \subset B_{C}(0)\cap V_h
\end{equation}
 for a constant $C$ independent of $h.$

The geometric decay estimate (see \cite[Lemma 2.2]{C96} or \cite[Lemma 7.6]{CM}) implies that for
 any given $s_1<1,$ there exists a constant $s_0<1$ independent of $h$ such that
	\begin{equation}\label{gdecay3}
	 S^c_{\bar{s} h}[v]\subset s_1S^c_h[v]\quad\  \forall\,\bar{s}\in (0, s_0).
	 \end{equation}

 Since \eqref{gdecay3} is invariant under the normalisation \eqref{novh}, the inclusion \eqref{gdecay3} still holds for $v_h,$
 namely, given $h$ small we have
  \begin{equation}\label{decaybegin}
 S^c_{\bar{s} \tilde h}[v_h]\subset s_1S^c_{\tilde h}[v_h]\quad\  \forall\,\bar{s}\in (0, s_0).
  \end{equation}
   for $\tilde{h}<1.$
Choose $s_1=\frac{1}{2}$ and let $\bar{s}=\frac{1}{4}s_0<\frac{1}{4}.$
By \eqref{noSh} we have $B_{\frac{1}{C}}(0)\subset S^c_1(v_h)\subset B_C(0).$
For any $y\in B_1(0)\cap V_h,$ let $k$ be the positive integer satisfying 
\begin{equation}\label{ydd1}
C2^{-k}< |y|\leq C2^{-k+1}.
\end{equation}
By \eqref{decaybegin}, we have $y\notin \frac{1}{2^k}S_1^c[v_h]\supset S^c_{\bar{s}^k}[v_h].$
By  \eqref{equi0}, we have $S^c_{\bar{s}^k}[v_h]\cap V_h\supset S_{b^{-1}\bar{s}^k}[v_h].$
Hence $v_h(y)\geq b^{-1}\bar{s}^k.$
From \eqref{ydd1}, it follows that $k\geq \frac{\log(2C)-\log|y|}{\log 2}.$ 
Therefore, $v_h(y)\geq C_1|y|^{\beta'},$ where
$C_1=b^{-1}\bar{s}^{\frac{\log(2C)}{\log 2}},$ and $\beta'=-\frac{\log\bar{s}}{\log 2}.$

To prove the second inequality, we \emph{claim} that there exists a constant $\delta>0$ such that
\begin{equation}\label{regu11}
v(\frac{1}{2}z)\leq \frac{1}{2}(1-\delta)v(z)\ \ \text{for any $z\in B_1(0)\cap V.$}
\end{equation}
 Indeed, if the claim fails, then there exist  $\delta_k\rightarrow 0$, $z_k\in B_1(0)\cap V$
 such that 
  $$v(\frac{1}{2}z_k)\geq \frac{1}{2}(1-\delta_k)v(z_k).$$ 
The strict convexity of $v$ implies that 
$z_k\rightarrow 0$  and $h_k:=v(z_k)\rightarrow 0$  as $k\rightarrow \infty.$ 
Denote $\hat{z}_k=T_{h_k}z_k.$ 
Then  we have $v_{h_k}(\hat{z}_k)=1$ and 
$$v_{h_k}(\frac{1}{2}\hat{z}_k)\geq \frac{1}{2}(1-\delta_k)v_{h_k}(\hat{z}_k).$$
 By passing to a subsequence, we may assume that $\hat z_k\rightarrow z_0\in \overline{V_0}$ 
 and $v_0(\frac{1}{2}z_0)=\frac{1}{2}v_0(z_0).$ 
 By convexity, we see that $v_0$ is linear on the segment  $\overline{0z_0},$ 
 which contradicts to the strict convexity of $v_0$ in $\overline{V_0}.$ 
 Hence, the claim \eqref{regu11} is proved. 
 
 Since \eqref{regu11} is invariant under the normalisation \eqref{novh}, 
 it also holds for $v_h$. 
 Hence $v_h(\frac{1}{2}y)\leq \frac{1}{2}(1-\delta)v_h(y)$ for $y\in B_{\frac{1}{C}}(0)\cap V_h.$
By iteration we obtain $v_h(\frac{1}{2^k}y)\leq \frac{1}{2^k}(1-\delta)^kv_h(y)$.  
Hence there exist constants $\alpha'\in (0,1]$ and $C_2>0$, independent of $h$, 
such that $v_h(y)\leq C_2|y|^{1+\alpha'}$ for $y\in B_1(0)\cap V_h.$ 
 \end{proof}

\begin{lemma}\label{split1}
For the limit $V_0=\lim_{h\to0}V_h$, we have the decomposition
	\begin{equation}\label{split11}
	V_0=\omega_0^*\times H^*_0,
	\end{equation}  
where $H^*_0$ is an $n-2$ dimensional subspace of $\mathbb{R}^n$,  
$\omega_0^*\subset (H^*_0)^{\perp}:=\{y\in\mathbb{R}^n: y\perp H^*_0\}$ is convex,
and $\omega_0^*$ is smooth. Moreover, $\omega_0^*$ can be represented as an epigraph of some convex 
function.
\end{lemma}

\begin{proof}  
Recall that  the boundary $\partial V$ is uniformly convex and is given by the function $\rho^*$ in \eqref{defppp}.
Let $e\in H:=\text{span}\{e_2, e_3,\cdots, e_{n-1}\}$ be any given unit vector. 
Let 
	$$z=te+\rho^*(te)e_1 \in \partial V$$ 
be a boundary point, where $t=h^{\frac{1}{2}-2\epsilon}$ and $\epsilon>0$ is sufficiently small. 
Let's track the behaviour of the point $z$ under the affine transformation $T_h=T_2\circ T_1$.

By \eqref{T1}, we see that $T_1z=h^{-2\epsilon}e+h^{-2/3}\rho^*(te)e_1$.   Hence by \eqref{ag2} we have
\begin{equation}\label{flatt1}
|T_hz|\geq C_\epsilon h^{-\epsilon}\rightarrow\infty \quad \text{ as } h\rightarrow 0.
\end{equation}
Meanwhile, since $0\leq \rho^*(te)\leq Ct^2=Ch^{1-4\epsilon}$,  by \eqref{ag2} we also have 
\begin{equation}\label{flatt2}
\text{dist}(T_hz,T_hH )\leq \|T_2\|  h^{-2/3}\rho^*(te)
        \leq C_\epsilon h^{\frac{1}{3}-5\epsilon}\rightarrow 0 \quad \text{ as } h\rightarrow 0.
\end{equation}
Up to a subsequence, we assume that $T_hH$ converges to an $n-2$ dimensional subspace $H_0^*$
in the sense that $T_hH\cap B_k(0)$ converges to $H_0^*\cap B_k(0)$ in Hausdorff distance, for all given $k>0.$ Indeed,
since $T_hH$ is an $n-2$ dimensional subspace, we may assume $T_hH$ to be the  orthogonal complement of 
$\text{span}\{e_{1h}, e_{nh}\}$ with two orthogonal
unit vectors $e_{1h}$ and $e_{nh}.$ Then since $e_{1h}, e_{nh}\in \mathbb{S}^n$, up to a subsequence we may assume 
$e_{1h}, e_{nh}$ converges to $e_{10}, e_{n0}$, respectively. Let $H_0^*$ be the $n-2$ dimensional subspace orthogonal to 
$\text{span}\{e_{10}, e_{n0}\},$ then we have the desired convergence as above.

Given any  $y\in H_0^*,$ by the discussion above, we have that there exists a point $y_h\in T_hH$ such that
$y_h\rightarrow y$ as $h\rightarrow 0.$ Let $e_h:=\frac{T_h^{-1}y_h}{|T_h^{-1}y_h|},$ and
$z_h=te_h+\rho^*(te_h)e_1\in \partial V,$ where $t=|T_h^{-1}y_h|$ provided $h$ is small enough.
 Then, by \eqref{ag2} we have that $t\leq \frac{1}{C_\epsilon}h^{\frac{1}{2}-\epsilon}.$
 By the same computation leading to \eqref{flatt2}, we have that $\text{dist}(T_hz_h, T_h(te_h) )\rightarrow 0$ as $h\rightarrow 0.$ Note that $T_h(te_h)=y_h\rightarrow y$ as $h\rightarrow 0.$ Hence $\partial V_h\ni T_hz_h\rightarrow y$ as $h\rightarrow 0,$
 which implies that $y\in \partial V_0.$ Hence, $H_0^*\subset \partial V_0.$
By  the convexity of $V_0$, it follows that $V_0=\omega_0^*\times H^*_0$, where $\omega_0^*$ is a convex set in $(H^*_0)^{\perp}.$

Next we prove the smoothness of $\omega_0^*.$
From \eqref{defppp}, one  sees that 
\begin{equation}\label{faxiang1}
\tilde{e}_{h}:=\frac{(T_h^t)^{-1}e_1}{|(T_h^t)^{-1}e_1|}
\end{equation}
 is the unit inner normal of $V_h$ at $0$, where $T_h^t$ is the transpose of $T_h$ as a matrix.    
 Denote the unit vector $\tilde{e}'_{h}=\frac{T_h^{-1}(T_h^t)^{-1}e_1}{|T_h^{-1}(T_h^t)^{-1}e_1|},$ namely $T_h\tilde{e}'_{h}$ is in the direction of 
 $\tilde{e}_{h}.$
  By the definition of $T_h,$ a direct computation
 shows that 
\begin{equation} \label{e'h}
 \tilde{e}'_h\cdot e_1\geq C_\epsilon h^{4\epsilon}.
 \end{equation}
  By the $C^2$ regularity of $\partial V$ at 0 (see \eqref{defppp}) we have $x_h=(x_1, x_2, \cdots, x_n):=h^{6\epsilon}\tilde{e}'_h\in V$ provided $h$ is small.
  Indeed, by \eqref{e'h} we have
  $$x_1=h^{6\epsilon}\tilde{e}'_h\cdot e_1\geq C_\epsilon h^{10\epsilon} \gg h^{12\epsilon}\geq \sum_{i=2}^n |x_i|^2$$ for $h$
  small, which implies that $x_h\in V.$
Hence
 \begin{equation}\label{e'h11}
 T_hx_h=|T_hx_h|\tilde{e}_h\in V_h.
 \end{equation}
  By the definition of $T_h$ we have
  \begin{equation}\label{e'h12}
  |T_hx_h|\geq C_\epsilon h^{-\frac13+7\epsilon}\rightarrow \infty\ \ \text{ as } \ h\rightarrow 0.
  \end{equation}

 Extend the quadratic polynomial $P$ in \eqref{defppp} to $\R^n$ such that
 $$ \tilde P(y_1, \hat y) = P(\hat y), \quad \hat y=(y_2,\cdots, y_n).$$ 
Recall that, by   \eqref{defppp},
 	$$\partial V= \left\{(y_1,\hat y) \,:\, y_1 = P(\hat y)+o(P)\right\}\ \ \text{ near }0.$$    
By a straightforward computation, we have
\begin{equation}\label{h10}
\partial V_h=\big\{y \,:\, \langle y,\tilde{e}_h\rangle=\tilde{P}_h(y)+o(\tilde{P}_h)\big\} \ \  \text{ near }0,
\end{equation}
where $\tilde{P}_h(y)=\frac{1}{|(T_h^t)^{-1}e_1|}\tilde P(T_h^{-1}y)\geq 0$,
and
\begin{equation}\label{colla}
B_1(0)\cap V_h\subset \Big\{y \,:\, \langle y,\tilde{e}_h\rangle\geq \frac12\tilde{P}_h(y) \Big\} \ \ \text{ for $h>0$ small}. 
\end{equation}

We \emph{claim} that the coefficients of the quadratic polynomial $\tilde P_h$ are uniformly bounded as $h\to0$. 
Assume the claim for a moment.
Then by passing to a subsequence,
we have $\tilde{e}_h\rightarrow e_0^*, \tilde{P}_h \rightarrow \tilde{P}_0$
for a unit vector $e_0^*$ and a quadratic polynomial $\tilde{P}_0$.
Moreover,  the higher order term $o(\tilde P_h)$ in \eqref{h10}
converges  to $0$ locally uniformly as $h\to 0$.
Hence $\partial V_0=\{y  \,:\, \langle y, e_0^*\rangle=\tilde{P}_0\}$ is smooth,  
which implies that $\omega_0^*$ is  smooth.
 By \eqref{e'h11}, \eqref{e'h12} and convexity of $V_0,$ passing to limit, we have
\begin{equation}\label{epigraph1}
\{te_0^*: t>0\}\subset V_0,
\end{equation}
which implies that $w^*_0$ is an epigraph of some convex function.

It remains to prove the above claim.    
Let $d_h$ be the largest coefficient of $\tilde P_h$.
Suppose by contrary that $d_h\rightarrow \infty$ as $h\rightarrow 0$.   
Then $\frac{1}{d_h}\tilde{P}_h$ has bounded coefficients, and up to a subsequence we assume that 
$\frac{1}{d_h}\tilde{P}_h\rightarrow \tilde{P}_*$ for a quadratic polynomial $\tilde P_*$ whose largest coefficient equals $1$.
Hence by \eqref{colla},
$$ B_1(0)\cap V_h\subset 
B_1(0)\cap \Big \{ y \,:\, \frac{1}{d_h}\langle y,\tilde{e}_h\rangle\geq \frac{1}{2d_h}\tilde{P}_h(y) \Big\}.$$
Since $\tilde{P}_h(y)$ is a non-negative quadratic polynomial, we have that 
$$Q_h:=B_1(0)\cap \Big \{ y \,:\, \frac{1}{d_h}\langle y,\tilde{e}_h\rangle\geq \frac{1}{2d_h}\tilde{P}_h(y) \Big\}$$ is convex and uniformly bounded. Then, by Blaschke selection theorem, up to a subsequence, we may assume $Q_h$ converges to a convex set $Q_\infty$ 
in Hausdorff distance.  We claim that $|Q_\infty|=0.$ Suppose not, then there exists a ball $B_r(q)\subset Q_\infty.$ Hence,
$B_{\frac{r}{2}}(q)\subset Q_h$ for $h$ sufficiently small. This implies that $\frac{1}{d_h}\langle y,\tilde{e}_h\rangle\geq \frac{1}{2d_h}\tilde{P}_h(y)$ in $B_{\frac{r}{2}}(q),$ and passing to limit $h\rightarrow 0,$ we have that  $\tilde{P}_*=0$ in $B_{\frac{r}{2}}(q),$ contradicting to the fact that  the largest coefficient of $\tilde P_*$ equals $1$.  Therefore $|Q_\infty|=0.$ 
Since the convex set $Q_h\rightarrow Q_\infty$ in Hausdorff distance, and $B_1(0)\cap V_h\subset Q_h,$
we see that $|B_1(0)\cap V_h| \to 0$ as $h\rightarrow 0$.
On the other hand, by the uniform density property (Lemma \ref{ud1}), we have 
$
|B_1(0)\cap V_h|\geq \epsilon_0
$
 for some positive constant $\epsilon_0$ independent of $h$,
which leads to a contradiction.
The claim is thus proved.
\end{proof}

Note that under the normalisation \eqref{novh}, we have 
	\begin{equation}\label{Dvhy}
		Dv_h(y)=\frac{1}{h}(T^t_h)^{-1}Dv(T_h^{-1}y),
	\end{equation}  
where $T_h^t$ is the transpose of $T_h$ as a matrix. 
Denote $T_h^*:=\frac{1}{h}(T^t_h)^{-1}$. 
Then correspondingly, the free boundary $\mathcal{F}\subset Dv(\partial V)$ 
is changed to $T_h^*(\mathcal{F})$ by the normalisation \eqref{novh}.

Similarly to the decomposition following \eqref{T1}, we can decompose $T_h^*=T_2^*\circ T_1^*$ 
with $T_1^*=\frac{1}{h}(T_1^t)^{-1}$ and $T_2^*=(T_2^t)^{-1}$. From \eqref{T1}, 
the transform $T_1^* : x\mapsto \bar x$ is a rescaling given by
$$ 
\left\{ 
{\begin{array}{ll}
  \bar{x}_1=h^{-\frac13}x_1; & \\
  \bar{x}_i=h^{-\frac12}x_i & \ \ \ i=2,\cdots, n-1,\\
  \bar{x}_n=h^{-\frac23}x_n. &
  \end{array}}\right. 
$$
By Lemma \ref{gse1}, we also have the estimate $\|T_2^*\|+\|(T_2^*)^{-1}\|\lesssim h^{-\epsilon}$, 
similarly to \eqref{ag2}. 
In the following we denote $T_h^*(U)$ by $U_h.$

 \begin{lemma}\label{goodinclu1}
 For any $\tau>0$ large, there exists a constant $M_\tau>0$ independent of $h$ such that 
 \begin{equation}\label{ginl}
	B_\tau(0)\cap U_h\subset Dv_h(B_{M_\tau}(0)\cap V_h) \quad\text{ for $h>0$ small}.
\end{equation}
 \end{lemma}

 \begin{proof}
The inclusion \eqref{ginl} essentially follows from Lemma \ref{uniformestimate1}. 
In particular, for $\tau>0$ small enough (say, $\tau<C_1$ in \eqref{ddes}), 
\eqref{ginl} follows directly from the the first inequality in \eqref{ddes}. 
For $\tau>0$ large, we prove \eqref{ginl} by a re-scaling as follows.

Let $y\in V_h\setminus \{v_h<1\},$ such that $v_h(y)\geq1$.   
By the convexity of $v_h$ and 
\eqref{ddes} we have  
 \begin{equation}\label{uni444}
 \frac{v_h(y)}{|y|}\geq c_1
 \end{equation}
for some constant $c_1$ independent of $h.$
 For the given $\tau>0,$ by \eqref{ddes} and
 since the $C^{1,\alpha'}$ norm of $v_h$ is independent of $h,$
 there exists a small constant $\epsilon_\tau>0$, independent of $h$, such that 
 \begin{equation}\label{uni222}
 Dv_h(\{v_h<\epsilon_\tau\}\cap V_h)\subset \frac{1}{\tau}B_{c_1}(0)\cap U_h.
 \end{equation}
  
 Let $q$ be the intersection of the segment $\overline{0y}$ 
 and level set $\{v_h=\epsilon_\tau\}$, such that $v_h(q)=\epsilon_\tau$. 
 By \eqref{uni222} we have 
 \begin{equation}\label{uni555}
 \frac{v_h(q)}{|q|}\leq |Dv_h(q)| \leq \frac{1}{\tau}c_1.
 \end{equation}
  Let $\check q:=T_{\epsilon_\tau h}T_h^{-1}q$ such that $v_{\epsilon_\tau h}(\check q)=1$,  
 and let $\check y:=T_{\epsilon_\tau h}T_h^{-1}y$ such that $v_{\epsilon_\tau h}(\check y) \geq 1/\epsilon_\tau$. 
  Then, since \eqref{uni444} is independent of $h,$ we have
 \begin{equation}\label{uni666}
 \frac{v_{\epsilon_\tau h}(\check q)}{|\check q|}\geq c_1.
 \end{equation}
 Since $\frac{v_{\epsilon_\tau h}(\check y)}{v_h(y)}=\frac{v_{\epsilon_\tau h}(\check q)}{v_h(q)} $ 
 and $\frac{|\check{y}|}{|y|}=\frac{|\check{q}|}{|q|},$ by \eqref{uni444}, \eqref{uni555} and \eqref{uni666}
 we obtain 
 \begin{equation}\label{uni777}
 \begin{split}
 |Dv_{\epsilon_\tau h}(\check y)| &\geq \frac{v_{\epsilon_\tau h}(\check y)}{|\check y|} =  \frac{v_{\epsilon_\tau h}(\check q)}{|\check q|} \left(\frac{v_h(y)/|y|}{v_h(q)/|q|}\right) \\ 
 & \geq c_1 \frac{c_1}{c_1/\tau} \geq \tau c_1.
 \end{split}
 \end{equation}
 
Note also that for the $\epsilon_\tau<1$ small, by the convexity of $v_{\epsilon_\tau h}$ and \eqref{ddes} one has
 \begin{equation}\label{tauinclu}
 \Big\{v_{\epsilon_\tau h}<\frac{1}{\epsilon_\tau}\Big\}\cap V_{\epsilon_\tau h}
   \subset \frac{1}{\epsilon_\tau}\Big\{v_{\epsilon_\tau h}<1\Big\}\cap V_{\epsilon_\tau h}
    \subset B_{\frac{C}{\epsilon_\tau}}\cap V_{\epsilon_\tau h}
    \end{equation}
for some constant $C$ independent of $h.$
Therefore, from \eqref{tauinclu} it follows that 
 for any $\check y\in V_{\epsilon_\tau h}$ with $|\check y| \geq C/\epsilon_\tau$, one has
$v_{\epsilon_\tau h}(\check y)\geq \frac{1}{\epsilon_{\tau}},$ and then by \eqref{uni777}
 we have $|Dv_{\epsilon_\tau h}(\check y)|\geq \tau c_1$.  Namely,
 \begin{equation}\label{uni888}
 B_{\tau c_1}(0)\cap U_{\epsilon_\tau h}\subset Dv_{\epsilon_\tau h}\Big(B_{\frac{C}{\epsilon_\tau}}\cap V_{\epsilon_\tau h}\Big).
 \end{equation}
The conclusion \eqref{ginl} now follows from \eqref{uni888} by replacing $h$ with ${h}/{\epsilon_\tau}$.  
 \end{proof}

Denote by $U_0$ the interior of  $\partial v_0(\mathbb{R}^n)$. 
We have the following observation. 

\begin{lemma}\label{split2}
The set $U_0$ is convex, and can be decomposed into 
\begin{equation}\label{split3}
U_0=\omega_0\times H_0,
\end{equation}
 where $H_0$
is an $n-2$ dimensional subspace of $\mathbb{R}^n$,  and $\omega_0$ is a convex set in 
$H_0^\perp:=\{x:x\perp H_0\}.$
\end{lemma}
\begin{proof}
Since $v_0$ is a convex function on the entire space $\mathbb{R}^n,$ it is well known that the interior of $\partial v_0(\mathbb{R}^n)$ is a convex set.
Originally, by the second inequality of \eqref{domain1}
we have
$$\hat U:=\{x: x_n > C|x|^2\}\cap B_{r_1}\subset U\cap B_{r_1}$$ 
for some small $r_1>0$. 
By passing to a subsequence, 
we may assume the sequence of convex sets $T^*_h\hat U$ converges to a convex set $\hat U_0$ locally uniformly, as $h\to0$.
Similarly to the proof of Lemma \ref{split1}, by replacing $T_1, T_2$ therein with $T_1^*, T_2^*$, 
we see that $\partial{\hat U_0}$  contains an $n-2$ dimensional subspace $H_0$ of $\mathbb{R}^n.$
  By Lemma \ref{goodinclu1} we have $H_0\subset \partial{\hat U_0}\subset \partial v_0(\mathbb{R}^n)\subset \overline{U_0}.$
 By convexity of $U_0$,  we see that it must split as \eqref{split3}. 
\end{proof}

Let $u_0$ be the Legendre transform of $v_0,$ 
namely, 
	\beq\label{v0dual121}
		u_0(x)=\sup_{y\in \mathbb{R}^n} \left\{ x\cdot y-v_0(y) \right\} \quad \text{for}\ x\in \overline{U_0}.
	\eeq

\begin{lemma} \label{u0v0}
We have the following properties for $u_0, v_0$:\\
1. $v_0$ is $C^1$ and strictly convex in $\overline{V_0}.$ Moreover,
$v_0,$ as a convex function defined on $\mathbb{R}^n,$  is differentiable at all point $y\in \overline{V}_0.$\\
2.  $u_0$ is $C^1$ and strictly convex in $B_r(0)\cap \overline{U_0}$ for some $r>0$ small.
\end{lemma}
\begin{proof}
Since $V_0$ is convex, we have that the Monge-Amp\`ere measure $\det\, D^2v$ is doubling, hence
$S^c_h[v_0](y)$  has geometric decay property for any $y\in \overline{V_0}\cap B_{K}(0) ,$ given any fixed $K.$  
 By the similar proof to 
Lemma \ref{uniformestimate1}, we have that $v_0$ restricted to $B_{K}(0)\cap \overline{V_0}$  is strictly convex and $C^{1},$ for any fixed $K>0.$
Now, we only need to prove that $v_0$ is differentiable at $\partial V_0.$ The proof follows  \cite[Section 3, Proof of Theorem 2.1 (i)]{AC}. For reader's convenience, we sketch the proof here. 
Since $v_0$ is convex, for any unit vector $e$, the lateral derivatives
\begin{align*}
 &  \p_e^+ v_0(y) =: \lim_{t\searrow 0}  t^{-1} [v_0(y+te)-v_0(y)] \\
  &  \p_e^- v_0(y) =: \lim_{t\searrow 0} t^{-1} [v_0(y)-v_0(y-te)]
\end{align*}
exist. To prove that $v_0$ is $C^1$ at $y\in \partial V_0$, it suffices to prove that
\begin{equation}\label{lder}
 \p_e^+ v_0(y) = \p_e^- v_0(y)
 \end{equation}
for all unit vector
$e$.
By convexity of $v_0$, it suffices to prove  \eqref{lder} for $e= e'_k$ for all $k=1, 2, \cdots, n$,
where $e'_k$, $k=1,\cdots, n$, are any fixed $n$ linearly independent unit vectors. 
Since $V_0$ is convex, we can choose all of them point inside
$V_0,$ namely,  $te'_k\in V_0$ for $t>0$ small.
Assume to the contrary that $v_0$ is not $C^1$ at $y\in \partial V_0$. 
Suppose \eqref{lder} fails for some $e'_k.$  
Let us assume that  $x=0$, $v_0(0)=0$, $v_0\ge 0$, and $\p_{e'_k}^+v_0(0)>\p_{e'_k}^-v_0(0)=0$.

Now we consider a section $S^c_{h}[v_0](z)$, where $z=a'e_k'$ for some small constant $a'>0.$
Note that by John's lemma, there exits an ellipsoid $E$ with center $z$ such that $E\subset S^c_{h}[v_0](z)\subset C(n)E.$
Since $v_0$ is Lipschitz and $\p_{e'_k}^+v_0(0)>0$, we have that $C^{-1}\varepsilon\leq v_0(-\varepsilon e'_k)\leq C\varepsilon$ for any small positive $\varepsilon,$
where $C$ is a positive constant.
Since $\p_{e'_k}^+v_0(0)=0$, we have $v_0(Ma'e'_k)=o(a'),$ where $M=2C(n).$
 Hence, we can choose $a'>0$ small and  $\varepsilon=Cv_0(Ma'e'_k)$ and so that the following properties hold:\\
 1) $o(a')=v_0(Ma'e'_k)= C^{-1}\varepsilon \ll a'$, and\\
 2) $-\varepsilon e'_k$ is on the boundary of some section $S^c_{h}[v_0](z)$.\\
 The existence of such section $S^c_{h}[v_0](z)$ in 2) follows from the property that centered section, say $S^c_h[v_0](z)$, varies  continuously with respect to the height $h$,
see \cite[Lemma A.8]{CM}. 

Suppose  $S^c_{h}[v_0](z)=\{v_0<L\}$ for some linear function $L.$ 
Since $S^c_{h}[v_0](z)$ is balanced around $z=a'e_k'$ and $M=2C(n),$ we have that $Ma'e'_k\notin S^c_{h}[v_0](z).$
Hence $L(Ma'e'_k)\leq v_0(Ma'e'_k)\leq C^{-1}\varepsilon \leq  v_0(-\varepsilon e'_k)=L(-\varepsilon e'_k),$
where the second inequality follows from property 1) and the last equality follows from property 2).
Hence, $L$ is increasing in $-e'_k$ direction,
 which implies
 \begin{equation} \label {e1}
 (L-v_0)(0)\geq (L-v_0)(z)=h.
 \end{equation}
 On the other hand, since $\det\ D^2v_0$ is doubling for sections centered in $\overline{V}_0,$  we have that
 \begin{equation}
 \label{e2}
 (L-v_0)(0)\leq C(\frac{\varepsilon}{a'})^{\frac{1}{n}}h
 \end{equation}
contradicting to \eqref{e1} since $a' \gg \varepsilon.$ Hence $v_0$ must be differentiable at $y.$

By the strict convexity of $v_0$ in $\overline{V_0},$ we have that $|Dv_0(y)|\geq 2r>0$ for all $y\in \overline{V_0}\backslash B_1(0).$
Indeed, by convexity of $v_0,$ we have that $|Dv_0(y)| \geq \inf_{\partial B_1(0)\cap \overline{V_0}}v_0(y)$ for all $y\in \overline{V_0}\backslash B_1(0).$
Hence, 
\begin{equation}\label{Dvmapto1}
B_r(0)\cap \overline{U_0}\subset Dv_0(B_1(0)\cap \overline{V_0}).
\end{equation}
Now, $Du_0$ is  the optimal map from $Dv_0(B_1(0)\cap \overline{V_0})$ with density $1$ to $B_1(0)\cap \overline{V_0}$
with density $c_0.$ Since the densities are bounded from below and above, and the target domain is convex, by Caffarelli's regularity theory we have that $u_0$ is strictly convex and $C^1$ in  $B_r(0)\cap U_0.$ Note that this is an interior regularity property. It  follows that
\begin{equation}\label{mapinside1}
Du_0(B_r(0)\cap U_0)\subset B_1(0)\cap V_0,
\end{equation}
namely, the interior points in $B_r(0)\cap U_0$ will be mapped to the interior points of $V_0.$

First, we show that $u_0$ is strictly convex in  $B_r(0)\cap \overline{U_0}.$ Suppose not, then there exist
points $x, \tilde{x}\in  B_r(0)\cap \overline{U_0}$ such that $u_0$ is affine along the segment $x\tilde x.$ Let $p$ be the mid point of the segment $x\tilde x.$ Let $q\in \overline{V_0}$ such that $Dv_0(q)=p.$ Since $u_0$ is the Legendre transform of $v_0,$ it implies that the segment $x\tilde x$ is contained in the subdifferential of $v_0$ at $q,$ contradicting to the property that all the points in
 $\overline{V_0}$ are differentiable points of  $v_0.$

Now, we show that $u_0$ is $C^1$ in $B_r(0)\cap \overline{U_0}.$ We already have the interior regularity. For any 
$x\in \partial \overline{U_0}\cap B_r(0),$ If $u_0$ is not $C^1$ at $x,$ then there exists two sequence $U_0 \ni x_k, \tilde x_k
\rightarrow x$ such that $V_0\ni Du_0(x_k), Du_0(\tilde x_k)$ converges to two different points $y, \tilde y\in \overline{V_0}\cap B_1(0)$ respectively. It implies that $Dv_0(y)=Dv_0(\tilde{y}),$ by convexity of $v_0$ we have that $v_0$ is affine along the segment
$y\tilde y,$ contradicting to the strict convexity of $v_0$ in $B_{1}(0)\cap \overline{V_0}.$  Hence $u_0$ is $C^1$ in $B_r(0)\cap \overline{U_0}.$  
\end{proof}

\begin{remark}\label{U01}
Since $\det\, D^2 v_0=c_0\chi_{_{V_0}}\ \text{in } \R^n$ and $V_0$ is convex, we have that 
$|\partial v_0(\mathbb{R}^n\backslash V_0)|=0.$ It implies that 
for almost everywhere $x\in U_0,$ we can find $y\in V_0,$ such that $Dv_0(y)=x.$ 
 Note also that by continuity of $Dv_0$ in $\overline{V_0}$ we have $\overline{U_0}=Dv_0(\overline{V_0}).$
 Suppose for a subsequence $h_k\rightarrow 0,$ we have that
  $v_k:=v_{h_k}$ converges to $v_0$ locally uniformly in $\mathbb{R}^n.$  In particular, $v_k\rightarrow v_0$ uniformly in 
  $B_r(0)$ for any $r>0$ fixed. Now, we claim that
  $Dv_k$ converges to $Dv_0$ uniformly in $B_{\frac{r}{2}}(0)\cap \overline{V}_0.$
 Indeed, suppose $Dv_k$ does not converge to $Dv_0$ uniformly in $B_{\frac{r}{2}}(0)\cap \overline{V_0}.$ Then, there exists a positive constant $\epsilon>0$ and a sequence of points $y_k\in  B_{\frac{r}{2}}(0)\cap \overline{V_0},$ such that 
 \begin{equation}\label{uec112}
 |Dv_k(y_k)-Dv_0(y_k)|\geq \epsilon.
 \end{equation}
 By \eqref{aac2}, we have that $Dv_k$ is uniformly bounded in $B_r(0)$ for all $k.$ Passing to a subsequence, we may assume
 \begin{equation}\label{uec113}
Dv_k(y_k)\rightarrow x\in  \overline{U_0}
\end{equation}
 and $y_k\rightarrow y\in B_\delta(0)\cap \overline{V_0}.$  By continuity of $Dv_0$ we have that 
$Dv_0(y_k)$ converges to $Dv_0(y).$ 
By \eqref{uec112} we have that 
\begin{equation}\label{uec114}
|x-Dv_0(y)|\geq \epsilon.
\end{equation}
By convexity of $v_k,$ we have that
$v_k(z)\geq v_k(y_k)+Dv_k(y_k)\cdot (z-y_k).$ Since  $v_k\rightarrow v_0$ uniformly in $B_{r}(0),$ by \eqref{uec113}, passing to limit we have $v_0(z)\geq v_0(y)+x\cdot (z-y),$ which implies that $Dv_0(y)=x$ contradicting to \eqref{uec114}.
Hence  $Dv_k$ converges to $Dv_0$ uniformly in $B_{\frac{r}{2}}(0)\cap \overline{V_0}.$

Since $u_0$ is strictly convex and $C^{1}$ in $B_{r}(0)\cap \overline{U_0},$ similar to \eqref{Dvmapto1} we have that
 \begin{equation}\label{Dumapto}
 B_{r'}(0)\cap \overline{V_0}\subset Du_0(B_{r}(0)\cap \overline{U_0})
 \end{equation}
 for some positive constant $r'.$
 Then, for any $y\in \partial V_0\cap B_{r'}(0),$ we claim that $Dv_0(y)\subset \partial U_0\cap B_r(0).$ 
 Suppose not, then $x:=Dv_0(y)\subset U_0,$ which implies that $y=Du_0(x)$ is in the interior of $V_0,$ contracting to the assumption that  $y\in \partial \overline{V_0}\cap B_{r'}.$  Therefore
\begin{equation}\label{Dvmap}
Dv_0(\partial V_0\cap B_{r'}(0))\subset  \partial U_0\cap B_r(0).
\end{equation}
 
 \end{remark}

Similarly to \eqref{faxiang1}, by straightforward computation we see that
\begin{equation}\label{faxiang2}
\bar{e}_h:=\frac{(T_h^{*t})^{-1}e_n}{|(T_h^{*t})^{-1}e_n|}
\end{equation}
 is the unit inner norm of $U_h=T_h^*(U)$ at $0.$
By the definition of $T_h^*$,  we have $(T_h^{*t})^{-1}=hT_h.$
By passing to a subsequence, we may assume $\bar{e}_h\rightarrow e_0$ as $h\rightarrow 0$.   
Then we have the following nice property.

\begin{lemma}\label{nice1}
The hyperplane
$e_0^{\perp}:=\{x\in \mathbb{R}^n: x\cdot e_0=0\}$ is the supporting hyperplane of $U_0$ at $0$.
\end{lemma}

\begin{proof}
Let $y\in T_h(S_h[v])$. Then $T_h^{-1}(y)\in S_h[v]$, and 
by Corollary \ref{coroabove}, we have
\begin{equation}\label{ppn1}
	Dv(T_h^{-1}y)\cdot e_n\geq -Ch^{1-\epsilon}.
\end{equation}
By Remark \ref{shapelev} and \eqref{ddes}, there exists a constant $c$ independent of $h$ such that
for any $x\in B_c(0)\cap T_h^*(U),$ there exists $y\in  T_h(S_h[v])$ such that $x=Dv_h(y)$. %}
Then from \eqref{Dvhy}, 
\begin{equation}\label{ppn2}
	x = \frac{1}{h}(T_h^t)^{-1}Dv(T_h^{-1}y).
\end{equation}
Combining \eqref{ppn1}, \eqref{ppn2} together with \eqref{faxiang2}, we obtain
\begin{equation}\label{ppn3}
	x\cdot \bar{e}_h\geq -\frac{Ch^{1-\epsilon}}{|hT_he_n|}.
\end{equation}  
By the arbitrariness of $x$,
it suffices to show that the right hand side of inequality \eqref{ppn3} converges to $0$, as $h\to0$.  
Recall that $T_h=T_2\circ T_1$. From \eqref{T1}, we have $T_1e_n = h^{-\frac13}e_n$. 
From \eqref{ag2}, we also have $|T_he_n| \gtrsim h^{-\frac13+\epsilon}$. 
Therefore, by \eqref{ppn3} we infer that 
\begin{equation}\label{U02}
x\cdot \bar{e}_h\geq -Ch^{\frac13-2\epsilon} \rightarrow 0,
\end{equation}
 as $h\rightarrow 0$.

Now, for almost everywhere $x\in U_0\cap B_c(0),$ by Remark \ref{U01}, we can find $y\in V_0$ such that 
$x=Dv_0(y).$ Since $V_h\cap B_c(0)$ converges to $V_0\cap B_c(0)$ in Hausdorff distance, we have that 
$y\in V_h\cap B_c(0),$ provided $h$ is sufficiently small. Hence by \eqref{U02}, we have that $Dv_h(y)\cdot \bar{e}_h\geq0.$
By Remark \ref{U01} we have that, up to a subsequence, $Dv_h(y)\rightarrow Dv_0(y)=x.$
Hence, passing to limit, we have that $x\cdot e_0\geq 0.$
By continuity, we have that $x\cdot e_0\geq0$ for all $x\in U_0\cap B_c(0).$ 
Hence, by the convexity of $U_0$ in Lemma \ref{split2}, we reach the conclusion of  Lemma \ref{nice1}.
\end{proof}

From the definitions \eqref{faxiang1} and \eqref{faxiang2}, one can verify that $\bar{e}_h \perp \tilde{e}_h$ for any $h>0$. 
Passing to the limit we have 
\begin{equation}\label{cuizhi}
e_0\perp e_0^*, 
\end{equation}
where $e_0^*$ is the unit inner normals of $\partial V_0$ at $0$, and $e_0$ is the same as that in Lemma \ref{nice1}.
We remark that despite  the decompositions $U_0=\omega_0\times H_0$ in Lemma \ref{split2} and $V_0=\omega_0^*\times H^*_0$ in Lemma \ref{split1}, the $n-2$ dimensional subspaces $H_0, H_0^*$ may differ from each other, see Fig. \ref{f2H}. 
The next lemma says that we can align them by an affine transformation.

\renewcommand{\figurename}{Fig.}
\renewcommand{\captionlabeldelim}{}
\begin{figure}[h]
	\centering
	\includegraphics[width=0.75\textwidth]{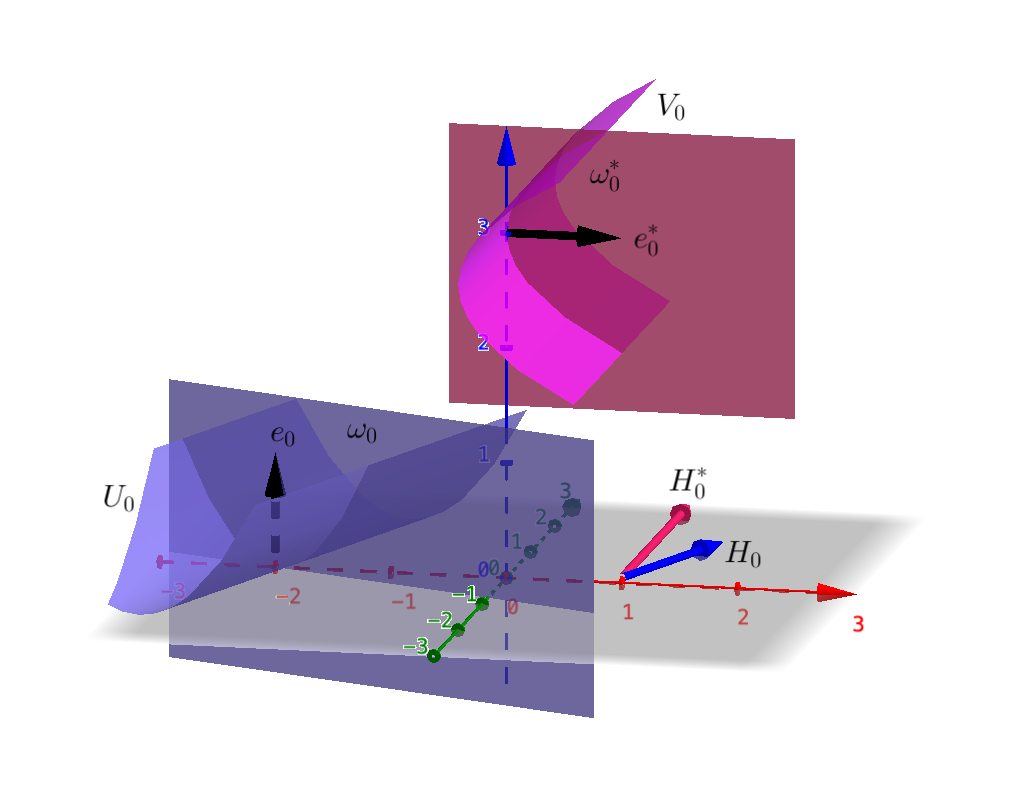}
	\caption{}
    \label{f2H}
\end{figure}

\begin{lemma}\label{hkey6}
There exists an affine transformation $A$ with $\det\, A=1$ such that $AH_0=(A^t)^{-1}H_0^*$.  
Hence,  by the affine transform $A$ and another coordinate change, 
we can make both $A(U_0)$ and $(A^t)^{-1}(V_0)$ flat in the $e_2,\cdots, e_{n-1}$ directions.
\end{lemma}

\begin{proof}
We first \emph{claim} that for any unit vector $e\in H_0$, $e$ cannot be parallel to $e_0^*$. 
For if not, then $e_0^*\in H_0$.  
Let $u_0$ be the Legendre transform of $v_0,$
namely, 
	\beq\label{v0dual}
		u_0(x)=\sup_{y\in \mathbb{R}^n} \left\{ x\cdot y-v_0(y) \right\} \quad \text{for}\ x\in \overline{U_0}.
	\eeq

By Lemma \ref{u0v0}, we have that 
 $u_0$ is strictly convex and $C^{1}$ in $B_{r_0}(0)\cap \overline{U_0}$ for some $r_0>0.$
  Note that since $v_0(0)=0$, {$v_0\geq 0$}, we also have $u_0(0)=0$, $u_0\geq 0$.
 On the other hand, by \eqref{mapinside1}
  $Du_0(U_0\cap B_r(0)) \subset \{y: y\cdot e_0^*\geq 0\}$,  we have
 $Du_0\cdot e_0^*\geq 0$ in $U_0\cap B_r(0),$ namely $u_0$ is monotone increasing in the $e_0^*$ direction.  
 It follows that $u_0(-te_0^*)\leq 0$ for $t>0$ small, and thus $u_0(-te_0^*)=0$ for $t>0$ small, which contradicts to the strict convexity of $u_0$.
The claim is proved.

Now, for a fixed unit vector $e\in H_0$,  by the above claim we can find a vector $\tilde{e}\in H_0^*$ such that $e$ is not orthogonal to $\tilde{e}$.   Hence there exists an affine transformation $A_1$ with $\det A_1=1$ such that $A_1e$ is parallel to $(A_1^t)^{-1}\tilde{e}$ (see \eqref{matA} and \cite[(4.7)]{CW1}).   
The unit inner normals of $A_1(U_0)$ and $(A_1^t)^{-1}V_0$ at $0$ are still orthogonal to each other. 
Denote $\bar{e}_2=\frac{A_1e}{|A_1e|}$.   
 Then, $A_1(U_0)= \omega_1\times H_1\times \text{span}\{\bar{e}_2\}$,  where $\omega_1$ is a two dimensional convex subset and $H_1$ is an $n-3$ dimensional subspace in $\mathbb{R}^n$. 
 Similarly,  $(A_1^t)^{-1}V_0 = \omega^*_1\times H^*_1\times \text{span}\{\bar{e}_2\}$, where $\omega^*_1$ is a two dimensional convex subset and $H^*_1$ is an $n-3$ dimensional subspace in $\mathbb{R}^n$.

Then we restrict ourself to the sets $\omega_1\times H_1$ and $\omega^*_1\times H^*_1$ in the $(n-1)$-space $(\bar{e}_2)^\bot$.
Similarly as above, we can find unit vectors $e'\in H_1, \tilde{e}' \in H^*_1$ and an affine transform $A_2$ such that $A_2e'$ is parallel to $(A_2^t)^{-1}\tilde{e}'$,  and $A_2\bar{e}_2=\bar{e}_2$ remains unchanged. 
Let $\bar{e}_3=\frac{A_2e'}{|A_2e'|}$.   Repeating this process, after a sequence of affine transformations $A_i$, $i=1,\cdots, n-2$,  we have $AH_0=(A^t)^{-1}H_0^*$,  where $A=A_{n-2}\cdots A_1.$ 
\end{proof}

\begin{proof}[Proof of Proposition \ref{blowuppic} when $n\geq 3$]
By Lemma \ref{split1}, Lemma \ref{split2}, Lemma \ref{hkey6}
and the relation \eqref{cuizhi}, up to an affine transformation and a change of coordinates we may assume $V_0=\omega_0^*\times H$ and $U_0=\omega_0\times H,$ where $H=\text{span}\{e_2,\cdots,e_{n-1}\},$ and 
\begin{equation}\label{funcv0}
\omega_0^*=\{(y_1, y_n): y_1\geq \rho_0^*(y_n)\}
\end{equation}
  for some smooth convex function $\rho_0^*$ satisfying $\rho_0^*\geq 0$, $\rho_0^*(0)= 0.$  
  Meanwhile, $\omega_0$ is a convex set in $\text{span}\{e_1, e_n\}$ with 
$0\in \partial \omega_0$ and $\omega_0\subset \{(x_1,x_n) : x_n\geq 0\}.$ 
However, $\partial\omega_0$ may not be a graph of a function of $x_1$, for example see Fig. \ref{figsliding}.
To make $\partial\omega_0$ locally a graph, we can apply a sliding transform as follows.

\renewcommand{\figurename}{Fig.}
\renewcommand{\captionlabeldelim}{}
\begin{figure}[h]
	\centering
	\includegraphics[width=0.7\textwidth]{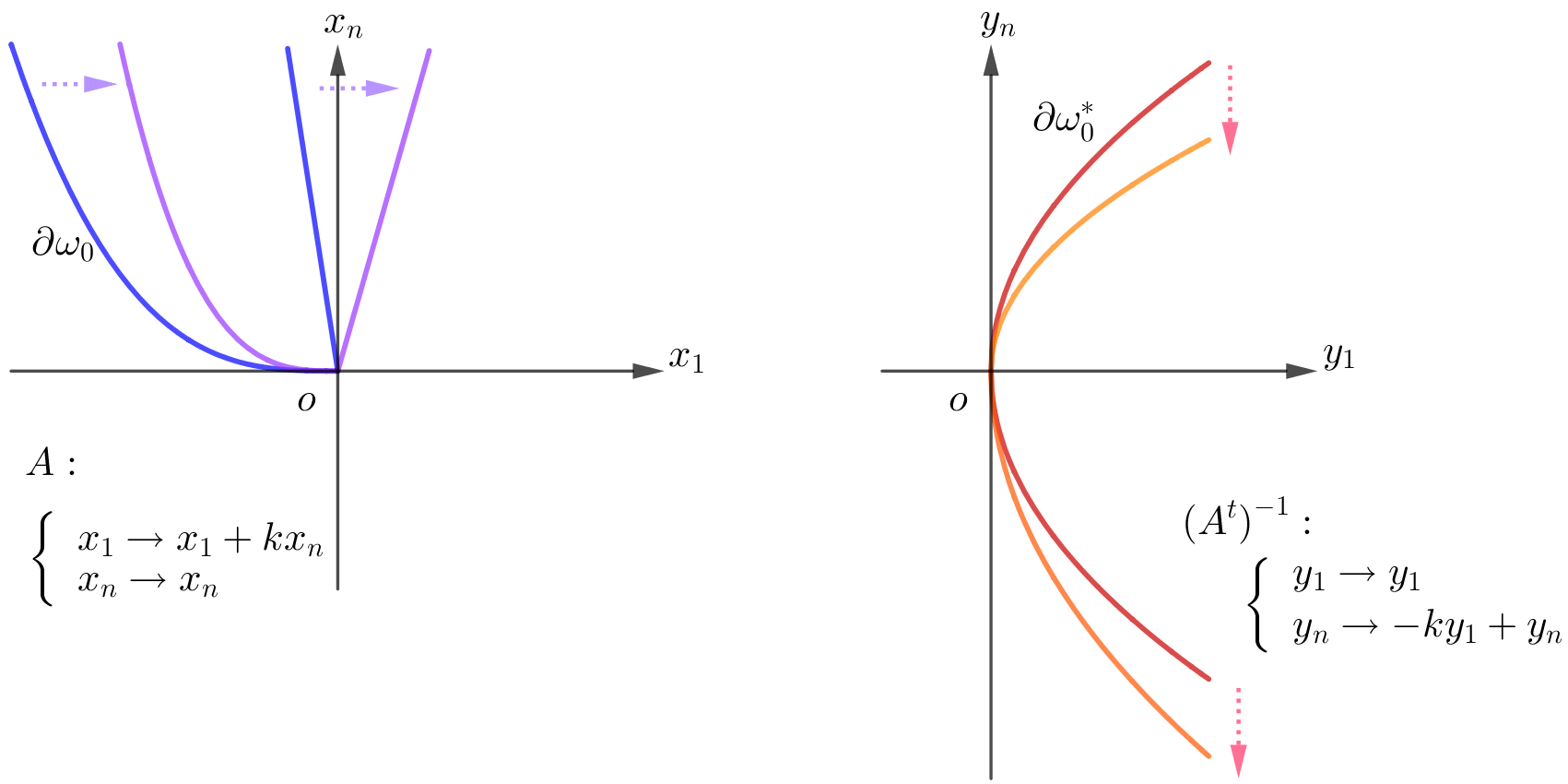}
	\caption{}
    \label{figsliding}
\end{figure}

Let $A$ be an affine transform such that
\begin{equation*}
	A: \left\{ \begin{array}{rll}
		x_1 &\to\ \  x_1 + k x_n & \text{for a constant $k\in\R$} \\
		x_i &\to\ \  x_i\quad & \text{for } i=2,\cdots,n.
	\end{array}
	\right.
\end{equation*}
Note that $A$ makes $U_0$ to slide along the $x_1$ direction, and at the same time $(A^t)^{-1}$ makes $V_0$ slide along the $y_n$ direction, while the $(n-2)$-space $H$ remains invariant. 
Hence, by choosing a proper constant $k\in\R$, we may assume that $\omega_0=\{(x_1, x_n): x_n\geq \rho_0(x_1)\}$ for a convex function $\rho_0.$ Note that since $\rho_0^*$ is smooth, after the corresponding affine transform $(A^t)^{-1},$ $\omega_0^*$ still satisfies \eqref{funcv0} but with a different smooth function $\rho_0^*.$ 
\end{proof}

\begin{remark}\label{lastuse}
By the proof of Lemma \ref{split1}, after the above transform, we have that $\partial V_0=\{y_1>P_0(y)\}$ for some nonnegative quadratic polynomial. Since $\partial V_0$ is flat in $e_2, \cdots, e_{n-1}$ directions, we have that $P_0$ depends only on $y_1, y_n.$
Since $P_0$ is nonnegative, we may denote it as 
$P_0(y)=ay_n^2+2by_ny_1+cy_1^2$ with $b^2\leq ac.$
 We claim that $b=c=0$ and $a>0.$
In fact, if $c>0,$ then $P_0(te_1)>y_1$ for $t$ large, which implies $te_1\notin V_0.$ On the other hand, by
\eqref{epigraph1} we have that $te_1\in V_0$ for any $t$ large, which is a contradiction. Hence, $c=0,$ which also implies 
$b=0.$ 
If $a=0,$ then $\partial V_0$ is flat, which implies that $te_n\in \partial V_0$ for $t<0.$ Since 
$Dv_0(\mathbb{R}^n)\subset \{x: x_n\geq 0\},$ which implies that $v_0$ is increasing in $e_n$ direction. Since $v_0\geq 0$ 
and $v_0(0)=0,$ it implies that $v_0(te_n)=0$ for all $t<0,$ contradicting to the strict convexity of $v_0$ in $B_1(0)\cap \overline{V_0}.$  Therefore, we may denote $\partial V_0:=\{y: y_1=\rho^*_0(y_n)\}$ with $\rho^*_0(y_n)=ay_n^2$ for some positive constant $a.$
\end{remark}

\section{Proof of obliqueness}\label{Sob}
In this section we will use the limit profile obtained in Section \ref{Sbu} to prove the following obliqueness estimate.

\begin{proposition}\label{oblique1}
Assume that  $\overline\Om, \overline{\Om^*}\subset\mathbb{R}^n$ are disjoint, uniformly convex domains with $C^2$ boundaries,
and that the densities $f\in C(\bom)$, $g\in C(\overline{\Omega^*})$ are positive, continuous functions.
 Then for any $x_0\in\mathcal{F}$ and $y_0=Du(x_0)$, we have
\beq \label{obq-n}
\nu_{_U} (x_0)\cdot \nu_{_V} (y_0)>0,
\eeq
where $\nu_{_U}(x_0)$ is the unit inner normal of $U$ at $x_0$, and $\nu_{_V}(y_0)$ is the unit inner normal of $V$ at $y_0$. 
\end{proposition}

\subsection{Obliqueness in dimension two}\label{S3}

In the argument below, we will adopt some techniques from \cite{CLW1}.
Recall that if the obliqueness fails at $x_0,$ then Proposition \ref{blowuppic} holds.
Let $v_0, U_0, V_0$ be as in Proposition \ref{blowuppic}:
	\begin{equation}\label{ang321}
	 V_0 =\big\{(y_1,y_2)\in\R^2 \,:\, y_1> \rho_0^*(y_2) \big\}, 
	 \end{equation}
	 where $\rho_0^*(t)=at^2$ for some constant $a>0$, and  
	\begin{equation}\label{ang322}
	 U_0 = \big\{(x_1,x_2)\in\R^2 \,:\, x_2> \rho_0(x_1) \big\}, 
	 \end{equation}
	 where $\rho_0$ is a convex function satisfying $0\leq \rho_0(t)\leq Ct^2$ for a constant $C>0$,
	and $\rho_0(t)=\frac{1}{2r}t^2$ for $t<0$,
	where $r>0$ is a constant. 
By subtracting a constant we may assume that
$v_0(0)=0.$

 Recall that by \eqref{Dvmap}, we have that $Dv_0(\partial V_0\cap B_{r'}(0))\subset \partial U_0.$
Then
by the  monotonicity of convex function $v_0$  we  have
\begin{equation}\label{bdrypt}
Dv_0(y)\in \partial U_0\cap \{x_1<0\}\quad \ \forall \, y\in \partial V_0\cap \{y_2>0\}\cap B_{r'}(0). 
\end{equation}
Indeed, given any $y\in  \partial V_0\cap \{y_2>0\}\cap B_{r'}(0),$ suppose $x=Dv_0(y)\in  \partial U_0\cap \{x_1\geq 0\}.$ Let $\{\tilde x\in \mathbb{R}^2: (\tilde{x}-x)\cdot e= 0\}$ be a supporting line of the convex set $U_0$ at $x\in \partial U_0,$ for some unit vector $e.$ Replacing $e$ by $-e$ if necessary, we may also assume that $U_0\subset \{\tilde x\in \mathbb{R}^2: (\tilde{x}-x)\cdot e> 0\}.$ Note that $e$ can be chosen as the unit inner normal vector of $\partial U_0$ at $x$ when $\partial U_0$ is $C^1$ at $x.$ 
Then, by \eqref{ang321}, \eqref{ang322} and using the assumptions that $y_2>0$ and $x_1\geq 0,$
 we have that the angle between $e$ and the unit inner normal of $V_0$ at $y$ is strictly large than $\frac{\pi}{2}.$ Hence,
by the smoothness of $\partial V_0$ we have that  $-e$ points inside $V_0,$ namely $y-te\in V_0$ for $t>0$ small.
Denote $x_t:=Dv_0(y-te)\in U_0\subset \{\tilde x\in \mathbb{R}^2: (\tilde{x}-x)\cdot e> 0\}.$ Then,
$$\left(Dv_0(y-te)-Dv_0(y)\right)\cdot (y-te-y)=(x_t-x)\cdot (-e)<0,$$ contradicting to the monotonicity of $Dv_0.$

By ($i$) of Proposition \ref{blowuppic}, we have that both $\partial U_0\cap \{x_1<0\}$ and $\partial V_0\cap \{y_2>0\}$ are smooth and uniformly convex.
 Hence by the localised estimates of Caffarelli \cite{C96}, 
$v_0$ is smooth up to the boundary in $V_0\cap\{y_2>0\}$.
Let $p, \xi$ be the points on $\partial S_h[v_0]$ such that 
\begin{align}\label{defp}
   p_2 = p\cdot e_2 &=\sup\{y\cdot e_2 \,:\, y\in S_h[v_0]\},\\
   \xi_2 = \xi\cdot e_2 &=\inf\{y\cdot e_2 \,:\, y\in S_h[v_0]\} . \nonumber
\end{align}
 From \eqref{bdrypt}, one sees that $p$ is in the interior of $V_0$.
Hence
$\{x\in\mathbb{R}^2:x_2=p_2\}$ is the tangent line of $\{v_0<h\}$ at $p.$
We \emph{claim} that 
\begin{equation}\label{balan1}
p_2\geq C|\xi_2|
\end{equation}
for a constant $C>0$ independent of $h$. 
The proof of \eqref{balan1} is similar to that of \cite[Lemma 4.1]{CLW1}. 
For the reader's convenience, we include a brief proof below.

Suppose \eqref{balan1} is not true, then there exists a sequence $h\rightarrow 0$ such that
\begin{equation}\label{balan2}
\frac{p_2}{\xi_2}\rightarrow 0\quad \text{as }\ h\rightarrow 0.
\end{equation}
Let $T_h$ be a linear transformation such that $T_h(S_h[v_0])\sim B_1$, and let 
$v_{0h}(\cdot)=\frac{1}{h}v_0(T_h^{-1}(\cdot)).$ 
 Similarly to $v_h$,  $v_{0h}$ sub-converges to a convex function $\bar{v}$ locally uniformly as $h\to0$. 
Denote $H_{1h}=T_h\left(\{x_2=p_2\}\right)$ and $H_{2h}=T_h\left(\{x_2=\xi_2\}\right).$ 
By \eqref{balan2} we have 
$$\frac{\text{dist}(0, H_{1h})}{\text{dist}(0, H_{2h})} \rightarrow 0\quad \text{as }\ h\rightarrow 0.$$
Along a subsequence, $H_{1h}$ and $H_{2h}$ converge to straight lines $H_{1}$ and $H_2$, respectively. 
Since $T_h(S_h[v_0])$ has a good shape, we have $\text{dist}(H_{1h}, H_{2h})\approx 1$.
Then the limit $H_1$ passes $0$.
On the other hand, since $H_{1h}$ is a tangent line of $\{v_{0h}=1\}$, we have $v_{0h}\geq 1$ on one side of $H_{1h}$. 
Passing to the limit, we have $\bar{v}\geq 1$ on one side of $H_1$, which however contradicts 
to the facts that $0\in H_1$, $\bar{v}(0)=0$ and $\bar{v}$ is continuous. 
Hence claim \eqref{balan1} is proved.

Recall that  $Dv_0(V_0)\subset \{x_2\geq 0\}$.
Hence $v_0$ is increasing in $y_2$, and $\sup\{y\cdot e_1 : y\in S_h[v_0]\}$ is achieved  at $\xi$,
the point  defined in \eqref{defp}.
That is 
	$$\xi_1=\sup\{y\cdot e_1 \,:\, y\in S_h[v_0]\}.$$ 
From \eqref{nm1}, \eqref{balan1} and noting that $\xi\in\partial V_0=\{y_1=ay_2^2\}$, 
we have the estimates 
\begin{align*} 
 	h \approx |S_h[v_0]| & \leq Cp_2\xi_1  \\
	 		 & \leq Cp_2\xi_2^2    \leq Cp_2^3 .
\end{align*}
It implies that $p_2 \gtrsim h^{1/3}$.   Therefore, as $p\in\partial S_h[v_0]$, we obtain 
  $$v_0(p)=h\leq Cp_2^3.$$

Denote $\tilde{p}=(p_1,\frac{1}{2}p_2).$
Since $v_0$ is increasing in the $e_2$ direction, we have $v_0(\tilde p) \leq v_0(p)$ and
\begin{equation}\label{gg1}
v_0(\tilde{p})\leq h\leq Cp_2^3.
\end{equation}
By the convexity of $v_0$,  
\begin{equation}\label{gg2}
	\partial_2 v_0(\tilde{p})\leq \frac{v_0(p)-v_0(\tilde{p})}{\frac{1}{2}p_2}\leq C\frac{h}{p_2}\leq Cp_2^2,
\end{equation}
where $\partial_2 v_0=\partial_{y_2} v_0 \geq 0$. 

Introduce the function
	\begin{equation}\label{aux}
		w(y) := \partial_2v_0(y) + v_0(y)- y_2 \partial_2 v_0(y) \quad\mbox{ in }  V_0. 
	\end{equation}
By equation \eqref{maeq1},  $w$ satisfies 
$$\sum_{i,j=1}^2 V^{ij}w_{ij}=0\ \ \ \text{ in }  V_0, $$
where $\{V^{ij}\}$ is the cofactor matrix of $\{D^2v_0\}$.

\begin{lemma}\label{decay001}
Let
	\begin{equation}\label{aux2}
		\underline{w}(t):= \inf \big\{w(y_1, t)  \,:\,  y_1 > \rho_0^*(t)\big\},\quad 0<t<1.
	\end{equation}
Then for $t>0$ small, say $t\in (0, \delta_0)$, we have
\beq
0 \leq \underline{w}(t)\leq Ct^2.
\eeq  
\end{lemma}

\begin{proof}
Observe that $w=(1-y_2) \partial_2v_0 + v_0 \geq 0$ for $y_2<1.$ 
Let $p$ be the point defined in \eqref{defp}. 
By \eqref{gg1} and \eqref{gg2}, we have
$$\underline{w}(\frac{1}{2}p_2)\leq w(\tilde{p})\leq Cp_2^3+Cp_2^2\leq 2Cp_2^2,$$
for $p_2>0$ small. By sending $h\rightarrow 0,$ $p_2$ will take all arbitrarily small positive values, hence the desired estimate follows. 
\end{proof}

\begin{lemma}\label{min1}
For $t>0$ small, the minimum of $w(\cdot, t)$ in \eqref{aux2} is attained in the interior of $V_0.$
\end{lemma}

\begin{proof}
Recall that $v_0$ is smooth up to the boundary in $V_0\cap \{y_2>0\},$ and
$$  \partial V_0\cap\{y_2>0\}= \left\{(y_1,y_2) \,:\, y_1=\rho_0^*(y_2)=ay_2^2,\ y_2>0 \right\}.$$  
For $y=(y_1, y_2)\in   \partial V_0\cap\{y_2>0\}$,
by \eqref{bdrypt} and \eqref{limU0}, we have
 $$Dv_0(y)\in \Big\{(x_1, x_2) \,:\, x_2= \rho_0(x_1) =\frac{1}{2r}x_1^2,\ x_1<0 \Big\}.$$
Hence  
$$\partial_2v_0(\rho_0^*(t), t)=\rho_0(\partial_1v_0(\rho_0^*(t),t))\quad \ \text{ for }\ t>0.$$ 
Differentiating the above equation in $t$, we obtain
$$\partial_{21} v_0 \cdot\left((\rho_0^*)'-\rho_0'\right)=\rho'_0(\rho_0^*)' \partial_{11}v_0 - \partial_{22}v_0.$$
Since $(\rho_0^*(t))' > 0$,  $\rho_0' (\partial_1v_0(\rho_0^*(t),t))<0$ for $t>0$,  and $\partial_{11}v_0>0$, $\partial_{22}v_0>0$,
from the above formula it follows that $\partial_{21}v_0<0$ for $t>0$. 
Hence for $y=(\rho_0^*(y_2), y_2)\in \p V_0$ with $0<y_2<1$, we obtain
$$\partial_{1} w(y)=(1-y_2)\partial_{21}v_0+\partial_1v_0<0.$$

On the other hand,
recall that  $\partial_2 v_0 \ge 0$ and $v_0 \ge 0$. 
For any small $\delta>0$, by the strict convexity of $v_0$ in $V_0$,   
there exists $\epsilon>0$ such that  
$$w(y)=(1-y_2)\partial_2 v_0+v_0 \geq \epsilon\ \ \ \text{for}\ y\in B_1(0)\backslash B_{\delta}.$$
 By the assumption in the beginning of Section 5, we have that $v(0)=0, v\geq 0,$ which implies that
$v_h(0)=0, v_h\geq 0,$ passing to limit $h\rightarrow 0,$ we have $v_0(0)=0, v_0\geq 0.$ Hence $Dv_0(0)=0.$ Note
that by Lemma \ref{u0v0}, we have that $v_0,$ as a convex function defined on $\mathbb{R}^n,$ is differentiable at $0.$
By the definition of $w,$ we have $w(0) = \partial_2v_0(0) + v_0(0)- y_2 \partial_2 v_0(0)=0.$
Hence, by the $C^{1}$ regularity of $v_0$,
there exists $\delta_0>0$,  such that $w(\cdot, t)$ attains its minimum in the interior of $V_0$ 
for any $0<t<\delta_0.$
\end{proof}

\begin{lemma}\label{concave}
For $t\in (0, \delta_0)$,  the function $\underline{w}$ defined in \eqref{aux2} is concave,
\end{lemma}

\begin{proof}
If $\underline{w}$ is not concave, 
there exist constants $0<r_1<r_2<\delta_0$  and
 an affine function $L(t)$ such that 
$\underline{w}(r_i)= L(r_i)$ for $i=1,2$,  and the set 
  $\{t\in (r_1, r_2) : w(t)< L(t)\}\ne\emptyset$.
Extend $L$ to $\R^2$ such that $L(s, t)=L(t)$, namely, $L$ is independent of $s$.
Denote 
$$D_\eps =\{y\in V_0 \ :\ y_2\in(r_1, r_2),\text{ and } w(y)< L(y)-\eps\} . $$
By our definition of $\underline w$ and Lemma \ref{min1}, 
we can choose $\eps>0$ such that 
\beq\label{Deps}
\emptyset\ne D_\eps \Subset V_0.
\eeq
Indeed, by our choice of $L$, $D_{\eps\,|\,\eps=0}\ne\emptyset$.
Let $\eps_0=\sup\{\eps \,:\, D_{\eps}\ne\emptyset\}$.
Then \eqref{Deps} holds for $\eps<\eps_0$ and sufficiently close to $\eps_0$.

Recall that $\sum_{i,j}V^{ij}w_{ij}=0$ in $V_0$.
The strong maximum principle implies that $w= L$ in $D_\eps$.
However, $w<L$ in $D_\eps$ by our definition of $D_\eps$.
We reach a contradiction.
\end{proof}

\begin{proof}[Proof of Proposition \ref{oblique1} in 2d]
Suppose the obliqueness fails.
By Lemma \ref{decay001} and Lemma \ref{concave}, 
$\underline{w}(t)$ is concave in $(0, \delta_0)$ and satisfies $0\leq \underline{w}(t)\leq Ct^2$.   
Note that $\underline{w}(t)\to 0$ as $t\to 0$.
Hence, we must have $\underline{w}(t)\equiv 0$ for $t\in (0, \delta_0)$.
On the other hand, for a fixed $t_0\in(0,\delta_0)$,
by the strict convexity of $v_0$, we have 
$w(y_1,t_0)=(1-t_0)\partial_2 v_0 + v_0 >\epsilon_0$ for any $(y_1,t_0)\in \overline{V_0}$, where the constant $\epsilon_0>0$ is independent of $y_1$. Therefore, $\underline{w}(t_0)\geq \epsilon_0>0$. 
We reach a contradiction. 
\end{proof}

\subsection{Obliqueness in higher dimensions}\label{S4}
Suppose the obliqueness fails at $x_0,$ let $v_0, U_0, V_0$ be as in Proposition \ref{blowuppic}.
 When $n\geq 3,$ since $U_0$ is not $C^{1,1}$ in general, we do not have the $C^2$ regularity of $v_0$ up to $\partial V_0\cap\{y_n>0\}$ as that in dimension 2. Hence,
in the proof we need to use the approximation technique developed in \cite[Section 5.2]{CLW1}.

\begin{proof}[Proof of Proposition \ref{oblique1} for general dimensions] \ 

\noindent {\it Step 1.} 
By Proposition \ref{blowuppic},  we may assume that
\begin{align}
\partial U_0 &=\{x : x_n=\rho_0(x_1)\}; \label{eqlem} \\
\partial V_0 &=\{y : y_1=\rho_0^*(y_n)\} \nonumber
\end{align} 
for a convex function $\rho_0$ satisfying $\rho_0(0)=0$, $\rho_0\geq0$; 
and for a smooth convex function $\rho_0^*$ satisfying $\rho_0^*(0)=0$, 
$\rho_0^*\geq 0$. 

We remark that the smoothness of $\rho_0^*$ follows from Lemma \ref{split1},  
but the function $\rho_0$ may not be smooth. 
Unlike \eqref{limU0} in dimension two, {the lack of smoothness of $\rho_0$}
prevents us from obtaining further regularity of $v_0$. By \eqref{maeq111},  $v_0$ satisfies
\begin{align}\label{maeq666}
	& \det\, D^2v_0=c_0\chi_{_{V_0}}\ \ \text{in}\ \R^n,\\
	& Dv_0(V_0)=U_0 \nonumber
\end{align}
for a constant $c_0>0$. 
To overcome this obstacle, in the following we first show that $v_0$ can be approximated by a sequence of smooth functions.

Fix a small $r_0>0$, let $\tilde V_0$ be interior of the convex hull of $\Sigma:=Du_0(B_{r_0}\cap \overline{U_0})$, where $u_0$ is as in \eqref{v0dual}.
By the proof of Lemma \ref{u0v0}, in particular \eqref{mapinside1} and \eqref{Dumapto}, we have that
 \begin{equation}\label{inclus12}
 B_{\delta}(0)\cap \overline{V_0}\subset \Sigma \subset \overline{V_0},
 \end{equation}
  for $\delta$ small.
Now, by \eqref{inclus12} and convexity of $V_0,$ when we take convex hull of $\Sigma,$ the part $B_{\delta}(0)\cap \overline{V_0}$ is not changed.
Therefore,
we have that
\beq\label{localapp}
	\tilde V_0\cap B_\delta(0)=\Sigma \cap B_\delta(0)=V_0\cap B_\delta(0)
\eeq
when $\delta>0$ is small.

Approximating $\rho_0$ by smooth convex functions $\rho_k,$ we can approximate $B_{r_0}\cap U_0$ in Hausdorff distance by a sequence of convex set $U_k:=\{x:\ x_n> \rho_k(x_1)\} \cap B_{r_0}$,  which is smooth near $0$, such that for each $k$, 
$$\partial U_k\cap B_{r_0}=\{ x \,:\, x_n=\rho_k(x_1)\}$$ 
for a convex, smooth function $\rho_k$ satisfying $\rho_k(0)=0$, $\rho_k\geq0$, 
and $\rho_k'(t)<0$ when $t<0$; 
and such that $\rho_k\to\rho_0$ locally uniformly as $k\to\infty$. 
Now, let $v_k$ be the convex function solving 
$$
(Dv_k)_{\#}(c_k\chi_{_{\Sigma}}+\frac{c_k}{k}\chi_{_{\tilde V_0\backslash\Sigma}})=\chi_{_{U_k}} ,  $$ 
where the constant 
$$c_k= \frac{|U_k|}{|\Sigma|+\frac{1}{k}|\tilde V_0\backslash\Sigma|}  
                  \rightarrow c_0\ \ \text{as}\  k\rightarrow \infty .$$
   By the definition of $U_k$ and the fact that the convex fucntions $\rho_k\to\rho_0$ locally uniformly as $k\to\infty,$
   we can deduce that $|U_k|$ converges to $|B_{r_0}\cap U_0|.$

Then by \eqref{localapp} and subtracting a constant if necessary,  
we have that $v_k\rightarrow v_0$ uniformly in $B_{r_1}(0)\cap \overline{V_0}$ as $k\rightarrow\infty$, 
for some $r_1<r_0$ independent of $k$.

We also extend $v_k$ to $\mathbb{R}^n$ as follows
$$v_k(x):=\sup\{L(x) \,:\, L\ \text{is affine}, L\leq v_k\ \text{in}\ \tilde V_0, \ \text{and} \ DL\in U_k\}$$
for any $x\in \mathbb{R}^n.$
By subtracting a constant, we may assume $v_k(0)=0.$ Since $$\|Dv_k\|_{L^\infty(\mathbb{R}^n)}\leq \text{diam} (U_k) \leq r_0,$$
up to a subsequence, we may assume $v_k$ converges to a convex function $\tilde{v}_0$ locally uniformly. Now, by weak convergence of Monge-Amp\`ere measure we have
$\det D^2\tilde{v}_0=c_0\chi_{\Sigma}$ in $\mathbb{R}^n.$ 
 Moreover, $D\tilde{v}_0$ is the optimal map from $\Sigma$ to 
$B_{r_0}(0)\cap U_0.$ By uniqueness of optimal maps we have that $\tilde{v}_0=v_0$ in $V_0\cap B_\delta(0).$
 Since $v_0$
is differentiable at $0$ (follows from Lemma \ref{u0v0}), we have that $\partial v_0(B_{\delta}(0))\subset B_{r_0}(0)\cap \overline{U_0},$ provided $\delta$ is small enough. This implies that $v_0=\tilde{v}_0$ in $B_{\delta}(0).$ 
Since $v_k\rightarrow v_0$ uniformly in $B_{\delta}(0)$ and
 $v_0$ is differentiable at points in $B_{\delta}(0)\cap \overline{V_0}$ (follows from Lemma \ref{u0v0}), 
 by the argument in Remark \ref{U01}, we have that $Dv_k$ converges to $Dv_0$ uniformly in $B_{r_1}(0)\cap \overline{V_0}$ by choosing $r_1=\frac{\delta}{2}.$

Since $\partial U_k, \partial\tilde V_0$ are also smooth near $0,$ 
by the localised $C^{2,\alpha}$ estimate in \cite[Theorem 1.1]{CLW1}, 
$v_k$ is smooth in $B_{r_2}\cap \overline{V_0}$, for some $r_2>0 $ independent of $k$. Here $r_2<r_1$ is chosen small such that $Dv_k(B_{r_2}\cap \overline{V_0})\subset B_{\frac{r_0}{2}}(0)\cap \overline{U_k}.$ Since $Dv_k$ converges to $Dv_0$ uniformly in $B_{r_1}(0)\cap \overline{V_0},$ $v_0\in C^1(B_{r_1}(0)\cap \overline{V_0})$ and $Dv_0(0)=0,$ we can choose such $r_2$ uniformly for  all $k.$
Note that the statement of \cite[Theorem 1.1]{CLW1} is a global one, but the proof is actually a local one. Indeed, for any
$y\in B_{r_2}\cap \partial V_0,$ by the above discussion we have that $Dv_k(y)\in  B_{\frac{r_0}{2}}(0)\cap  \partial U_k.$
Since both   $B_{r_2}\cap \partial V_0$ and $B_{\frac{r_0}{2}}(0)\cap  \partial U_k$ are smooth, and densities are positive constants in $B_{r_2}\cap \overline{V_0}$ and $B_{\frac{r_0}{2}}(0)\cap \overline{U_k},$ by \cite[Lemma 3.1]{CLW1} we have the tangential $C^{1,1-\epsilon}$ estimate of $u_k$ holds at $y,$  then, by \cite[Section 5]{CLW1} we have that the obliqueness holds at points $y$ and $Dv_k(y).$ Finally by \cite[proof of Theorem 1.1, Section 6]{CLW1}, we have that $v_k$ is $C^{2,\alpha}$ smooth at $y.$
Therefore we obtain a smooth approximation sequence of $v_0$. Note that we only need to use the smoothness of $v_k$ in $B_{r_2}\cap \overline{V_0}$ for taking the second order derivative, but we do not need to use the bound of $C^2$ norm for $v_k.$

\noindent{\it Step 2}.\ 
Let $w(y):=\partial_nv_0(y)+v_0(y)-\frac{n}{2}y_n\partial_nv_0(y)$, and define 
$$\underline{w}(t) = \inf\{w(y_1, y_2,\cdots, y_{n-1}, t)  \,:\,  y_1 > \rho_0^*(t)\},\quad 0<t<1.$$ 
Replacing $v_0$ by $v_k$, we can also define $w_k$ and $\underline{w_k}$ in the same way. 
Note that  for a point $y=(\rho_0^*(y_n), y_2,\cdots, y_n)\in \partial V_0\cap B_{\delta}(0)$ with $y_n> 0,$ 
we have that $x=Dv_k(y)\in \partial U_k.$ Similar to the reason for \eqref{bdrypt}, we also have that $x_1<0.$
By the definition of $U_k,$ we have that $x_n=\rho_k(x_1),$ hence,
$$\partial_nv_k(\rho_0^*(y_n), y_2,\cdots, y_n)=\rho_k\left(\partial_1v_k(\rho_0^*(y_n), y_2,\cdots, y_n)\right).$$
Then similar to the computation in Lemma \ref{min1}, we can show that $\partial_{n1}v_k(y)< 0.$ 
Now, analogously to Lemmas \ref{min1} and \ref{concave}, one can verify that 
$\underline{w_k}(t)$ is a concave function in $(0, \delta_0)$ for some positive constant $\delta_0$ independent of $k$. 
Hence by passing to the limit,   $\underline{w}(t)$ is also concave in $(0, \delta_0)$.

Denote $\hat{U}_0=Dv_0(B_1(0)\cap V_0).$ 
By the strict convexity of $v_0$ in $\overline{V_0},$ we have $B_{r_1}(0)\cap U_0\subset \hat{U}_0$ for some small $r_1>0$.  
Hence $\hat{U}_0$ is locally convex near 0.
Let 
$$\tilde{u}_0(x):=\sup\{L(x): L\ \text{is affine}, \  L\leq u_0\ \text{in}\ \hat{U}_0, \ \text{and}\ DL\in B_1(0)\cap V_0\},
\ \ \ x\in \mathbb{R}^n,$$ 
where $u_0$ is the Legendre transform of $v_0$ as in \eqref{v0dual}. 
Then $\tilde u_0$ satisfies 
 \begin{equation}\label{uconvex}
\det\, D^2\tilde{u}_0=\frac{1}{c_0}\chi_{\hat{U}_0}\ \ \ \text{in}\ \R^n.
 \end{equation}
Since $\tilde{u}_0$ is strictly convex in $\overline{\hat U_0}$, and $B_{r_1}(0)\cap U_0\subset \hat{U}_0$,
we have $S^c_h[\tilde{u}_0]\cap U_0=S^c_h[\tilde{u}_0]\cap \hat{U}_0$ 
for $h$  small. 

Since $U_0$ is flat in $e_2,\cdots, e_{n-1}$ directions near $0$, 
the right hand side of \eqref{uconvex} is independent of $x_2,\cdots, x_{n-1}$ near $0$. 
By Pogorelov's interior second derivative estimate (see \cite[Corollary 1.1]{C96}),
$\tilde{u}_0$ is $C^{1,1}$ smooth in the $e_i$-direction near $0$, for $i=2,\cdots, n-1$.
Namely, $u_0(te_i)=\tilde{u}_0(te_i)\leq C_1t^2$ near $t=0$.  
Hence, for $i=2,\cdots, n-1$ and $y \in V_0$ close to $0$, 
  \begin{align*}
  v_0(y)&=u_0^*(y)\\
  &= \sup_{x\in \overline{U_0}}\left\{x\cdot y-u_0(x)\right\}\\
  &\geq \sup_{t\in(-1,1)}\left\{te_i\cdot y-C_1t^2\right\}\\
  &\geq  C_2y_i^2
  \end{align*}
  for a constant $C_2>0$.   Hence 
 \begin{equation}\label{contain11}
  S_h[v_0]\subset \left\{y\in \mathbb{R}^n \,:\, |y_i|\leq Ch^{\frac{1}{2}}, \ i=2,\cdots, n-1\right\}
  \end{equation}
  for some constant $C$ independent of $h.$

\noindent{\it Step 3}.\ 
We introduce the points $p, \xi, q \in \partial S_h[v_0]$ such that 
 \begin{align*}
 p_n &= \sup\{y_n \,:\, y\in S_h[v_0]\}, \\
 \xi_n &= \inf\{y_n \,:\, y\in S_h[v_0]\}, \\
 q_1 &= \sup\{y_1 \,:\, y\in S_h[v_0]\}.
 \end{align*}
 Similarly to the proof of \eqref{balan1} (see also \cite[Corollary 5.1]{CLW1}), we have $p_n\geq C|\xi_n|.$
By \eqref{contain11}, 
% and \eqref{equi0
$S_h[v_0]$ is contained in a cuboid, that is
\begin{equation}\label{bbox}
	 S_h[v_0] \subset [0, q_1]\times[-Ch^{\frac12}, Ch^{\frac12}]^{n-2}\times[-Cp_n, Cp_n]. 
\end{equation}
Since $Dv_0(V_0)\subset \{x_n\geq 0\}$,  the function $v_0$ is monotone increasing in the $e_n$-direction, 
which implies  $q\in \partial V_0$.    Hence, from \eqref{eqlem},
$$q_1=\rho_0^*(q_n)\leq Cq_n^2 \leq Cp_n^2.$$
%where $\tilde q=(q_2,\cdots,q_{n-1})$.
%{\Small\color{blue} (In Proposition \ref{oblique1}, 
%%$\tilde x=(x_2, \cdots, x_{n-1})$. The notations $x', \tilde x$ etc are inconsistent in the paper.
%We should fix the notations by using $x', \tilde x$ or $\hat x$ throughout the paper, 
%and make clear at the beginning of \S2)}{\color{red} we unified the notations at the beginning of \S4, as there is no confusion in \S2 and \S3.}
From \eqref{bbox} and the volume estimate \eqref{nm1}, we have
$$h^{\frac{n}{2}}\approx  |S_h[v_0]|\leq Ch^{\frac{1}{2}(n-2)}p_nq_1 \leq Ch^{\frac{1}{2}(n-2)}p_n^3, $$   
which implies $p_n\geq Ch^{1/3}$. 
It then follows, analogously to \eqref{gg1}, 
$$v_0(p)=h \leq C p_n^3. $$
By following the proof of Lemma \ref{decay001}, we can further deduce the decay estimate 
\begin{equation}\label{lei1}
0\leq \underline{w}(t)\leq Ct^2.
\end{equation}

\noindent{\it Step 4}.\ 
In the above we have shown that $\underline w$ is concave and satisfies the estimate \eqref{lei1}.
We can now derive a contradiction as in dimension two, by showing that $\underline w$ is positive when $t>0$. 
On the one hand, by \eqref{lei1} and the concavity of $\underline{w}(t)$, we have $\underline{w}(t)\le 0$ $\forall\, t\in (0,\delta_0). $
On the other hand, for a fixed $0<t_0<\delta_0$ small, by the strict convexity of $v_0$, we have 
$$ \underline w(y_1, y_2,\cdots, y_{n-1},  t_0)=(1-\frac{n}{2}t_0)\partial_nv_0+v_0\geq \epsilon_0,   $$ 
where the constant $\epsilon_0>0$ is independent of $y_1,\cdots, y_{n-1}$.
Therefore, $\underline{w}(t_0)\geq \epsilon_0>0$,  which is a contradiction.
\end{proof}

\vskip10pt  

\section*{Acknowledgements}

The authors wish to thank the anonymous referee for his/her careful reading of the manuscript and valuable comments.

\vskip10pt

\bibliographystyle{amsplain}

\end{document}